%% file: subgauss_aistat.tex
\documentclass[twoside]{article}

\usepackage{amsmath,amssymb, amsthm}
\usepackage{hyperref}
\usepackage{titlesec}
\usepackage{url}
\usepackage{blindtext}
\usepackage{svg}
\usepackage{wrapfig}
\usepackage{wrapstuff}
\usepackage{enumitem}
\usepackage{pgf, pgfplots}
\usepackage{natbib}
\usepackage{booktabs}
\usepackage{bbm}
\usepackage{eucal}

\theoremstyle{definition}
\newtheorem*{theorem*}{Theorem}
\newtheorem{theorem}{Theorem}

\newtheorem*{lemma*}{Lemma}
\newtheorem{lemma}{Lemma}
\newtheorem*{proposition*}{Proposition}

\newtheorem{proposition}{Proposition}

\newtheorem{corollary}{Corollary}

\newcommand*\eps{\epsilon}

\renewcommand*\epsilon{\varepsilon}
\renewcommand*\phi{\varphi}
\newcommand{\R}{\mathbb{R}}

\newcommand{\N}{\mathbb{N}}
\newcommand{\Q}{\mathbb{Q}}
\newcommand{\E}{\mathbb{E}}
\newcommand{\1}{\mathbbm{1}}
\renewcommand{\P}{\mathbb{P}}

\DeclareMathOperator*{\argmax}{argmax}

\usepackage[preprint]{aistats2027}
\begin{document}



\twocolumn[

\aistatstitle{Estimation of the sub-Gaussian Parameter}

\aistatsauthor{ Jason Liu$^\dag$ \And Min Xu$^\dag$  \And Jinchuan Xing$^\ddag$}

\runningauthor{Liu, Xu, and Xing}

\aistatsaddress{ $^\dagger$ Department of Statistics, Rutgers University  \And $^\ddag$ Department of Genetics, Rutgers University} ]

\begin{abstract}
  The sub-Gaussian parameter (also called the variance proxy) of a mean-zero random variable $X$ is defined as $\xi^2_\star = \sup_{\lambda \in \mathbb{R}} L(\lambda)$ where $L(\lambda) = \frac{2}{\lambda^2} \log \mathbb{E} e^{\lambda X}$ is a weighted cumulant generating function. We study the estimation of $\xi^2_\star$ and prove that the minimax risk is governed by a non-increasing function $\delta_P(C) = \sup_{|\lambda| \geq C} L(\lambda) - \sup_{|\lambda| \leq C} L(\lambda)$ which captures the influence of the tail behavior of the distribution $P$. Over the class of distributions with $\delta_P \leq r$ for a non-increasing function $r$, the minimax risk is, up to a multiplicative constant, lower bounded by $r(\sqrt{\log n}) + n^{-1/2}$ and upper bounded by $r((\log n)^{1/2-\varepsilon}) + n^{-1/2 + \eps}$ for any $\eps > 0$. Our estimator for the upper bound is based on constrained maximization of the empirical analogue of $L$.
  In addition to being almost minimax optimal and adaptive, we further prove that the estimator is asymptotic normal under suitable conditions and that, if the underlying distribution is not sub-Gaussian, the estimator diverges with a rate determined by the heaviness of the distributional tail.
\end{abstract}

\section{Introduction}

A random variable $X$ is called sub-Gaussian if there exists $\xi^2 > 0$ such that $\mathbb{E}[e^{\lambda X}] \leq e^{\lambda^2 \xi^2 / 2}$ for all $\lambda \in \mathbb{R}$; we let $\xi^2_\star \geq 0$ denote the smallest such number and refer to it as the sub-Gaussian parameter of $X$, also called the variance proxy. Sub-Gaussian random variables play a central role in machine learning and statistics because they admit an exponential tail bound $\mathbb{P}(X \geq t) \leq \exp\{ - \frac{t^2}{2 \xi_\star^2}\}$ (the sub-Gaussian condition implies that $\mathbb{E}X = 0$). This is in fact an equivalent characterization: any random variable satisfying such a tail bound is sub-Gaussian. Given independent and identically distributed random variables $X_1, \ldots, X_n$ from a distribution $P$, we study the estimation of the sub-Gaussian parameter $\xi^2_\star$, with the convention that $\xi^2_\star = \infty$ if $P$ is not sub-Gaussian. 

Our motivation comes from using empirical sub-Gaussian tail bounds in large-scale permutation tests. Suppose we are testing $G$ hypotheses and, for each hypothesis $g$, use $M$ random permutations to generate the null test statistic samples $Z^{(g)}_1, \ldots, Z^{(g)}_M$. The empirical $p$-value is then bound below by $\frac{1}{M+1}$, so after a multiple testing correction such as Bonferroni, the power becomes trivial when the number of hypotheses is large compared to $M$; this is often the case when the permuted test statistics are expensive to compute so that $M$ cannot be too large. However, if we know a priori that the null distribution is sub-Gaussian (which holds for example if the test statistic is bounded), then we can instead plug in an estimate of $\xi^2_\star$ into the sub-Gaussian tail bound to obtain $p$-value much smaller than $\frac{1}{M+1}$ and recover power in regimes where the empirical $p$-values are powerless. Directly using $\xi_\star^2$ may result in some loss of statistical efficiency since $\xi_\star^2$ may be influenced by the bulk of the distribution, we can however prevent this by adding a simple bulk-shrinking preprocessing step so that $\xi_\star^2$ and the resulting $p$-value estimate is determined entirely by the tail of the distribution. 

Beyond empirical tail bounds, estimating $\xi^2_\star$ is also useful for testing the sub-Gaussianity of a random variable, which is of interest in cases where the theoretical risk bounds assume sub-Gaussianity or where the estimation algorithms are designed for settings where the data is sub-Gaussian.

Our estimation approach is based on the observation that any $\xi^2 > 0$ that satisfies $\mathbb{E}_P[e^{\lambda X}] \leq e^{\lambda^2 \xi^2 / 2}$ must be larger than $L_P(\lambda) = \frac{2}{\lambda^2} \log \mathbb{E}_P [e^{\lambda X}]$ for any $\lambda \ne 0$; note that we add the subscript $P$ to emphasize the dependence on the underlying distribution $P$. The function $L_P$ is continuous with a removable singularity at 0 so that we can define $L_P(0) = \text{Var}(X)$ by continuity. Therefore, $\xi^2_\star$ can be equivalently defined as 
\begin{equation}
    \xi_\star^2 \equiv \xi_\star^2(P) := \sup_{\lambda \in \mathbb{R}} L_P(\lambda),
    \label{eqn:xi-star-defn}
\end{equation}
where 
\begin{equation}
    L_P(\lambda) := \frac{2}{\lambda^2}\log \E_P[e^{\lambda X}].
\end{equation}
We prove the characterization \eqref{eqn:xi-star-defn} in Proposition~\ref{prop:xi-L-char} of the appendix. 

Given independent and identically distributed observations $X_1, \ldots, X_n$, we define the empirical analogue of $L_P$ as $L_n(\lambda) = \frac{2}{\lambda^2} \log \bigl\{ \frac{1}{n} \sum_{i=1}^n e^{ \lambda (X_i - \bar{X})} \bigr\}$ for $\lambda \neq 0$ and $L_n(0) = \frac{1}{n}\sum_{i=1}^n (X_i - \bar{X})^2$. One idea is to estimate $\xi^2_\star$ by taking $\sup_{\lambda \in \mathbb{R}} L_n(\lambda)$, i.e. the sub-Gaussian parameter of the empirical distribution, but this can incur excessive estimation error due to the unstable behavior of $L_n(\lambda)$ when $\lambda$ is relatively large, e.g. of order $\sqrt{\log n}$. We thus define our estimator as $\hat{\xi}^2_n = \sup_{|\lambda| \leq C_n} L_n(\lambda)$ for a slowly diverging $C_n$ which can be taken as $(\log n)^{\alpha}$ for any $\alpha \in (0, 1/2)$. The truncation introduces a bias whose magnitude is measured by a truncation gap function
\[
\delta_P(C) := \sup_{|\lambda| \geq C} L_P(\lambda) - \sup_{|\lambda| \leq C} L_P(\lambda), \, \text{ for $C \geq 0$.}
\]

The function $\delta_P(C)$ is decreasing and has a non-positive limit as $C \rightarrow \infty$; the bias of $\hat{\xi}^2_n$ vanishes quickly if $\delta_P(C)$ decays at a fast rate. Formally, we have the following bound on the expected error of our estimator:
\begin{align}
\mathbb{E}_P| \hat{\xi}^2_n - \xi^2_\star| \leq \delta_P( (\log n)^{\frac{1}{2} - \varepsilon})_+ + n^{-\frac{1}{2} + \varepsilon},  \, \forall \eps > 0.
\label{eq:emse_xihat}
\end{align}
It turns out that the estimator $\hat{\xi}^2_n$ (and the error bound~\eqref{eq:emse_xihat}) is (almost) optimal in the minimax sense. That is, we prove that 
over the class of all distribution $P$ such that $\delta_P(C) \leq r(C)$ for a positive decreasing function $r(\cdot)$, the worst-case expected error of \emph{any} estimator is at least as large as $r( \sqrt{\log n}) \vee n^{-1/2}$. 

This minimax lower bound shows that $\delta_P(\cdot)$ governs the fundamental difficulty of estimating $\xi_\star^2$. The intuitive reason is that how fast $\delta_P$ decays is a measure of how much the tail of the distribution influences $\xi_\star^2(P)$. If $\delta_P$ decays slowly or not at all, then the value of $\xi_\star^2$ is determined entirely by the far tail and is thus much harder to estimate. 

In the following, we summarize the organization of the rest of the paper as well as our key contributions.

\begin{enumerate}[leftmargin=4mm]
\item \textbf{Section 2:} We prove that the minimax risk of estimating $\xi^2_\star$ is governed by the decay rate of $\delta_P$. Let $r$ be a decreasing positive function.  
The minimax risk over subclasses of distributions satisfying $\delta_P(C) \leq r(C)$ for all $C > 0$ 
is lower bounded by order $r(\sqrt{\log n}) + n^{-1/2}$ and upper bounded by order $r( (\log n)^{1/2 - \varepsilon}) + n^{-1/2 + \varepsilon}$ for any arbitrarily small $\varepsilon > 0$.  If $r(\cdot)$ takes on a negative value, then the minimax risk is of order $n^{-1/2}$. 
\item \textbf{Section 3:} We propose an estimator $\hat \xi_n^2$ for $\xi_\star^2$ to attain 
upper bound on the minimax risk. When the underlying distribution is sub-Gaussian and when $\argmax L_P(\lambda) = \{ \lambda^\star\}$ with $L_P''(\lambda^\star) < 0$, we also prove that $\hat{\xi}_n^2$ is asymptotically normal and derive an asymptotic valid confidence interval for $\xi_\star^2$. 

%We also establish its pointwise consistency when the underlying distribution is sub-Gaussian. 
%If there is $C_0 > 0$ such that $\delta(C_0) \leq 0$, then we show $|\hat{\xi}^2_n - \xi^2_\star| = O_p(n^{-1/2 + \eps})$ for any $\eps > 0$. If $C_0$ satisfies $\delta(C_0) < 0$, then we further show that $|\hat{\xi}^2_n - \xi^2_\star| = O_p(n^{-1/2})$. If in addition $\argmax L(\lambda) = \{\lambda^\star\}$ and $L''(\lambda^\star) < 0$, then we prove that $\hat{\xi}^2_n$ is asymptotically normal and construct confidence interval for $\xi^2_\star$. 

On the other hand, when the underlying distribution is not sub-Gaussian, we prove that $\hat \xi_n^2$ diverges almost surely, where the rate of divergence is faster if the departure from sub-Gaussianity is more significant, i.e. if the distribution is more heavy-tailed.  
\item \textbf{Section 4:} We apply our sub-Gaussian estimator to large-scale permutation testing, and propose a bulk-shrinking method to ensure that the sub-Gaussian parameter captures the tails of the distribution. We then compare our method with the peaks-over-threshold method in a Gene Ontology (GO) enrichment study, showing that it can serve as a reliable alternative, particularly in regimes where the peaks-over-threshold method is of uncertain validity.
\end{enumerate}

\noindent \textbf{Related Work}: The most relevant work is \cite{mies2026empiricalorlicznorms}, which considers the estimation of the Orlicz norm 
\[
\sigma_\psi \equiv \|X\|_{\psi} := \inf \{\sigma > 0: \E \psi(|X|/\sigma) < 1\}
\]
where $\psi: [0, \infty) \to [0, \infty)$ is any increasing convex function with $\psi(0) = 0$ and $\lim_{x \to \infty} \psi(x) = \infty$. 
They study the estimator $\hat \sigma_\psi := \inf \{\sigma > 0: \frac1n\sum_{i=1}^n\psi(|X_i|/\sigma) < 1\}$, i.e. the Orlicz norm of the empirical distribution, and prove a central limit theorem for it under regularity conditions. However, the sub-Gaussian parameter and the sub-Gaussian norm (corresponding to $\psi_2(x) = \exp(|x|^2)- 1$) are \textit{not} equivalent: \cite{leskel2026sharp} proved that $0.612\approx \sqrt{3/8} \leq \xi_*/\|X\|_{\psi_2} \leq \sqrt{\log 2} \approx 0.832$ and showed the bounds to be sharp. Also, while $\xi_\star^2$ does not generally have a closed form, \cite{atouani2025optimalsubgaussianvarianceproxy} studied the function $g_Y(\lambda; \sigma^2) = \frac12\lambda^2\sigma^2 - \log \E[e^{\lambda X}]$ to give explicit expressions of $\xi_\star^2$ for three-point distributions. Other studies on the empirical moment generating function include \cite{Stewart2003-cd}, which focus on the supremum over $\lambda$ of the normalized process $n^{-1/2}\sum_{i=1}^n e^{\lambda X_i - \psi(2\lambda)/2} - \mathbb E[e^{\lambda X - \psi(2\lambda)/2}]$ where $\psi(\lambda) = \log \E[e^{\lambda X}]$.

\noindent \textbf{Notation}: We denote $[n] = \{1,2,\dots,n\}$. We do not take $\N$ to include $0$, and will use the notation $\N_0 := \N \cup \{0\}$. We write $a \gtrsim b$ if there is a universal constant $C > 0$ such that $a \geq C b$ and $a \asymp b$ if $a \gtrsim b$ and $b \gtrsim a$.

\section{Minimax Risk of Estimating $\xi_\star^2$}
For a class of distributions $\mathcal{P}$, we denote by $\mathcal{M}_n(\mathcal{P})$ the minimax optimal risk
\[
\mathcal{M}_n(\mathcal{P}) := \inf_{\hat{\xi}^2} \sup_{P \in \mathcal{P}} \mathbb{E}_P | \hat{\xi}^2(X_1, \ldots, X_n) - \xi_\star^2(P) |
\]
where the infimum is over all estimators (measurable real-valued functions of $n$ variables). The minimax risk describes the fundamental difficulty of estimating $\xi_\star^2$ in that no estimator can have a risk uniform over $\mathcal{P}$ that is lower than the minimax risk. If an estimator $\hat{\xi}^2$ has risk equal to $\mathcal{M}_n(\mathcal{P})$ up to a multiplicative constant that does not depend on $n$, then we refer to the estimator as \emph{minimax optimal}. If, in addition, the estimator is agnostic to the class $\mathcal{P}$, then we refer to it as \emph{adaptive}. 

Fix $\Xi > 0$. We define the class
\[
 \mathcal{P}(\Xi) := \{ P \,:\, P \text{ sub-Gaussian and } \, \xi^2_\star(P) \leq \Xi \}.
\]
As we will see, the minimax risk over all sub-Gaussian distributions $\mathcal{M}_n(\mathcal{P}(\Xi))$ is bounded away from zero so that no estimator can be uniformly consistent. To obtain meaningful minimax risk therefore, we have to restrict to a suitable subclasses. To that end, let us first define the truncation gap function
\begin{equation}
    \delta_P(C) := \sup_{|\lambda| \geq C} L_P(\lambda) - \sup_{|\lambda| \leq C} L_P(\lambda).
\end{equation}
Note that $\delta_{P}$ is non-increasing and $\lim_{C \to \infty} \delta_P(C) \leq 0$. Since truncating the optimization in \eqref{eqn:xi-star-defn} at $[-C, C]$ results in an error of $\delta_P(C) \vee 0$, faster rates of decay on $\delta_P \vee 0$ correspond to a better approximation of $\xi_\star^2$ by a truncated maximization. We give some examples of $\delta_P$ in Section~\ref{sec:estimator-convergence} and Figure~\ref{fig:placeholder}

To study the minimax risk, we impose an envelope over the truncation gap. Define for any $\Xi, C_0 >0$ and non-increasing $r: [0, \infty) \to \R$ the class 
\[
\mathcal P(\Xi, r, C_0) := \{ P \in \mathcal{P}(\Xi) \,:\, \delta_P(C) \leq r(C),\, \forall C > C_0 \}.
\]
We are ready to state the main result of this section. 
\begin{theorem}
\label{thm:minimax-lower-bound}
Fix $\Xi > 0$ and $C_0 > 0$. Let $r: [0, \infty) \to \R$ be a non-increasing function with $\lim_{t \to \infty} |r(t)| < \Xi/3$. For arbitrarily small $\eps > 0$, there exists constants $n_0, c_1, C_1 > 0$ depending only on $r, \Xi, C_0, \eps$ such that
%Let $C_1 := (1 - \sqrt{2(1-e^{-1/2})})/16$ and $C_2 := (1-6^{-1/2})/8$ be universal constants. There exists a constant $n_0 > 0$ depending only on $r$ and $\Xi$ such that whenever $n > n_0$, we have the following:
\begin{multline*}
r\bigl( \Xi^{-\frac{1}{2}} (\log n)^{\frac{1}{2} - \varepsilon}\bigr)_+ + C_1 \Xi n^{-\frac{1}{2} + \varepsilon} \\
            \gtrsim \mathcal M_n(\mathcal{P}(\Xi, r, C_0))\\ \gtrsim r\bigl(8 \Xi^{-\frac{1}{2}} \sqrt{\log n}\bigr)_+ + c_1 \Xi_1 n^{-\frac{1}{2}} 
\end{multline*}
for all $n > n_0$, where $\Xi_1 = (\Xi - 576 C_0^{-2})_+$. 
\end{theorem}

We defer the proof of Theorem~\ref{thm:minimax-lower-bound} to Section~\ref{sec:app-minimax-lower} of the appendix and give a sketch of the lower bound proof in the end of this section and a detailed description of the estimator that attains the upper bound in the next section. We note that the upper bound in Theorem~\ref{thm:minimax-lower-bound} is slightly loose, which we believe is an artifact of our proof technique. We conjecture that minimax risk actually holds with $\varepsilon = 0$. 

We give a few remarks to help interpret Theorem~\ref{thm:minimax-lower-bound}.
\begin{enumerate}
\item One immediate consequence of Theorem~\ref{thm:minimax-lower-bound} is that if $r(t) = 1$ (or any other positive constant), then the minimax risk does not tend to 0 as $n \rightarrow \infty$ so that no estimator can be uniformly consistent over $\mathcal{P}(\Xi, r, C_0)$. Therefore, conditions on $\delta_P$ is \emph{necessary} to be able to estimate $\xi_\star^2$ well. 

\item If $r(t) = \frac{2}{t}$, then $\mathcal{P}(\Xi, r, C_0)$ contains all distributions supported on $[-1, 1]$, since, if $|X| \leq 1$, 
\[
    \delta_P(C) \leq \sup_{|\lambda| \geq C} L(\lambda) \leq \sup_{|\lambda| \geq C} \frac{2}{\lambda^2}\log \E[e^{|\lambda|}] = \frac{2}{C}.  
\]
The minimax risk of estimating $\xi_\star^2$ over this class is of order $\sqrt{\frac{\Xi}{\log n}}$, which tends to 0 very slowly. In fact, in the proof of Theorem~\ref{thm:minimax-lower-bound}, we see that the minimax risk over all distributions in $\mathcal{P}(\Xi, r, C_0)$ that are supported on $[-1, 1]$ is still of order $\sqrt{\frac{\Xi}{\log n}}$ so that having a compact support does not make the estimation problem any easier. 

\item If $r(t_0) \leq 0$ for some $t_0 \geq 0$, then we see that the minimax risk over $\mathcal{P}(\Xi, r, C_0)$ is nearly root-n, with an order between $n^{-1/2+\varepsilon}$ and $n^{-1/2}$ for any arbitrarily small $\varepsilon$. If $r(t) = 0$ for all $t > t_0$, then it contains the Gaussian distribution, where $\delta_P(C) = 0$ for all $C \geq 0$.

\end{enumerate}

The lower bound in Theorem~\ref{thm:minimax-lower-bound} is proved via an information-theoretic argument using Le Cam's two point method. Specifically, we reduce the estimation problem to that of testing between the Rademacher distribution and a scale-mixture of two Rademacher distributions. 

In the next section, we describe the estimator that attains the upper bound in Theorem~\ref{thm:minimax-lower-bound}. This estimator, in addition to being almost minimax optimal, is also adaptive to the unknown class $\mathcal{P}(\Xi, r, C_0)$. In other words, the estimator does not require knowledge of $\Xi, r, C_0$. 

    % \begin{enumerate}
    %     \item[(a)] If $r(t) > 0$ for all $t \geq 0$, then, for any arbitrarily small $\varepsilon > 0$,
    %     \begin{multline*}
    %         r( \Xi^{-1/2} (4 \log n)^{\frac{1}{2} - \varepsilon}) + {\color{red} C?} n^{-1/2 + \varepsilon} \\
    %         \geq \mathcal M_n(\mathcal{P}(\Xi, r, C_0))\\ \geq C_1 r\bigl(4 \Xi^{-1/2} \sqrt{\log n}\bigr) + {\color{red} C?} n^{-1/2}.
    %     \end{multline*}
    % \item[(b)] Suppose that $r(C_0) < 0$ for some $C_0 > 0$. Then 
    % \[
    % \mathcal M_n(\Xi, r, C_0) \geq C_2\frac{\Xi'}{\sqrt n},
    % \]
    % where $\Xi' = (\Xi' - 576C_0^{-2})_+$.
    % \end{enumerate}

% To interpret Theorem~\ref{thm:minimax-lower-bound}, we give a few examples. 
% \begin{enumerate}
%     \item When $r_0(t) \equiv 1/2$, no rate of decay on $\delta_P$ is imposed, so Theorem~\ref{thm:minimax-lower-bound}(a) shows that no estimator can be uniformly consistent. 
%     \item If $r_\nu(t) = 2t^{-\nu}$ for $\nu \in (0, 1]$, then $\mathcal M_n(\Xi, r_\nu, 2) \gtrsim(\log n)^{-\nu/2}$. The case $\nu = 1$ corresponds to $P$ supported on $[-1, 1]$, because in this case 
%     \[
%     \delta_P(C) \leq \sup_{|\lambda| \geq C} L(\lambda) \leq \sup_{|\lambda| \geq C} \frac{2}{\lambda^2}\log \E[e^{|\lambda|}] = \frac{2}{C}.  
%     \]
%     \item Let $r_-(t) = -\delta_0$. If $P$ is supported on $[-1, 1]$ and furthermore $\sigma^2(P) \geq \frac32 \delta_0$, then $P \in \mathcal P(\Xi, r_-, 8\delta_0^{-1})$ (Proposition~\ref{prop:unif-supp-delta-inclusion}). The proof of Theorem~\ref{thm:minimax-lower-bound} reveals that $\mathcal M(\Xi, r_-, 8\delta_0^{-1}) \gtrsim n^{-1/2}$.  
% \end{enumerate}
\section{A (Nearly) Minimax Optimal Estimator for $\xi_\star^2$}

Based on Proposition~\ref{prop:xi-L-char}, an immediate idea for estimating $\xi_\star^2$ is to replace $L$ by its empirical analogue 
\[
L_n(\lambda) = \begin{cases}
    \frac{2}{\lambda^2} \log \frac1n \sum_{i=1}^n e^{\lambda(X_i - \bar X)} & \lambda \ne 0 \\
    \frac1n \sum_{i=1}^n (X_i - \bar X)^2 & \lambda = 0.
\end{cases}
\]
Centering the empirical distribution by subtracting the sample mean $\bar{X}$ is critical even if $\mathbb{E}X_1 = 0$, otherwise $|L_n(\lambda)| \rightarrow \infty$ as $|\lambda| \rightarrow 0$. We may then consider estimating $\xi_\star^2$ via $\sup_{\lambda \in \mathbb{R}} L_n(\lambda)$.  This however behaves poorly due to high variance of $L_n(\lambda)$ for large $\lambda$. In this work, we instead study the truncated maximization of $L_n$. We let $C_n > 0$ be a slowly diverging sequence and define our estimator
\[
\hat{\xi}_n^2 := \sup_{|\lambda| \leq C_n} L_n(\lambda). 
\]
We will see in our theoretical analysis that $C_n$ can be taken as $(\log n)^\alpha$ for any $\alpha \in (0, 1/2)$.

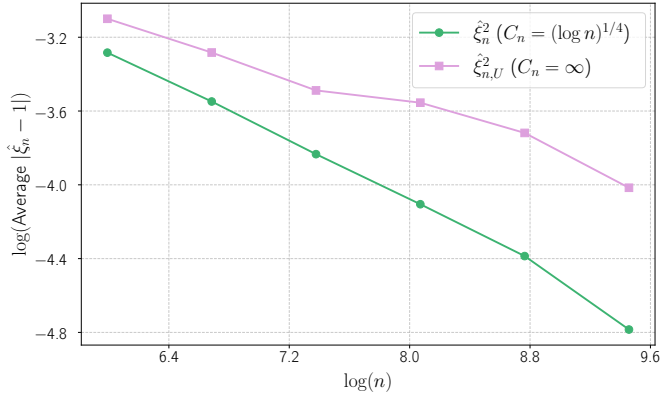
\begin{figure}
    \centering
    \resizebox{\linewidth}{!}{\input{figures/log-log_plot.pgf}}
    \caption{\textbf{Log-log plots for the mean absolute deviation for truncated $\hat \xi^2_n$ and untruncated $\hat \xi^2_{n, U}$.} }
   \label{fig:trunc-untrunc-comparison}
\end{figure}

To illustrate the importance of truncation, we conduct the following numerical experiment: we draw $X_1, \dots, X_n \sim N(0, 1)$ and compare the untruncated estimator
$\hat \xi_{n, U}^{2} := \sup_{\lambda \in \R} L_n(\lambda)$ with $\hat \xi_n^2$ when $C_n = (\log n)^{1/4}$. In Figure~\ref{fig:trunc-untrunc-comparison}, we display the log-mean absolute deviation for $\hat \xi_n^2$ and $\hat \xi^2_{n, U}$ at different values of $n$ the averaged over 600 simulations; we take $n \in \{2^{j} \times 10^{2}: 2 \leq j \leq 7\}$. We observe that the error curve for $\hat \xi_n^2$ has a slope of approximately $-1/2$ in contrast to a slope of approximately $-1/4$ for $\hat \xi_{n, U}^2$. This is because $L_n(\lambda)$ exhibits unstable behavior when $\lambda$ is of order $\sqrt{\log n}$, slowing the convergence. Truncating removes this issue. 

\subsection{Convergence of the empirical moment generating function}

In order to study the consistency of $\hat \xi^2_n$ and quantify the rate of convergence, we must establish results on the uniform convergence of $L_n$ to $L$. To do this, we first study the empirical moment generating function. We shall denote the moment generating function and its centered empirical analogue by  
\[
M_P(\lambda) := \E_P[e^{\lambda X}], \quad M_n(\lambda) = \frac1n\sum_{i=1}^n e^{\lambda (X_i - \bar X)}.
\]
We shall also write $\psi_P(\lambda) := \log M_P(\lambda)$ and $\psi_n(\lambda) = \log M_n(\lambda)$. When $P$ is fixed, we will abbreviate $M_P(\lambda) \equiv M(\lambda)$ and $\psi_P(\lambda) \equiv \psi(\lambda)$.

Previous studies have established uniform control of the empirical moment generating functions over a fixed compact set \cite[Theorem 2.3]{feuerverger1989empirical}, but our truncation level $C_n$ is diverging (albeit at a slow rate). We therefore need control over the convergence of $M_n$ and its derivatives beyond a fixed compact set. We first establish this for the empirical moment generating function. 
\begin{proposition}
\label{prop:sub-gaussian-convergence-lemma}
    Let $k \in \N_0$.
    Fix $\alpha \in (0, 1/2)$, and set $C_n := (\log n)^\alpha$. Let $P$ be a sub-Gaussian distribution. There exists $n_0(\alpha, k, \Xi)$ such that for all $n \geq n_0$,
    \[
    \sup_{\lambda \in [-C_n, C_n]} |M_n^{(k)}(\lambda) - M^{(k)}_P(\lambda)| \leq 7d_nn^{-1/2}
    \]
    with probability at least $1 - 3/n$, and 
    $d_n \leq (2\log n)^{k+3}\exp(2(\log n)^{\alpha + \beta}) \leq n^{o(1)}$ for any $\beta \in (1/2, 1 - \alpha)$. 
\end{proposition}
Having established a result on the convergence of the empirical moment generating function, we can easily extend the result to derivatives of the empirical cumulant generating function. According to the next proposition, this directly controls $L_n$. 
\begin{proposition}
\label{prop:L_n-cgf-bound}
    For each $n$, $L_n$ is real-analytic and hence infinitely differentiable. If $C_n > 0$, then 
    \begin{multline}
        \sup_{\lambda \in [-C_n, C_n]} |L_n^{(k)}(\lambda) - L^{(k)}_P(\lambda)| \\ \leq \sup_{\lambda \in [-C_n, C_n]} |\psi^{(k+2)}_n(\lambda) - \psi^{(k+2)}_P(\lambda)|.
        \label{eqn:L-psi2-bound}
    \end{multline}
    Consequently, if $C_n = (\log n)^\alpha$ for $\alpha \in (0, 1/2)$, then there are constants $n_1(\alpha, k, \Xi)$ and $K(k, \Xi)$ such that for all $ n \geq n_1$, 
    \begin{multline}
            \sup_{\lambda \in [-C_n, C_n]} |L_n^{(k)}(\lambda) - L^{(k)}_P(\lambda)| \\ \leq K(k+2, \Xi)d_n^{1 + k\Xi} n^{-1/2},
        \end{multline}
        with probability at least $1 - 3/n$.
\end{proposition}

%Actually, the convergence in Proposition~\ref{prop:L_n-cgf-bound}(a) is also uniform over $\mathcal P(\Xi)$. The proof requires some machinery and is therefore delayed to Section~\ref{sec:minimax-opt}.
\subsection{Convergence of $\hat \xi_n^2$}
\label{sec:estimator-convergence}
Equipped with the above propositions, we are ready to deduce consistency of $\hat \xi_n^2$. 
\begin{theorem}
\label{prop:second-consistency}
    Let $P$ be a sub-Gaussian distribution and let $X_1, \ldots, X_n$ be i.i.d samples from $P$. Then 
    \begin{equation}
        |\hat \xi^2_n - \xi_\star^2| \leq \sup_{|\lambda| \leq C_n} |L_n(\lambda) - L_P(\lambda)| + \delta_P(C_n) \vee 0.
        \label{eqn:xi-hat-n-est}
    \end{equation}
    Therefore, if $C_n = (\log n)^\alpha$ for any $\alpha \in (0, 1/2)$, then 
    \vspace{-0.1in}
    \begin{enumerate}
        \item[(a)] $\hat \xi_n^2 \xrightarrow{p} \xi_\star^2$ as $n \to \infty$.
        \item[(b)] If $\delta_P(C_0) \leq 0$ for some $C_0 > 0$, then there is $n_0(\alpha, \Xi, C_0) > 0$ such that $|\hat \xi_n^2 - \xi_\star^2| \lesssim d_n^{\Xi}n^{-1/2}$ with probability $\geq 1-2/n$ for all $n > n_0$, where $d_n$ is defined in Proposition~\ref{prop:sub-gaussian-convergence-lemma} and satisfies $d_n \leq n^{o(1)}$. 
        \item[(c)] If $\delta_P(C_0) < 0$, then $|\hat \xi_n^2 - \xi_\star^2| = O_p(n^{-1/2})$. 
    \end{enumerate}
\end{theorem}
The proof of Theorem~\ref{prop:second-consistency} reveals that, so long as $\delta(C_0) = 0$ for some $C_0$, an exact rate of $n^{-1/2}$ can be achieved by choosing $C_n \equiv C_0$. For example, if $X \sim N(0, 1)$, then any such $C_0$ would work. However, we do not know $C_0$ a priori, so we cannot make this choice.

%\begin{enumerate}
    %\item[(a)] If $P=N(0, \sigma^2)$, then $L(\lambda; P) \equiv \sigma^2$ for all $\lambda \in \R$, so $\delta(C) \equiv 0$. 
    %\item[(b)] If $P=\frac12 N(0, 1) + \frac12 N(0, 2)$, then $L(\lambda;P)$ admits no finite maximizers, so $\delta_P(C) > 0$ for all $C > 0$. 
    %\item[(c)] If $P$ is supported on $[-1, 1]$, then as $|\lambda| \to \infty$, \begin{equation} L(\lambda; P) \leq \frac{2}{\lambda^2}\log \E_P[e^{|\lambda|}] = \frac{2}{|\lambda|} \to 0. \label{eqn:bounded-rv-example} \end{equation} Consequently, $\lim_{C \to \infty} \delta_P(C) = -\xi_\star^2(P)$. 
%\end{enumerate}
\textbf{Example 1.}
We now give some cases that illustrate the different regimes described in Theorem~\ref{prop:second-consistency}. 
\begin{enumerate}
    \item[(a)] Suppose $X$ is bounded almost surely. Then $L$ obeys the inequality 
    \[
    L(\lambda) \leq 2\lambda^{-2}\log \E[e^{\|X\|_\infty \lambda}] = 2\|X\|_\infty |\lambda|^{-1},\]
    
        which vanishes as $\lambda \to \infty$. Thus $\lim_{C \to \infty}\delta_P(C) = -\xi_\star^2$, so Theorem~\ref{prop:second-consistency}(c) applies. If in particular $X \sim \text{Unif}[-1 ,1]$, then $L$ satisfies the hypotheses of Proposition~\ref{prop:ptwise-asymptotic-normality} so $\sqrt n(\hat \xi_n^2 - \xi_\star^2) \xrightarrow{d}  N(0, 4/45)$.
        \item[(b)] If $X \sim N(0, 1)$, then $L(\lambda) \equiv 1$ for all $\lambda \in \R$, so $\delta_P(C) \equiv 0$. Therefore (c) does not apply, but (b) does.
\item[(c)] If $X \sim \frac12N(0, 1) + \frac12N(0, 2)$, then $\delta_P(C) > 0$ for all $C > 0$ so that  Theorem~\ref{prop:second-consistency}(a) holds. In fact, one can show that $\delta_P(C) \asymp C^{-2}$, so that, by upper bound~\eqref{eqn:xi-hat-n-est} and Proposition~\ref{prop:L_n-cgf-bound} that $| \hat{\xi}_n^2 - \xi_\star^2| \leq O_p( (\log n)^{-2\alpha})$. 
\end{enumerate}
\begin{figure}[h]
    \centering
\includegraphics[width=0.95\linewidth]{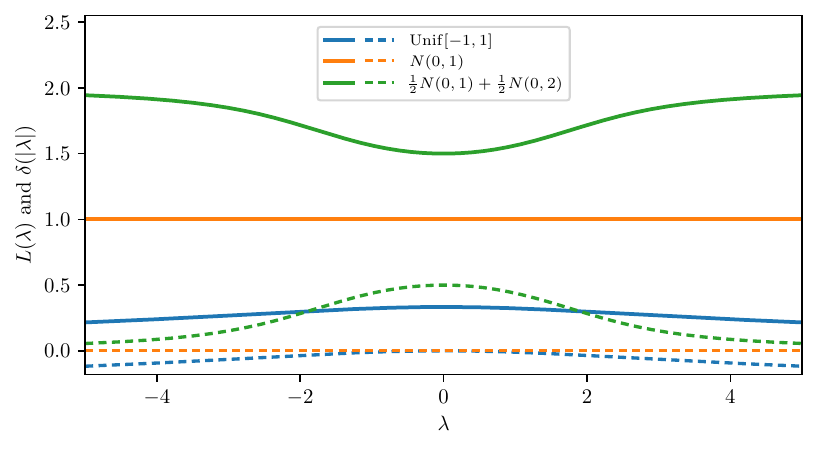}
\vspace{-0.1in}
    \caption{\textbf{Plots of $L_P(\lambda)$ (solid) and $\delta_P(|\lambda|)$ (dashed).} In Example 1(a), $\argmax L$ exists so $\delta$ is eventually negative. In (c), $L$ is increasing away from 0 so that $\delta(C) > 0$ for all $C > 0$. }
    \label{fig:placeholder}
\end{figure}
\vspace{-0.1in}
\subsection{Minimax upper bounds via $\hat \xi_n^2$}
In this section, we establish a slightly stronger minimax upper bound that implies the one in Theorem~\ref{thm:minimax-lower-bound}.
\begin{theorem}
\label{thm:minimax-upper-bound}
    Set $C_n = (\log n)^\alpha$ for any $\alpha \in (0, 1/2)$. Let $r: [0, \infty) \to \R$ be a non-increasing function. 
    \vspace{-0.1in}
    \begin{enumerate}
        \item[(a)] Let $C_0 = 0$. There are constants $n_1, C_1 > 0$ depending on $\alpha, \Xi, \eps$ such that, for any $n > n_1$ and $\eps > 0$, 
        \[
        \sup_{P \in \mathcal{P}(\Xi, r, C_0)} \E_P|\hat \xi_n^2 - \xi_\star^2| \leq r(C_n)_+ + C_1n^{-\frac12 + \eps}.
        \]
        \item[(b)] If $r(C_0) < 0$ for some $C_0 > 0$, there are constants $n_2, C_2 > 0$ depending on $\alpha, \Xi, r, C_0$
        such that, for all $n > n_2$,
        \[
        \sup_{P \in \mathcal{P}(\Xi, r, C_0)}\E_P|\hat \xi_n^2 - \xi_\star^2| \leq C_2n^{-1/2}.
        \]
        
    \end{enumerate}
\end{theorem}
%Our upper bound is almost sharp, but does not identify the minimax risk. For example, if $r(t) = 2t^{-1}$, then our upper bound does not preclude a rate of $\frac{\log \log n}{\sqrt{\log n}}$. 

The upper bound~\eqref{eqn:xi-hat-n-est} and Proposition~\ref{prop:L_n-cgf-bound} readily implies Theorem~\ref{thm:minimax-upper-bound}(a), which in turn implies the upper bound of Theorem~\ref{thm:minimax-lower-bound} by putting $\alpha = \frac12 - \eps$.  The rest of this section details a strategy to prove Theorem~\ref{thm:minimax-upper-bound}(b). Assume now that $r(C_0) < 0$ for some $C_0 > 0$. First, we show that the maximizers of $L_n$ and $L$ lie inside a compact interval $[-C_0, C_0]$ with probability tending to 1, uniformly over $\mathcal P(\Xi, r, C_0)$.
\begin{proposition}
\label{prop:maximizer-uniform-bound}
    Set $C_n = (\log n)^\alpha$ for any $\alpha \in (0, 1/2)$. Let $\Lambda^\star_n$ denote the maximizers of $L_n$ in $[-C_n, C_n]$ and similarly for $\Lambda^\star(P)$. There exists $n_0(\Xi, \alpha, r, C_0) > 0$ such that for any $n > n_0$,
\[
   \sup_{P \in \mathcal P(\Xi, r, C_0)} \P\left(\Lambda^\star_n \cup \Lambda^\star(P) \not\subset [-C_0, C_0]\right) \leq 3/n.
    \]
\end{proposition}
Proposition~\ref{prop:maximizer-uniform-bound} says that, with high probability, we can identify $L_n$ with its restriction to $[-C_0, C_0]$, independent of $P \in \mathcal P(\Xi, r, C_0)$, and similarly for $L$. 

This allows us control over the empirical moment generating function via a 
central limit theorem uniform in $\lambda \in [-C_0, C_0]$ and $P \in \mathcal P(\Xi)$. Thus we have the bound
\[
    \sup_{P \in \mathcal P(\Xi)} \E_P \left[\sup_{|\lambda| \leq C_0} |M_n^{(k)}(\lambda) - M^{(k)}_P(\lambda)|\right] = O(n^{-1/2}).
\]
Similarly to before, this can be used to uniformly control derivatives of the empirical cumulant generating function, which then controls $L_n$. 
We discuss these details precisely in Section~\ref{sec:uniform-donsker}. 
\begin{proposition}
\label{prop:L_n-uniform-rootn}
% {\color{orange}
% Fix $C_0 > 0$ and $k \in \N_0$. The empirical process $\lambda \mapsto \sqrt{n} \{L_n^{(k)}(\lambda) - L^{(k)}(\lambda; P)\}$ converges to a mean-zero Gaussian process on $[-C_0, C_0]$, uniformly for all $P \in \mathcal{P}(\Xi)$. As a direct consequence, as $n \to \infty$, uniformly over $P \in \mathcal P(\Xi)$.
%     \[
%      \sup_{|\lambda| \leq C_0} |L_n^{(k)}(\lambda) - L^{(k)}(\lambda; P)| = O_p(n^{-1/2}),
%     \]
% }
Fix $k \in \N_0$. There exists a constant $C = C(\alpha, k, \Xi, C_0)$ such that 
    \[
    \sup_{P \in \mathcal P(\Xi)}\E_P\left[\sup_{|\lambda| \leq C_0} |L_n^{(k)}(\lambda) - L^{(k)}_P(\lambda)|\right] \leq \frac{C}{\sqrt n}. \]
\end{proposition}
This can be combined with \eqref{eqn:xi-hat-n-est} to prove Theorem~\ref{thm:minimax-upper-bound}(b). Because $C_n$ only depends on $\alpha$, we have also shown that $\hat \xi^2_n$ is adaptive in the sense that it does not require knowledge of $\Xi$, $C_0$, or $\delta_0$.

\subsection{Asymptotic Normality}
\label{sec:inference}
Under even stronger assumptions, we can also derive the limiting distribution of $\hat \xi_n^2$. 
\begin{proposition}
\label{prop:ptwise-asymptotic-normality}
For $\lambda \ne 0$, define  
    \[
    V(\lambda) := \frac{4}{\lambda^4}\left(\frac{M(2\lambda)}{M(\lambda)^2} - 2\lambda\psi'(\lambda)+ \lambda^2\sigma^2 - 1\right)
    \]
    and set $V(0) := \E X^4 - \sigma^4$. Suppose that $L$ is uniquely maximized at $\lambda^\star$, $L''(\lambda^\star) < 0$, $V(\lambda^\star) > 0$, and there is $C_0$ such that $\delta(C_0) < 0$. 
    Then, for $C_n= (\log n)^\alpha$, $\alpha \in (0, 1/2)$, we have 
    \[
    \sqrt n(\hat \xi_n^2 - \xi_\star^2) \xrightarrow{d} N(0, V(\lambda^\star)).
    \]
\end{proposition}
The hypothesis that $L$ be uniquely maximized is necessary: if $X \sim N(0, 1)$, then $L(\lambda) \equiv 1$ so that all of $\R$ maximizes $L$. In this case, the maximizers of $L_n$ will be entirely determined by stochastic fluctuations in $L_n$, so the limiting distribution is not Gaussian. We include the $N(0, 1)$ case in Figure~\ref{fig:asymptotic-normality} to explicitly illustrate this behavior. 

According to Proposition~\ref{prop:ptwise-asymptotic-normality}, we may construct asymptotically valid confidence intervals by plugging in empirical estimates. For $\lambda \ne 0$, let
\[
\hat V_n(\lambda) := \frac{4}{\lambda^4}\left(\frac{M_n(2\lambda)}{M_n(\lambda)^2} - 2\lambda\psi_n'(\lambda) + \lambda^2\hat\sigma^2_n - 1\right),
\]
where $\hat \sigma_n^2$ is the sample variance, and $\hat V_n(0)$ is defined by continuity. If $z_{\alpha/2}$ denotes the $(1 - \alpha/2)$-quantile of the standard normal distribution and $\lambda^\star_n$ denotes a maximizer of $L_n$ over $[-C_n, C_n]$, an asymptotic $(1 - \alpha)$-level confidence interval for $\xi_\star^2$ is
\[
\left[\hat \xi_n^2 -  z_{\alpha/2}\sqrt{\frac{\hat V_n(\lambda^\star_n)}{n}}, \hat \xi^2_n + z_{\alpha/2}\sqrt{\frac{\hat V_n(\lambda^\star_n)}{n}}\, \right].
\]
To verify Proposition~\ref{prop:ptwise-asymptotic-normality}, we generate samples $X_1, \dots, X_n \sim P$ for the following choices of $P$: $\mathrm{Unif} [-1, 1]$, $\frac{1}{3}\delta_{-2} + \frac{2}{3}\delta_1$, $\frac12 \mathrm{Unif}[-3, 3] + \frac12N(0, 1)$, and $N(0, 1)$.
We set $C_n = (\log n)^{1/4}$, $n = 1000$, and run 600 simulations. For each simulation, we compute the estimated asymptotic variance $\hat V_n(\lambda^\star_n)$ and take the average over all simulations and compare it to the sample variance.

\begin{figure}
\centering
    \resizebox{0.95\linewidth}{!}{\input{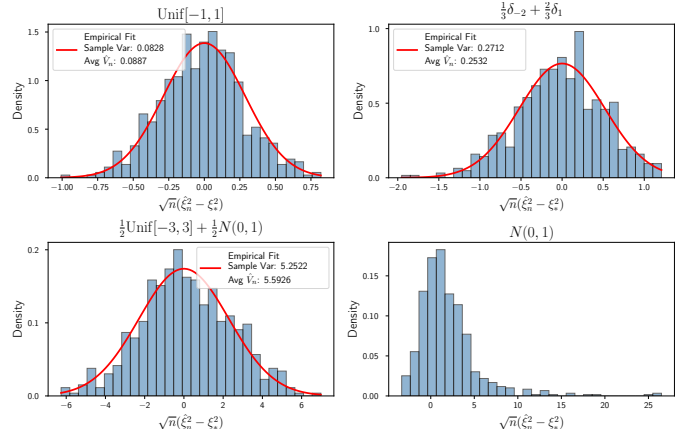}}
\caption{\textbf{Distribution of $\sqrt n(\hat \xi_n^2 - \xi_\star^2)$ for four choices of $P$.} Proposition~\ref{prop:ptwise-asymptotic-normality} gives asymptotic normality for all cases but $N(0, 1)$.}
    \label{fig:asymptotic-normality}
\end{figure}

In practice, the hypotheses of Theorem~\ref{prop:second-consistency} and Proposition~\ref{prop:ptwise-asymptotic-normality} are difficult to check. One heuristic approach is to plot $L_n$ on $[-C_n, C_n]$, and check if the maximizers are in the interior. An affirmative answer would be consistent with the hypotheses in Theorem~\ref{prop:second-consistency}(c) and Proposition~\ref{prop:ptwise-asymptotic-normality}; on the other hand if $L_n$ is flat or monotone then the guarantees likely do not apply and the confidence interval may not be reliable.

\subsection{Rate of divergence under violation of sub-Gaussianity}
\label{sec:divergence}

We now study the mis-specified case when the underlying distribution is not sub-Gaussian. In this setting, it is desirable for $\hat \xi^2_n$ to exhibit clear signs of divergence, which 
would serve as qualitative evidence against the sub-Gaussian assumption. As we next 
show, $\hat\xi^2_n$ diverges almost surely whenever sub-Gaussianity fails.

\begin{proposition}
\label{prop:divergence}
Let $X$ be a non-sub-Gaussian random variable, i.e. either $\E |X| = \infty$, $\E e^{\lambda (X- \E X)} = \infty$ for some $\lambda$, or if for any $\xi^2 > 0$, there exists $\lambda \in \R$ with $\E e^{\lambda (X- \E X)} > e^{\lambda^2\xi^2/2}$.  Let $X_1, \ldots, X_n$ be i.i.d. copies of $X$. If $C_n \rightarrow \infty$, then $\hat{\xi}^2_n \xrightarrow{\mathrm{a.s.}} \infty$. 
\end{proposition}

Although Proposition~\ref{prop:divergence} guarantees that $\hat{\xi}_n^2$ diverges, we expect faster divergence for heavier tails, reflecting a stronger departure from sub-Gaussianity. The following quantifies this through a deterministic lower bound on $\hat{\xi}_n^2$ in terms of the data range.

\begin{proposition}
\label{prop:divergence2}
Let $X_1, \ldots, X_n$ be any collection of real numbers and let $\Delta_n := \max_{i = 1, \ldots, n} | X_i - \bar{X}|$. If $\Delta_n \geq \frac{ 2\log n}{C_n}$, then $\hat{\xi}^2_n \geq \frac{\Delta_n^2}{2 \log n}$. 
\end{proposition}
If we set $C_n = (\log n)^\alpha$ for an $\alpha \in (0, 1/2)$, then Proposition~\ref{prop:divergence2} applies when $\Delta_n \geq 2(\log n)^{1- \alpha}$. Since $\Delta_n = O_p(\sqrt{ \log n})$ for sub-Gaussian data (\cite{Wainwright_2019}, Exercise 2.12), we expect $\Delta_n$ to be of larger order when $X_1, \ldots, X_n$ are not sub-Gaussian. We consider some examples:

\begin{enumerate}[leftmargin=3mm]
\vspace{-0.1in}
\item If $X_1, \ldots, X_n$ are distributed according to the mean-zero Laplace distribution, then $\Delta_n$ is of order $\Theta_p(\log n)$. Thus Proposition~\ref{prop:divergence2} applies for $n$ sufficiently large, so  $\hat{\xi}^2_n$ diverges at rate $\Omega_p(\log n)$. 
\item If $X_1, \ldots, X_n$ have the heavy tail $t$-distribution with degree $\nu > 0$, then $|X_{(n)} - X_{(1)}|$ is of order $\Theta_p( n^{1/\nu} )$ so that $\hat{\xi}^2_n$ diverges at rate $\Omega_p( \frac{n^{2/\nu} }{\log n})$. 
\end{enumerate}

\section{Application}
\vspace{-0.1in}
In this section, we apply our sub-Gaussian estimator to computing one-sided $p$-values in large-scale permutation testing. Suppose that we wish to test hypothesis $g \in \{1, 2, \ldots, G\}$, observe test statistic $Z^{(g)}_{\mathrm{obs}}$, and generate $M$ permutations to compute the null test statistic samples $Z_{1}^{(g)}, \dots, Z_M^{(g)}$. 
This yields an empirical $p$-value no smaller than $\frac{1}{M+1}$. If computing the statistic $\{Z^{(g)}_m\}_{g \in [G]}$ for each permutation is computationally expensive, then $M$ may be much smaller than $G$ so that, after multiple testing correction such as Bonferroni or Benjamini--Hochberg, it can be impossible to reject any null hypotheses. The problem is that empirical $p$-value remains at $\frac{1}{M+1}$ even if $Z^{(g)}_{\mathrm{obs}}$ is much larger than $\max_{m \in [M]} Z_m^{(g)}$. We propose to address this issue with the sub-Gaussian tail inequality where we plug in an estimate of the sub-Gaussian parameter. 

Write $P^{(g)}$ as the null distribution of $Z_m^{(g)}$ and suppose we know that $P^{(g)}$ is sub-Gaussian, which holds when, for instance, $Z_m^{(g)}$ is known to be bounded. Then, 
\begin{align*}
\mathbb{P}( Z_m^{(g)} \geq Z^{(g)}_{\text{obs}} \,|\, Z^{(g)}_{\text{obs}}) \leq \exp \biggl\{ - \frac{1}{2} \frac{(Z^{(g)}_{\text{obs}})_+^2 }{\xi_{\star,+}^2(P^{(g)})} \biggr\}
\end{align*}
where $\xi_{\star,+}^2(P^{(g)}) = \sup_{\lambda \geq 0} L_{P^{(g)}}(\lambda)$ is the one-sided sub-Gaussian parameter. Thus, we may estimate $\xi_{\star,+}^2(P^{(g)})$ with the one-sided $\hat{\xi}_{M,+}^2(g)$ computed on $\{Z_1^{(g)}, \ldots, Z_M^{(g)}\}$ and define the sub-Gaussian based one-sided $p$-value estimate as
%\noindent \textbf{Sub-Gaussian transform}: 
%Fix a hypothesis $g$, and assume that $P^{(g)}$ is sub-Gaussian (and thus centered). The sub-Gaussian tail bound justifies the $p$-value 
\vspace{-0.1in}
\[
\textsf{p}_{\mathrm{sG}}^{(g)} := \exp\left(- \frac{1}{2} \frac{(Z^{(g)}_{\mathrm{obs}} -  \bar Z^{(g)})_+^2}{ \hat \xi_{M,+}^2(g)}\right)
\]
where $\bar Z^{(g)} := \frac1M \sum_1^M Z^{(g)}_j$. Since $\hat \xi_{M,+}^2(g)$ is consistent, $\textsf{p}_{\mathrm{sG}}^{(g)}$ yields asymptotically valid type I error control. 

\textbf{Bulk-shrinking:} In our setting where only the right-tail is of interest, the sub-Gaussian based $p$-value may be overly conservative because the sub-Gaussian parameter is determined by the entire distribution -- both the bulk and the tail. To remove the influence of the bulk and the left-tail, we propose \textit{bulk-shrinking}, which leaves the right tail of the distribution intact while concentrating the bulk and left-tail to a single point. More precisely, for a threshold $u \geq 0$, we define the data preprocessing transformation
\[
\tilde{Z}^{(g)}_m := Z_m^{(g)} \mathbbm{1}\{ Z_m^{(g)} > u \} + c \mathbbm{1}\{ Z_m^{(g)} \leq u \}
\]
where $c \in \mathbb{R}$ is chosen so that $\{\tilde{Z}^{(g)}_m \}_{m=1}^M$ is centered. We then compute $\hat{\xi}_n^2(g)$ using the transformed data $\{\tilde{Z}^{(g)}_m\}_{m=1}^M$. 

Bulk-shrinking retains type I error control since $\mathbb{P}( Z^{(g)}_m \geq t) = \mathbb{P}(\tilde{Z}^{(g)}_m \geq t)$ for any $t \geq u$. The following lemma shows that bulk-shrinking can only increase statistical power. For a mean-zero distribution $P$ on $\mathbb{R}$ and $u \geq 0$, define the bulk-shrinking-transformed distribution as
$B_u[P] = P((-\infty, u]) \delta_c + P \mathbbm{1}_{(u, \infty)},$
%\vspace{-0.15in}
where $c = \E_P[X|X \leq u]$ is chosen to ensure that $B_u[P]$ also has mean-zero. 
\begin{lemma}
\label{lem:bulk-shrink}
    We have $\xi_{\star, +}^2(B_u[P]) \leq \xi^2_{\star, +}(P)$ where $\xi_{\star, +}^2(P) := \sup_{\lambda \geq 0}L_P(\lambda)$. 
\end{lemma}

Therefore, bulk-shrinking can never increase the resulting $p$-value. Although we focus on one-sided $p$-value, our discussion and the bulk-shrinking idea can be extended in a straightforward way to two-sided tests.

% Let $P$ be supported on $[-c, \infty)$ and sub-Gaussian, and fix a threshold $u$. We define the \textit{one-sided bulk-shrink} of $P$ by replacing any values at or below $u$ by their conditional mean $ m:= \E_P[X|X \leq u]$ while leaving observations above $u$ unchanged. Specifically we take 
% \vspace{-0.1in}
% \[
% B_uP:= P((-\infty, u])\delta_{m} + P\mathbbm1_{(u, \infty)}.
% \]

% Since $P$ is centered, the choice of $m$ keeps $B_uP$ centered as well. Importantly, by preserving the upper tail and collapsing the bulk to a single point, $B_uP$ eliminates any dependence on the shape of the bulk. 

% \vspace{-0.05in}
%By Jensen's inequality, this procedure cannot increase the moment generating function. Thus it also cannot sub-Gaussian parameter, as we now state:

% \vspace{-0.05in}
% Similarly, if $P$ is symmetric about zero, we shrink the central bulk while preserving both tails. For $u \geq 0$, let 
% \vspace{-0.1in}
% \[
% B_u^{(c)}P := P([-u, u])\delta_0 + P\mathbbm1_{(-\infty, -u) \cup (u, \infty)}.
% \]

% \vspace{-0.15in}
% A similar argument shows that $\xi_{\star,+}^2(B_u^{(c)}P) \leq \xi_{\star, +}^2(P)$.

We compare our proposed approach with an existing method for estimating $p$-value based on extreme value theory.

\noindent \textbf{Peaks-over-thresholds (POT)}: Write $F$ for the cdf of $P$. 
Another approach \citep{@knijnenburg2009fewer} is to define the \textit{excess function}
\[
F^{(u)}(y) =  \frac{ F(u+y) - F(u)}{1 - F(u)} = P(X - u \le y  \, | \, X > u).
\]
 By the Pickands-Balkema-de Haan theorem, as $u \to \sup \{x: F(x) < 1\}$, $F^{(u)}$ converges in distribution to a Generalized Pareto Distribution (GPD) with parameters $m,s,t$. This motivates fitting a GPD to the 
exceedances $\{Z^{(g)}_i - u^{(g)}: Z^{(g)}_i > u^{(g)}\}$, 
yielding estimates $\hat m^{(g)}, \hat s^{(g)}$, and $\hat Z^{(g)}$ and 
the $p$-value
\vspace{-0.1in}
\begin{multline*}\textsf{p}^{(g)}_{\mathrm{POT}}:= (1 - \hat F^{(g)}(u^{(g)})) \\ \times 
\left(1 +  
(\hat s_g^{(g)})^{-1}\hat Z^{(g)}(Z^{(g)}_\text{obs} - u^{(g)} - \hat m^{(g)})
\right)^{-\frac{1}{\hat t_g}}
\end{multline*}
where $\hat F^{(g)}$ is the empirical CDF. In practice, $u^{(g)}$ must be chosen with care. If it is too large, there may be too few exceedances to obtain a reliable fit. A common choice for $u^{(g)}$ is the 90-quantile.
\subsection{GO enrichment analysis}

\noindent \textbf{Background}:
In genomics, \textit{active module identification} (AMI) algorithms identify subsets of a gene-gene interaction network with high connectivity and candidate gene enrichment. Each module is tested with a set of genes called GO terms via the hypergeometric test, and the resulting $-\log_{10}$ $p$-value is called the \textit{HG-enrichment score} \citep{levi2021domino, liu2025beyond}.

For a GO term $g$, let $Z_g^{\mathrm{obs}}$ be the maximum
score across modules, and call $g$ \emph{HG-significant} if $Z^{\mathrm{obs}}_g$
exceeds a prescribed threshold.
Rerunning the pipeline on $n$ permuted datasets yields null
scores $Z_g^{(i)}$. HG-significant terms whose permutation
$p$-values pass Benjamini--Hochberg (BH) correction at level
$\alpha$ are called \emph{empirically validated}.

\noindent \textbf{Method Comparison}: We compare the GO terms rejected by both methods on widely used gene datasets. 
Following \cite{liu2025beyond}, we use the ``TNFa'' \citep{schmidt2015acute}  and ``Fly Transcriptome'' (FT) \citep{sun2024whole} datasets. We used the PAPER \citep{crane2024root}, DOMINO \citep{levi2021domino}, and FDRnet \citep{yang2021efficient} AMI algorithms. Other algorithms and datasets were excluded because they
yielded too few HG-significant GO terms.

For each algorithm and dataset, we compute $\textsf{p}^{(g)}_{\mathrm{sG}}$ and $\textsf{p}^{(g)}_{\mathrm{POT}}$ per GO term, apply BH correction at level $\alpha = 0.05$, and in Table~\ref{tab:ev-go-term} we report the number of empirically-validated GO terms followed by the number of BH-significant GO terms in parentheses. 

\begin{table}[htbp]
    \centering
    \begin{tabular}{llrr}
    \toprule
        Dataset & Algorithm & \textbf{sub-Gaussian} & \textbf{POT} \\\midrule
       FT & PAPER & 181 (279) & 115 (138) \\
       FT & DOMINO & 137 (143) & 167 (169) \\
       FT & FDRnet & 1 (1) & 1 (3) \\\midrule
       TNFa &  PAPER & 0 (0) & 4 (10) \\
       TNFa &DOMINO & 5 (5) & 25 (58) \\
       TNFa &FDRnet & 13 (39) & 20 (82)\\\bottomrule
    \end{tabular} 
    
\caption{\textbf{Empirically validated GO terms} by AMI algorithm 
(rows) and $p$-value method (columns); BH-significant counts in 
parentheses.}
    \label{tab:ev-go-term}
\end{table}

On the FT dataset, the sub-Gaussian transform yields more empirically validated GO terms under PAPER and comparable results for DOMINO and FDRnet. On the TNFa dataset, POT achieves more rejections across all three algorithms. The relative performance depends heavily on the null distribution, and the two methods are best understood as complementary.

The POT method can be unreliable for a few reasons. First, the rate of convergence can be slow: for normal tails, the empirical excess function $F^{(u)}_g$ converges to a GPD at rate $u^{-2}$ as $u \to \infty$ \citep{Degen_Embrechts_Lambrigger_2007}, but the threshold $u_g$ can be no greater than the sample maximum, which is $O(\sqrt{\log n})$ for sub-Gaussian data. The resulting $O(1/\log n)$ rate gives the $p$-values a questionable validity for small $n$. Second, for PAPER on the FT dataset, 255 out of 279 (91.4\%) of the POT fits have $\hat t_g < 0$, which corresponds to a distribution supported on $[0, m_g - s_g/t_g]$. Eight of the $Z_g^{\text{obs}}$'s exceeded $m_g - s_g/t_g$ resulting in a $p$-value of 0, a clear underestimate. Third, the standard 90th-quantile choice of $u_g$ does not diverge with 
$n$, but larger thresholds leave few exceedances in the light-tailed case, and threshold selection has no general solution \citep{Degen_Embrechts_Lambrigger_2007}.

On the other hand, $\hat \xi_n^2$ may also exhibit slow convergence under certain distributions (for example $\frac12 N(0, 1) + \frac12 N(0, 2)$). If the underlying distribution is heavy tailed, the divergence of $\hat \xi_n^2$ would make the $p$-values useless in comparison to the POT approach. In practice, we view these methods as complementary; when the data is sufficiently light-tailed the sub-Gaussian transform may be more reliable, in other cases POT may be better.

\vspace{-0.1in}
\section{Discussion}
\label{sec:discussion}
\vspace{-0.1in}

We have studied the estimation of the sub-Gaussian parameter $\xi_\star^2$ and showed that the decay of the truncation gap directly controls the minimax risk. When the truncation gap is uniformly controlled, our estimator is nearly minimax optimal, and in the regime where $n^{-1/2}$ rate is possible, our estimator is minimax optimal. Under departures from sub-Gaussianity, the estimator diverges almost surely at a rate depending on the range of the observed data, providing a natural diagnostic of mis-specification. Our GO enrichment application demonstrate that for large scale permutation testing, sub-Gaussian tail bounds can act as a complementary approach to the POT method.

There are a few limitations to our work. The hypotheses of Theorem~\ref{prop:second-consistency} and Proposition~\ref{prop:ptwise-asymptotic-normality} are difficult to deduce from observed data, though the heuristic in Section~\ref{sec:inference} offers practical guidance. The rate in Theorem~\ref{prop:second-consistency}(b) has an additional $d_n$ term which is not known to be sharp, and further analysis is required to close this gap. While taking $C_n = (\log n)^\alpha$ is theoretically sound, a data-driven choice would be more practical. Lastly, the minimax lower bound in Theorem~\ref{thm:minimax-lower-bound} is not met. Future work could address these directions.

\textbf{Acknowledgements:} The authors thank Lucy Xia (HKUST) for an insightful comment regarding the minimax lower bound. This study was partially supported by the National Institute of General Medical Sciences (NIGMS) (R01GM157610 to MX and JX), the National Science Foundation (DMS-2113671 to MX), and a grant from the Office of the Vice Provost for Research, Rutgers–New Brunswick. The content is solely the responsibility of the authors and does not necessarily represent the official views of the Office of the Vice Provost for Research. 
\subsection*{AI use statement}
In this work, we used generative AI tools to: 
\begin{enumerate}
    \item Generate synthetic datasets
    \item Clean/reformat datasets
    \item Assist in the writing of proofs
    \item Implement estimation methods in code
\end{enumerate}
We have not used generative AI tools to:
\begin{enumerate}
    \item Develop theoretical models or conceptual frameworks,
    \item Formulate mathematical claims,
    \item Provide critical ingredients for proving mathematical claims,
    \item Design or provide feedback on research methodology or experiments,
    \item Support qualitative and thematic data analysis, 
    \item Propose or refine hypotheses,
    \item Interpret results.
\end{enumerate}
AI-assisted translation is not applicable to this work.
Additionally, we used generative AI tools to:
\begin{enumerate}
    \item Search for information,
    \item Write code for numerical experiments and figure generation,
    \item Summarize existing literature, 
    \item Edit the paper for wording and clarity, and 
    \item Check proofs for correctness. 
\end{enumerate}

We have carefully reviewed all AI-assisted work. We checked any references that were reported by LLMs. We also verified LLM-generated code for correctness. We take responsibility for the final content of this work, including text, claims, or artifacts produced with the aid of generative AI.
\bibliography{all}
\newpage
\section*{Checklist}

\begin{enumerate}

  \item For all models and algorithms presented, check if you include:
  \begin{enumerate}
    \item A clear description of the mathematical setting, assumptions, algorithm, and/or model. \textbf{[Not Applicable]} No models or algorithms are introduced, and we clearly define any estimation procedures. 
    \item An analysis of the properties and complexity (time, space, sample size) of any algorithm. \textbf{[Not Applicable]} No non-trivial algorithms are discussed. 
    \item (Optional) Anonymized source code, with specification of all dependencies, including external libraries. [Yes/No/Not Applicable]
  \end{enumerate}

  \item For any theoretical claim, check if you include:
  \begin{enumerate}
    \item Statements of the full set of assumptions of all theoretical results. \textbf{[Yes]} All theoretic results are stated clearly with all assumptions in the main text. 
    \item Complete proofs of all theoretical results. \textbf{[Yes]} Proofs of all results can be found in the appendix. 
    \item Clear explanations of any assumptions. \textbf{[Yes]}
  \end{enumerate}

  \item For all figures and tables that present empirical results, check if you include:
  \begin{enumerate}
    \item The code, data, and instructions needed to reproduce the main experimental results (either in the supplemental material or as a URL). \textbf{[Not Applicable]} Numerical experiments, while present, are not a main contribution of our paper. 
    \item All the training details (e.g., data splits, hyperparameters, how they were chosen). \textbf{[Not Applicable]} 
    \item A clear definition of the specific measure or statistics and error bars (e.g., with respect to the random seed after running experiments multiple times). \textbf{[Not Applicable]}
    \item A description of the computing infrastructure used. (e.g., type of GPUs, internal cluster, or cloud provider). \textbf{[Not Applicable]}
  \end{enumerate}

  \item If you are using existing assets (e.g., code, data, models) or curating/releasing new assets, check if you include:
  \begin{enumerate}
    \item Citations of the creator if your work uses existing assets. \textbf{[Yes]}
    \item The license information of the assets, if applicable. \textbf{[Not Applicable]}
    \item New assets either in the supplemental material or as a URL, if applicable. \textbf{[Yes]} We included any code used for the large-scale permutation testing application in our supplement. 
    \item Information about consent from data providers/curators. \textbf{[Not Applicable]}
    \item Discussion of sensible content if applicable, e.g., personally identifiable information or offensive content. \textbf{[Not Applicable]}
  \end{enumerate}

  \item If you used crowdsourcing or conducted research with human subjects, check if you include:
  \begin{enumerate}
    \item The full text of instructions given to participants and screenshots. [Not Applicable]
    \item Descriptions of potential participant risks, with links to Institutional Review Board (IRB) approvals if applicable. [Not Applicable]
    \item The estimated hourly wage paid to participants and the total amount spent on participant compensation. [Not Applicable]
  \end{enumerate}

\end{enumerate}

\onecolumn
\appendix
\section{Appendix}
\label{sec:appendix}
Throughout the appendix, we will write $K(\alpha)$ to mean a constant depending only on $\alpha$, and add more variables if needed. These constants may change from line to line without renaming. 
\subsection{Proofs for Section 2}
\label{sec:app-minimax-lower}
In this section, our main goal will be to prove the lower bound on the minimax risk of Theorem~\ref{thm:minimax-lower-bound}. The upper bound is based on estimator $\hat \xi_n^2$ and will be presented in Section~\ref{sec:app-minimax-upper}.  

\begin{proposition}
\label{prop:xi-L-char}
    If $X \sim P$ is sub-Gaussian with variance $\sigma^2$, then 
    \begin{equation}
        \xi_\star^2(P) = \sup_{\lambda \in \R} L_P(\lambda),
    \end{equation}
    where
    $
     L_P(\lambda) := \begin{cases}
        \frac{2}{\lambda^{2}} \log \E_P [e^{\lambda X}] & \text{if } \lambda \ne 0 \\
        \sigma^2 & \text{if } \lambda = 0
    \end{cases}
    $
    is real analytic and thus infinitely continuously differentiable. For fixed $P$, we write $L(\lambda) \equiv L_P(\lambda)$ and $\xi_\star^2 \equiv \xi_\star^2(P)$.
\end{proposition}

\begin{proof}[Proof of Proposition~\ref{prop:xi-L-char}]
    Firstly, since $X$ is sub-Gaussian, $M(\lambda) < \infty$ for all $\lambda \in \mathbb{R}$ and hence $|L(\lambda)| < \infty$ as well. Now we observe that, by rearranging the inequality, we have that $\xi > 0$ satisfies
    \[
    \E[e^{\lambda X}] \leq \exp\left(\frac{\lambda^2\xi^2}{2}\right)
    \]
    if and only if $\xi^2 \geq L(\lambda)$ for $\lambda \ne 0$. Since $\xi^2 \geq \sigma^2$ and $\lambda$ was arbitrary, we have $\xi^2 \geq \sup_{\lambda \in \R} L(\lambda)$, so  $\xi_\star^2 \geq \sup_{\lambda \in \R} L(\lambda)$ too. Conversely, for any $\eps > 0$ there must be $\lambda \in \R$ such that 
    \[
    \E[e^{\lambda X}] >\exp\left(\frac{\lambda^2(\xi_\star- \eps)^2}{2}\right),
    \]
    i.e. $(\xi_\star - \eps)^2 < 2\lambda^{-2}\psi(\lambda) = L(\lambda)$, so that $\xi_\star$ is the least upper bound. We now show that $L$ is continuous. This can be done by applying L'H\^opital's rule to $ \lim_{\lambda \to 0} 2\lambda^{-2}\psi(\lambda)$ and recalling that $\psi''(0) = \sigma^2$. To prove that $L$ is analytic, we claim that the following holds for all $\lambda \in \R$:
    \begin{equation}
        L(\lambda) = 2\int_0^1 (1 - x)\psi''(\lambda x) \, dx
        \label{eqn:L-integral-rep}
    \end{equation}
    To see this, note that $\psi(0) = \psi'(0) = \E X= 0$ so by Taylor's theorem and then a change of variable,
    \[
    \psi(\lambda) = \int_0^\lambda (\lambda - x)\psi''(x) \, dx = \lambda^2\int_0^1 (1 - x) \psi''(\lambda x) \, dx 
    \]
    so the claim follows by recalling $L(\lambda) = 2\lambda^{-2}\psi(\lambda)$. Meanwhile, since $\psi''(0) = \sigma^2$, we have 
    \[
    L(0) = 2\int_0^1 (1 - x) \psi''(0) \, dx.
    \]
    Finally, $M(\lambda)$ is real-analytic, so $\psi''(\lambda)$ must be real-analytic too. By Fubini's theorem, conclude that $L$ is real-analytic. 
\end{proof}
Now we address minimax lower bounds. Our proof relies on an indirect construction based on the intermediate value theorem. To that end, we first show continuity of $P \mapsto \xi_\star^2(P)$ over the class of distributions we will use to derive our minimax lower bounds from. 

\begin{lemma}
    For $a \geq 1$, let $P_{a, n} = (1 - \frac1n)(\frac12\delta_{1} + \frac12\delta_{-1}) + \frac1n(\frac12\delta_{a} + \frac12\delta_{-a})$. When $n$ is fixed, $a \mapsto \xi^2_\star(P_{a, n})$ is continuous function on $[1, \infty)$. 
\end{lemma}
\begin{proof}
    Let $(a_j)_j$ be a sequence that converges to $a$. Since the $a_j$ are bounded, we may restrict the $P_{a_j, n}$ to a compact subset of $\mathbb R$ where $x \mapsto e^{\lambda x}$ is a bounded continuous function. Then weak convergence gives $M(\lambda; P_{a_j, n}) \to M(\lambda; P_{a, n})$ pointwise in $\lambda$. Conclude that  $L(\lambda; P_{a_j, n}) \to L(\lambda; P_{a, n})$ pointwise. Since any maximizer $\lambda^\star$ of $L(\,\cdot\,; P_{a, n})$ must be a critical point, we use $|\psi'(\lambda; P_{a, n})| \leq a$ to see that 
    \[
    |\lambda^\star| = \xi_\star^{-2}(P_{a, n})|\psi'(\lambda; P_{a, n})| \leq a\xi_\star^{-2}(P_{a, n}) \leq  a\sigma^{-2}(P_{a, n}) = \frac{a}{1 + n^{-1}(a^2 -1)}.
    \]
    Since $(a_j)_j$ converges, conclude that there is a constant $C > 0$ depending only on $a_0$ such that 
    \[
    \sup_{|\lambda| \leq C} L(\lambda; P_{a_j, n}) = \sup_{\lambda \in \mathbb R} L(\lambda; P_{a_j, n}). 
    \]
    Now use the bound $M^{(k)}(\lambda) \leq a_j^k M(\lambda)$ to see that 
    \[
    |\psi^{(3)}| \leq  \frac{|M^{(3)}M^2|+ 3|M'M'' M| + 2|M'|^3}{|M|^3} \leq 17a_j^3,
    \]
    so by \eqref{eqn:L-integral-rep}, we differentiate under the integral sign to conclude 
    \[
    L'(\lambda;P_{a_j, n}) = 2\int_0^1 x(1 - x) \psi'''(\lambda; P_{a_j, n}) \, dx \leq 17a_j^3. 
    \]
    Again, since $a_j$ converges, this is bounded by a constant depending only on $a_0$. Thus $\{L(\lambda; P_{a_j, n})\}_j$ is equicontinuous; this promotes the pointwise convergence to uniform convergence on $[-C, C]$. The claim readily follows. 
\end{proof}
For the remainder of this section, let $\Xi > 0$ be fixed. We now construct the distributions that we will apply Le Cam's method to, to obtain the $\sim r(\sqrt{\log n})$ lower bound on the minimax risk. 
\begin{proposition}
\label{prop:le-cam-r-construction-1}
    Let $r: [0, \infty) \to [0, \infty)$ and assume $\lim_{t \to\infty} r(t) < 1$. There exists $n_0(r) > 0$ such that, for any $n \geq n_0$, there are $a_n \in [1, 2\sqrt{\log n}]$ such that $\xi_\star^2(P_{a, n}) = 1 + r(4\sqrt{\log n})$.
\end{proposition}
\begin{proof}
    We have 
    \[
    M(\sqrt{\log n}; P_{2\sqrt{\log n}, n}) \geq \frac{\cosh(2\log n)}{n} \geq \frac n2,
    \]
    so 
    \[
    L(\sqrt{\log n}; P_{2\sqrt{\log n}, n}) \geq \frac{2(\log n - \log 2)}{\log n} \geq 2 - \frac{2\log 2}{\log n}
    \]
    Since $\lim_{t \to \infty} r(t) < 1$, pick $\delta, n_0 > 0$ such that $r(4\sqrt{\log n}) \leq 1 - \delta$ for all $n > n_0$. If needed, increase $n_0$ to make $\delta \log n \geq 2 \log 2$, so that the following inequality is true: 
    \[
   \xi_\star^2(P_{1, n}) = 1 \leq 1 + r(4\sqrt{\log n}) \leq 2 - \frac{2\log 2}{\log n} \leq \xi_\star^2(P_{2\sqrt{\log n}, n}).
    \]
    Since $a \mapsto \xi_\star^2(P_{a, n})$ is continuous, the intermediate value theorem yields the desired $a_n$.
\end{proof}
Next, we show that these distributions actually belong to the desired class. 
\begin{proposition}
\label{prop:le-cam-r-construction-2}
    Let $r : [0, \infty) \to [0, \infty)$ be a decreasing function with $\lim_{t \to \infty} r(t) < 1$. If $1 \leq a \leq2\sqrt{\log n}$ and $\xi_\star^2(P_{a, n}) = 1 + r(4\sqrt{\log n})$, then there exists $n_0$ such that for any $n > n_0$ we have $P_{a, n} \in \mathcal P( 2, r, 0)$.
\end{proposition}
\begin{proof}
Let $C > 0$. First, if $C \geq 4\sqrt{\log n}$, then we use the fact that $|a| \leq 2\sqrt{\log n}$ to see that   
\[
L(\lambda; P_{a, n}) \leq \frac{2a}{\lambda} \leq 1 < \xi_\star^2(P_{a, n})
\]
whenever $|\lambda| \geq C$. 
This proves that $\delta_{P_{a, n}}(C) \leq 0 \leq r(C)$. On the other hand, if $C \leq 4\sqrt{\log n}$, we use $1 \leq a$ to deduce that $\sigma^2(P_{a, n}) = 1 + n^{-1}(a^2 - 1) \geq 1$, so 
\[
\delta_P(C) \leq \xi_\star^2(P_{a, n}) - \sigma^2(P_{a, n}) \leq r(4\sqrt{\log n}) \leq r(C),
\]
since $r$ is decreasing. Finally, to see that $\xi_\star^2(P_{a, n}) \leq 2$, use the fact that $\lim_{t \to \infty} r(t) < 1$. 
\end{proof}
Next, we verify properties for the three-point distributions that we apply Le Cam's method to, that will supply a $\sim n^{-1/2}$ lower bound for the minimax risk. A key advantage of using these distributions is that the sub-Gaussian parameter can be computed explicitly. 
\begin{proposition}[\cite{atouani2025optimalsubgaussianvarianceproxy}, Theorem 3.1(a)]
\label{prop:3-point-sub-gauss}
    For $a, p > 0$, let $P := (1 - 2p)\delta_0 + p\delta_{-a} + p\delta_a$. If $p \in [1/6, 1/2)$, then $\xi_*^2(P) = \sigma^2(P) = 2pa^2$.
\end{proposition}
Now we show membership of these three-point distributions in the desired class. 
\begin{proposition}
\label{prop:unif-supp-delta-inclusion}
    Let $P \in \mathcal P(\Xi)$. If $P$ is supported on $[-1, 1]$, then $P \in \mathcal P(\Xi, 0, r_0)$ where $r_{0}(t) = 2/t$. If in addition $\sigma^2(P) \geq \frac32\delta_0$, then $P \in \mathcal P(\Xi, 8\delta_0^{-1}, r_-)$ where $r_-(t) = - \delta_0$.
\end{proposition}
\begin{proof}
    Let $P$ be supported on $[-1, 1]$. Then, for any $\lambda \in \mathbb{R}$, 
    \[
    L(\lambda; P) \leq \frac{2}{\lambda^2} \log e^{ |\lambda|} \leq \frac{2}{|\lambda|}.
    \]

    Combining this with the fact that $L(0 ; P) = \sigma^2(P) \geq 0$, we have that, for any $C > 0$,
    \begin{align*}
    \delta_P(C) = \sup_{|\lambda| \geq C} L(\lambda; P) - \sup_{| \lambda | \leq C} L(\lambda; P) 
    \leq \sup_{|\lambda| \geq C} L(\lambda ; P) \leq \sup_{|\lambda| \geq C} \frac{2}{|\lambda|} \leq \frac{2}{C}.
    \end{align*}
    
   The first claim of the Proposition thus follows.

   Now suppose $\sigma^2(P) \geq \frac32\delta_0$. Let $C_1 := 8 \delta_0^{-1}$ and observe that
   \[
   \sup_{|\lambda| \geq C_1} L(\lambda; P) \leq \sup_{|\lambda| \geq C_1} \frac{4}{|\lambda|}  = \frac{4}{C_1} \leq \frac{\delta_0}{2}.
   \]

    Therefore, we have 
    \begin{align*}
    \delta_P(C_1) &= \sup_{|\lambda| \geq C_1} L(\lambda ; P) - \sup_{|\lambda| \leq C_1} L(\lambda ; P) \leq \frac{\delta_0}{2} - L(0 ; P) \\
    &= \frac{\delta_0}{2} - \sigma^2(P) \leq - \delta_0,
    \end{align*}
as desired. The Proposition thus follows. 
\end{proof}
To simplify the proof of Theorem~\ref{thm:minimax-lower-bound}, we will prove the following lemma, which says that we may prove the minimax lower bound for special cases of $\Xi$, and then rescale to general values. 
\begin{proposition}
\label{prop:subgauss-scaling}
    Let $X \sim P$ be sub-Gaussian and write $P_a := \mathrm{Law}(aX)$. We have that 
    \[
    \mathcal M_n(\mathcal P(\Xi, r, C_0)) = \Xi \mathcal M_n(\mathcal P(1, \tilde r, \sqrt\Xi C_0)),
    \]
    where $\tilde r (t) := \Xi^{-1}r(t/\sqrt\Xi)$. 
\end{proposition}
\begin{proof}
    We first prove that  $\xi_\star^2(P_a) = a^2\xi_\star^2(P)$ and $\delta_{P_a}(C) = a^2\delta_P(aC)$. That $a^2\xi_\star^2(P) \leq \xi_\star^2(P_a)$ is clear from the definition. For the lower bound, let $(\lambda_n)$ be a sequence such that  $L(\lambda_n; P) \to \sup_\lambda L(\lambda; P)$. It then follows that $L(\lambda_n/a; P_a) = a^2L(\lambda_n; P)$, so $ L(\lambda_n/a; P_a) \to a^2\xi_\star^2(P)$, and thus $\xi_\star^2(P_a) \geq a^2 \xi_\star^2(P)$. The second claim follows by using the identity $L(\lambda; P_a) = a^2L(a\lambda; P)$ again. With these facts in hand, let $\hat \xi_\Xi$ be an estimator for $\xi_\star^2(P)$. Then if 
    \[
    \hat \xi_1^2 (X_1, \dots, X_n):= \Xi^{-1}\hat \xi_\Xi (\sqrt\Xi X_1, \dots, \sqrt\Xi X_n),
    \]
    we have that 
    \[
    \E_{P_{1/\sqrt\Xi}} |\hat \xi_1^2 - \xi_\star^2(P_{1/\sqrt\Xi})| = \Xi^{-1} \E_P|\hat \xi^2_\Xi - \xi^2_\star(P)|. 
    \]
    Since these transformations are bijective, the equality follows. 
\end{proof}
Having completed the constructions necessary, we now prove the two key lower bounds for specific cases of $\Xi$, $r$, and $C_0$.
\begin{proposition}
\label{prop:minimax-lb-1}
    Let $r: [0, \infty) \to [0, \infty)$ be a decreasing function with $\lim_{t \to \infty} r(t)  < 1$. There exists a constant $n_0$ depending only on $r $ such that, whenever $n > n_0$, 
    \[
    \mathcal M_n(\mathcal P(2, r, 0)) \gtrsim r(4\sqrt{\log n}). 
    \]
\end{proposition}
\begin{proof}
    By replacing $r$ by its positive part, we may henceforth assume that $r(t) \geq 0$ for all $t$. Since $P_{1, n} = \frac12(\delta_{1} + \delta_{-1})$ satisfies $\sigma(P_{1, n}) = 1=\xi_\star(P_{1, n})$, we have that  $\delta_{P_{1,n}}(C)\leq 0 \leq r(C)$ for any $C > 0$ and therefore $P_{1, n} \in \mathcal P(2, r, 0)$. Make $n_0 > 0$ large enough so that Propositions \ref{prop:le-cam-r-construction-1} and \ref{prop:le-cam-r-construction-2} allow us to choose $a_n$ such that $P_{a_n, n} \in \mathcal P(2, r,0)$ and
    \[
    |\xi_\star^2(P_{1, n}) - \xi_\star^2(P_{a_n, n})| = r(4\sqrt{\log n}).
    \]
    Let $H$ denote Hellinger distance. We have 
    \begin{align*}
        H^2(P_{1, n}, P_{a, n}) &= \frac12 \sum_{x \in \{\pm1, \pm a\}} \left(\sqrt{P_{1, n}(x)} - \sqrt{P_{a, n}(x)}\right)^2 \\ 
        &= 1  - \sum_{x \in \{\pm1, \pm a\}} \sqrt{P_{1, n}(x)P_{a, n}(x)} \\
        &= 1 - \sqrt{1 - n^{-1}}.
    \end{align*}
    Now since $1 - H^2(P_{1, n}^{\otimes n}, P_{a_n, n}^{\otimes n}) = (1 - H^2(P_{1, n}, P_{a_n, n}))^n$ and $\mathrm{TV}(P_{1, n}^{\otimes n}, P_{a_n, n}^{\otimes n}) \leq \sqrt 2 H(P_{1, n}^{\otimes n}, P_{a_n, n}^{\otimes n})$ we deduce that 
    \[
    \mathrm{TV}(P_{1, n}^{\otimes n}, P_{a_n, n}^{\otimes n}) \leq \sqrt{2(1 - (1 - n^{-1})^{n/2}},
    \]
    whence $\limsup_{n \to \infty} \mathrm{TV}(P_{1, n}^{\otimes n}, P_{a_n, n}^{\otimes n}) \leq \sqrt{2(1 - e^{-1/2})}$. Therefore, we may increase $n_0$ until 
    \[
    1 - \mathrm{TV}(P_{1, n}^{\otimes n}, P_{a_n, n}^{\otimes n}) \geq \frac{1 -\sqrt{2(1 - e^{-1/2})}}{2}, 
    \]
    and by Le Cam's two-point method, conclude that
    \[
    \inf_{\hat \theta_n : \mathbb R^n \to \mathbb R} \sup_{P \in \mathcal P(2, r, 0)} \mathbb E_P|\hat \theta_n(X_1, \dots, X_n) - \theta(P)| \geq C_1r(4\sqrt{\log n}), \quad C_1 = \frac{1 - \sqrt{2(1 - e^{-1/2})}}{8}
    \]
    whenever $n > n_0$. 
\end{proof}
Our first lower bound is useful when $r(t) > 0$. However, if $r$ ever becomes non-positive, the lower bound becomes trivial. In this case, we appeal to a different strategy. 
\begin{proposition}
\label{prop:minimax-lb-2}
    Let $r(t) \equiv -1/3$. We have 
    \[
    \mathcal M_n(\mathcal P(1, r, 24)) \geq \frac{1}{8\sqrt n}\biggl(1 - \frac{1}{\sqrt{6}}\biggr). 
    \]
\end{proposition}
\begin{proof}
    
    For $q \in (0, 1)$, define $P_q := (1 - q)\delta_0 + (q/2)\delta_1 + (q/2)\delta_{-1}$. Define $q_0 = 1/2$ and $q_n := q_0 + 1/(4\sqrt n)$. 
    
    It is clear that $\xi^2_\star(P_{q_0})$ and $\xi^2_\star(P_{q_n})$ are both no larger than 1 since the distributions are supported on $[-1, 1]$. Moreover, by Proposition~\ref{prop:unif-supp-delta-inclusion}, $P_{q_0}, P_{q_n} \in \mathcal P(1, r, 24)$. Then, by Proposition~\ref{prop:3-point-sub-gauss} and the fact that $q_0$ and $q_n$ are both greater than $1/3$, we have
    \[
    |\xi_\star^2(P_{q_0}) - \xi_\star^2(P_{q_n})| = |\sigma^2(P_{q_0}) - \sigma^2(P_{q_n})| =  1/(4\sqrt n).
    \]
    On the other hand, we have by Theorem 5 of \cite{gibbs2002choosingboundingprobabilitymetrics} that 
    \begin{align*}
        \text{KL}(P_{q_0}, P_{q_n}) &= q_0 \log \frac{q_0}{q_n} + (1-q_0)\log \frac{1-q_0}{1-q_n} \\
        &= \frac12 \log \frac{1}{2q_n} + \frac12 \log \frac{1}{2(1 - q_n)} \\ &= \text{KL}(\text{Ber}(q_0), \text{Ber}(q_n)) \\ 
        &\leq \chi^2(\text{Ber}(q_0), \text{Ber}(q_n)) \\
        &= \frac{(q_0 - q_n)^2}{4 q_n} + \frac{(q_0 - q_n)^2}{(1- q_n)} \\
        &\leq \frac{(1/4)^2}{n} \biggl( \frac{1}{q_n} + \frac{1}{1-q_n}\biggr) = \frac{(1/4)^2}{n}\frac{ (1 + 3q_n)}{4 q_n(1 - q_n)} \leq \frac{1}{3n},
    \end{align*}
    so $\text{KL}(P_{q_0}^{\otimes n}, P_{q_n}^{\otimes n} ) = n\text{KL}(P_{q_0}, P_{q_n}) \leq \frac{1}{3}$. By Pinsker's inequality, we have that $\mathrm{TV}(P_{q_0}^{\otimes n}, P_{q_n}^{\otimes n}) \leq \sqrt{ \frac{1}{2} \mathrm{KL}( P_{q_0}^{\otimes n}, P_{q_n}^{\otimes n} )} \leq \sqrt{1/6}$. Conclude by Le Cam's two-point method that 
    \[
    \mathcal M_n(\mathcal P(1, r, 24)) \geq \frac{1}{8\sqrt n}\biggl(1 - \frac{1}{\sqrt{6}}\biggr). 
    \]
\end{proof}
Having completed the two core estimates needed, we combine them and appeal to rescaling to prove the lower bound in Theorem~\ref{thm:minimax-lower-bound}. The upper bound will be proven in Section~\ref{sec:app-minimax-upper}. 
\begin{proof}[Proof of the lower bound in Theorem~\ref{thm:minimax-lower-bound}]
We first address the $n^{-1/2}$ term of the bound. To produce a general minimax lower bound, we rescale the class $\mathcal P(1, r, 24)$ by $\sqrt \Xi$. We may assume $\sqrt\Xi \geq 24C_0^{-1}$ since otherwise the bound is trivial. This implies that $24/\sqrt \Xi \leq C_0$. Since $\lim_{t \to \infty} r(t)>-\Xi/3$, we see that our rescaled distributions belong to the class $\mathcal P(\Xi, -\Xi/3, 24/\sqrt \Xi) \subset \mathcal P(\Xi, r, C_0)$. Therefore 
\[
\mathcal M_n(\mathcal P(\Xi, r, C_0)) \geq C_2 \Xi_1 n^{-1/2}
\]
by Proposition~\ref{prop:minimax-lb-2}. In the case that $r(t_0) \leq 0$ for some $t_0$, we increase $n_0$ until $r(4\Xi^{-\frac12}\sqrt{2\log n})_+$ is zero for $n > n_0$. 

Now further assume that $r(t) > 0$ for every $t > 0$, where the other part of our lower bound becomes nontrivial. We generalize to arbitrary $\Xi$ through rescaling by $\sqrt {2/\Xi}$, and apply the claim to 
\[
\tilde r(C) := \frac{2}{\Xi} r\left(C\sqrt{\frac{2}{\Xi}}\right),
\]
which the rescaled $\delta$ belongs to. We may also decrease $C_0$ to $0$, since this only shrinks the class $\mathcal P(\Xi, r, C_0)$. Scaling and then applying Proposition~\ref{prop:minimax-lb-1} then gives 
\[
\mathcal M_n(\mathcal P(\Xi, r, C_0)) \geq \frac{\Xi}{2}\mathcal M_n(\mathcal P(2, \tilde r, 0)) \geq \frac{C_1\Xi}{2}\tilde r(4\Xi^{-1} \sqrt{2\log n}) = r\left(4\Xi^{-\frac12}\sqrt{2\log n}\right).
\]
Assembling both cases proves that 
\[
\mathcal M_n(\mathcal P(\Xi, r, C_0)) \gtrsim r\left(4\Xi^{-1}\sqrt{2\log n}\right)_+ \vee c_1\Xi_1^2n^{-\frac12} \geq \frac12\left(r\left(4\Xi^{-1}\sqrt{2\log n}\right) + c_1\Xi_1^2n^{-\frac12} \right),
\]
where we have adjusted constants as needed. 
\end{proof}

\subsection{Proofs for Section 3}
\subsubsection{Proofs for Section }

Our first main goal of this section is prove Proposition~\ref{prop:sub-gaussian-convergence-lemma}. Throughout, assume $X_1, \dots, X_n \stackrel{\mathrm{iid}}\sim P$.   

First, we need the following concentration inequality. Note that we cannot immediately apply any concentration inequalities to $M_n^{(k)}$, because the terms in the summation are not independent. 
\begin{proposition}
\label{propn:mgf-concentration-ineq}
    Let $k \in \N_0$, $B> 1$, and $C, t, \delta > 0$ satisfy $2\delta \leq B$. Assume that $M(\lambda)$ exists at $\lambda$, $|X| \leq B$ a.s., $| \E X| \leq \delta$, and 
   \begin{equation}
       4\delta\left(k(2B)^{k-1} + C(2B)^k\right)e^{2BC} \le t.
       \label{eqn:hoeffding-condition}
   \end{equation}
    Then for any $|\lambda| \leq C$, 
    \[
    \mathbb{P}\left(|M_n^{(k)}(\lambda) - M^{(k)}(\lambda)| > t\right) \le 2\exp\left(-\frac{n\delta^2}{2B^2}\right) + 2\exp\left(-\frac{nt^2}{8B^{2k}e^{2BC}}\right).
    \]
\end{proposition}
\begin{proof}
     Put $\tilde M_n(\lambda) := \frac1n\sum_{i=1}^n e^{\lambda X_i}$, and then observe that by the triangle inequality,
    \begin{equation}
        |M_n^{(k)}(\lambda) - M^{(k)}(\lambda)| \leq |M_n^{(k)} - \tilde M_n^{(k)}(\lambda)| + |\tilde M_n^{(k)}(\lambda) - M^{(k)}(\lambda)|.
        \label{eqn:mgf-conc-triangle}
    \end{equation}
    By Hoeffding's inequality, we have 
    \begin{equation}
        \P(|\bar X| > \delta + |\E X|) \leq 2\exp\left(-\frac{n\delta_n^2}{2B^2}\right). \label{eqn:hoeffding-mean}
    \end{equation}
    Now suppose that $|\bar X| < \delta_n$. For the first term of~\ref{eqn:mgf-conc-triangle}, define $g(\mu) := \frac1n\sum_{i=1}^n (X_i - \mu)^k e^{\lambda (X_i - \mu)}$. Then $g(0) = \tilde M_n^{(k)}(\lambda)$ and $g(\bar X) = M_n^{(k)}(\lambda)$, so by the mean-value theorem there is $\tilde \mu \in [0 \wedge \bar X, 0 \vee \bar X]$ such that 
    \[
    M_n^{(k)}(\lambda) - \tilde M_n^{(k)}(\lambda) = g'(\tilde \mu)\bar X,
    \]
    but for any $|\mu| \leq B$ and $|\lambda| \leq C$, 
    \[
    |g'(\mu)| \leq \frac1n\sum_{i=1}^n \left|k(X_i - \mu)^{k-1} + \lambda(X_i - \mu)^{k}\right|e^{\lambda(X_i - \mu)} \leq \left(k(2B)^{k-1} + C(2B)^{k}\right)e^{2BC},
    \]
    so 
    \[
    |M_n^{(k)}(\lambda) - \tilde M_n^{(k)}(\lambda)| \leq \left(k(2B)^{k-1} + C(2B)^{k}\right)e^{2BC }2\delta \leq t/2.
    \] 
    For the second term, note that when $\lambda \in [-C, C]$, 
    \[
    \tilde M^{(k)}_n(\lambda) = \frac1n\sum_{i=1}^n X_i^{k}e^{\lambda X_i} \leq B^ke^{B|\lambda|} \leq B^ke^{2BC},
    \]
    so by Hoeffding's inequality again 
    \begin{align}
        \P\left(|\tilde M_n^{(k)}(\lambda) - M^{(k)}(\lambda)| > t/2\right) &\leq 2\exp\left(-\frac{nt^2}{8B^{2k}e^{2BC}}\right).\label{eqn:hoeffding-mgf}
    \end{align}
    Combining (\ref{eqn:mgf-conc-triangle}) and (\ref{eqn:hoeffding-mgf}) proves the claim. 
\end{proof}
Equipped with Proposition~\ref{propn:mgf-concentration-ineq}, we are ready to prove a weaker analogue of Proposition~\ref{prop:sub-gaussian-convergence-lemma}, first for bounded random variables. 
\begin{proposition}
\label{prop:bounded-convergence-lemma}
    Let $k \in \N_0$. Fix $\alpha \in (0, 1/2)$ and $\beta \in (0, 1)$ such that $\alpha + \beta < 1$, $C_0 > 0$, and set
    \[
    C_n := C_0(\log n)^\alpha, \quad B_n := B_0\sqrt{\log n}, \quad C_0, B_0 > 0.
    \]
    Assume that $M(\lambda)$ exists for all $\lambda \in \R$, $|\E X_i| \leq n^{-1/2}(\log n)^{\beta + 1}$, and $\max_{i \in [n]}|X_i| \leq B_n$ almost surely. Then, there exists $n_0(\alpha, \beta,k,B_0, C_0)$ such that whenever $n \geq n_0$, with probability $1 - n^{-1}$, we have
    \[
    \sup_{\lambda \in [-C_n, C_n]} |M_n^{(k)}(\lambda) - M^{(k)}(\lambda)| \leq 5d_n n^{-1/2}
    \]
    where $d_n := (\log n)^{1+\beta }\{k(2B_n)^{k-1} + C_n(2B_n)^{k}\}e^{2B_nC_n}$. 
\end{proposition}
\begin{proof}
    First we note that, because $\alpha < 1/2$, we have 
    \begin{equation}
        \exp(C_nB_n) = \exp(C_0B_0(\log n)^{\alpha + 1/2}) = n^{C_0B_0(\log n)^{\alpha -1/2}} = n^{o(1)}.
\label{eqn:alpha-beta}
    \end{equation}
    For each $n$, let $\Pi_n \subset [-C_n, C_n]$ be a $n^{-1}$-net over $[-C_n, C_n]$ of size $N:=C_nn + 1$. For any $\lambda \in [-C_n, C_n]$, there is $\pi_\lambda \in \Pi_n$ such that $|\lambda - \pi_\lambda| < n^{-1}$, and the following holds by the triangle inequality:
    \begin{equation}
        |M_n^{(k)}(\lambda) - M^{(k)}(\lambda)| \leq |M_n^{(k)}(\lambda) - M_n^{(k)}(\pi_\lambda)| + |M_n^{(k)}(\pi_\lambda) - M^{(k)}(\pi_\lambda)| + |M^{(k)}(\pi_\lambda) - M^{(k)}(\lambda)|.\label{eqn:cover-proj-ineq} 
    \end{equation}
    Let us first address the middle term of (\ref{eqn:cover-proj-ineq}). Put $\delta := n^{-1/2}(\log n)^{\beta+1}$ and see that (\ref{eqn:hoeffding-condition}) is true when $t := 4d_nn^{-1/2}$. Note then that 
    \[
    \frac{nt^2}{8B^{2k}e^{2BC}} = 2^{2k+1}C_n^2 e^{2B_nC_n} (\log n)^{2 \beta + 2},
    \]
    so by Proposition~\ref{propn:mgf-concentration-ineq}, 
    \begin{multline*}
\P\left(\max_{\pi \in \Pi_n} |M_n^{(k)}(\pi) - M^{(k)}(\pi)| > 4d_nn^{-1/2}\right) \leq N\sup_{\pi \in \Pi_n}\P\left(|M_n^{(k)}(\pi) - M^{(k)}(\pi)| > 4d_nn^{-1/2}\right) \\ \leq 2N\exp\left(-\frac{n\delta^2}{2B_n^2}\right) + 2N\exp\left(-\frac{nt^2}{8B_n^{2k}e^{2B_nC_n}}\right) \leq 4N\exp(-c(\log n)^{1 + 2\beta}) \leq n^{-1},
        \end{multline*}
    where $c > 0$ is a constant, and $n$ is made sufficiently large. 
    Next, with probability one, $\max_i |X_i| \leq B_n$. On that event, for all $\lambda \in [-C_n, C_n]$,
    \[
    |M_n^{(k+1)}(\lambda)| \leq  \frac1n\sum_{i=1}^n|X_i - \bar X|^{k+1}e^{\lambda |X_i - \bar X|} \leq (2B_n)^{k+1}e^{2C_nB_n},
    \]
    which implies by the mean-value theorem that for any $\lambda, \lambda' \in [-C_n, C_n]$, there is $\tilde \lambda$ such that 
    \[
    |M_n^{(k)}(\lambda) - M_n^{(k)}(\lambda')| \leq |M_{n}^{(k+1)}(\tilde \lambda)||\lambda - \lambda'| \leq (2B_n)^{k+1}e^{2C_nB_n}|\lambda - \lambda'|
    \]
    and in particular 
    \begin{equation}
        |M_n^{(k)}(\lambda) - M^{(k)}_n(\pi_\lambda)| \leq (2B_n)^{k+1}e^{2C_nB_n}n^{-1}.\label{eqn:emgf-lipschitz}
    \end{equation}
    Identical reasoning then yields
    \begin{equation}
        |M^{(k)}(\lambda) - M^{(k)}(\pi_\lambda)| \leq (2B_n)^{k+1}e^{C_nB_n}n^{-1}.\label{eqn:mgf-lipschitz}
    \end{equation}
    Now make $n$ large enough that $(2B_n)^{k+1}e^{C_n B_n}n^{-1} \leq 2^{-1}d_nn^{-1/2}$, possible by the fact that
    \[\frac{(2B_n)^{k+1}e^{C_n B_n}n^{-1}}{d_nn^{-1/2}} \leq \frac{2B_nn^{-1/2}}{C_n(\log n)^{1 + \beta}} = \frac{2B_0}{C_0} (\log n)^{-(1/2 + \alpha + \beta)} n^{-1/2} \to 0.
    \]
    Thus we conclude, noting that $2^{-1}d_nn^{-1/2} + 4d_n^{-1/2} \leq 5d_nn^{-1/2}$.  
\end{proof}
Our next goal is to extend the result of Proposition~\ref{prop:bounded-convergence-lemma} to sub-Gaussian random variables. To do this, we first recall some facts about them.
\begin{proposition}
\label{prop:sub-Gaussian-props}
    Let $X$ be a sub-Gaussian random variable with parameter $\xi_*$. Then: 
    \begin{enumerate}
        \item[(a)] For all $t > 0$, $\P(|X - \E X| > t) \leq 2\exp\left(-\frac{t^2}{2\xi_\star^2}\right).$
        \item[(b)] For all $k \in\N$, $\E(|X|^{k}) \leq 2^{1+k/2}\xi_\star^k\Gamma(1+k/2)$.
        \item[(c)] For all $\lambda \in \R$, $M(\lambda) \leq \exp(2^{-1}\lambda^2\xi_\star^2)$.
    \end{enumerate}
\end{proposition}
We shall also need the following generalization of H\"older's inequality: 
\begin{proposition}[\cite{folland1999real}, Exercise 6.31]
\label{prop:generalized-holder}
    Let $1 \leq p, q, r \leq \infty$ satisfy $p^{-1} + q^{-1} + r^{-1} = 1$, and let $(X, \mathcal M, \mu)$ be a measure space. If $f \in L^p(\mu), g \in L^q(\mu), h \in L^r(\mu)$, then $fgh \in L^1(\mu)$. Moreover,
    \[
    \int |fgh| \, d\mu \leq \left(\int |f|^p \, d\mu\right)^{\frac1p} \left(\int |g|^q \, d\mu\right)^{\frac1q} \left(\int |h|^r \, d\mu\right)^{\frac1r}.
    \]
\end{proposition}
The idea is to show that sub-Gaussian random variables can be well-approximated by truncating their range. We may then apply Proposition \ref{prop:bounded-convergence-lemma} to conclude. 
\begin{proof}[Proof of Proposition~\ref{prop:sub-gaussian-convergence-lemma}]
Let $X^\flat_i = X_i\1\{|X_i| \leq 2\sqrt{\Xi \log n}\}$, $\bar X^{\flat} = \frac1n\sum_{i=1}^n X_i^{\flat}$, and 
    \[
    M_n^{\flat}(\lambda) := \frac1n\sum_{i=1}^n e^{\lambda (X^{\flat}_i - \bar X^{\flat})},  \quad M^\flat(\lambda) := \E[e^{\lambda X^{\flat}_1}],
    \]
    and observe that
    \[
    \left\{\max_{i \in [n]} |X_i| \leq 2\sqrt{\Xi \log n}\right\} = \left\{X^\flat_i = X_i, \text{ for all $1 \leq i \leq n$}\right\}
    \subseteq \left\{(M^\flat_n)^{(k)}(\lambda) = M_n^{(k)}(\lambda) \text{ for all $\lambda$}\right\}.
    \]
    These events occur with high probability uniformly over $\mathcal P(\Xi)$, because by a union bound and then the sub-Gaussian inequality we have
    \[
    \P\left(\max_{i \in [n]} |X_i| \geq 2\sqrt{\Xi \log n}\right) \leq n\P\left(|X_1| \geq 2\sqrt{\Xi \log n}\right) \leq \frac2n.
    \]
    For the rest of the proof, work on the event $(M^\flat_n)^{(k)}(\lambda) = M_n^{(k)}(\lambda)$ for all $\lambda$. We then have 
    \begin{align}
        |M_n^{(k)}(\lambda) - M^{(k)}(\lambda)| &\leq |M_n^{(k)}(\lambda) - (M^\flat)^{(k)}(\lambda)| + |(M^\flat)^{(k)}(\lambda) - M^{(k)}(\lambda)| \notag \\ 
        &\leq |(M_n^\flat)^{(k)}(\lambda) - (M^\flat)^{(k)}(\lambda)|+|(M^\flat)^{(k)}(\lambda) - M^{(k)}(\lambda)|.
        \label{eqn:trunc-triangle}
    \end{align}
   Since $\E X_i = 0$, we have by Cauchy-Schwarz and then the sub-Gaussian tail bound that 
    \begin{align*}
        |\E X_i^\flat| = |\E X_i\1\{|X_i| \geq 2\sqrt{\Xi \log n}\}| &\leq \sqrt{\E X_i^2 \P(|X_i| \geq 2\sqrt{\Xi \log n})} \\ &\leq 2\sqrt{\E X_i^2} \exp\left(-\frac{\Xi\log n}{\xi_\star^2}\right)\leq \frac{4\Xi}{n}.
    \end{align*}
    Therefore, when $n$ is sufficiently large (depending only on $\Xi$), the assumptions of Proposition~\ref{prop:bounded-convergence-lemma} are true, and 
   \begin{equation}
       \sup_{|\lambda| \leq C_n} |(M_n^\flat)^{(k)}(\lambda) - (M^\flat)^{(k)}(\lambda)| \leq 5d_nn^{-1/2}.
       \label{eqn:bded-conv-lemma-appl}
   \end{equation}
   Note that if $|X_1| \leq 2\sqrt{\Xi\log n}$, then $X^\flat_1 = X_1$. Therefore 
    \begin{equation}
        |M^{(k)}(\lambda) - (M^\flat)^{(k)}(\lambda)| \leq \E\left|(X^\flat_1)^{k}e^{\lambda X^\flat_1} - X_1^{k}e^{\lambda X_1}\right| 
    \leq\E|X_1|^{k}e^{C_n|X_1|}\1\{|X_1| \geq 2\sqrt{\Xi \log n}\}.
        \label{eqn:mgf-cutoff-eq-1}
    \end{equation}
    Then, we use the generalized H\"older inequality (Proposition~\ref{prop:generalized-holder}) with exponents $(3, 3, 3)$, and then apply Proposition~\ref{prop:sub-Gaussian-props} to get 
    \begin{align}
        \E\left[|X_1|^{k}e^{C_n|X_1|}\1\{|X_1| \geq 2\sqrt{\Xi \log n}\}\right]
        &\leq\left\{\E|X_1|^{3k}\E[e^{3C_n|X_1|}]
        \P(|X_1| \geq 2\sqrt{\Xi \log n})\right\}^{1/3} \notag \\
        &\leq \left\{8\Xi^{3k/2}\Gamma\left(\frac{3k}{2}+1\right)
        \exp\left(\frac92C_n^2\Xi-2\log n\right)\right\}^{1/3}.
        \label{eqn:mgf-cutoff-eq-2}
    \end{align}
    Now note that since $\alpha < 1/2$, 
    \begin{equation}
        \exp\left(\frac92C_n^2\Xi-2\log n\right)
        \leq \exp\left(\frac92\Xi(\log n)^{2\alpha}-2\log n\right)
        = n^{-2+o(1)}.
        \label{eqn:exp-C_n-B_n}
    \end{equation}
    Therefore when $n$ is large enough, (\ref{eqn:mgf-cutoff-eq-1}), (\ref{eqn:mgf-cutoff-eq-2}), and (\ref{eqn:exp-C_n-B_n}) show that 
    \begin{equation}
        |M^{(k)}(\lambda) - (M^\flat)^{(k)}(\lambda)| \leq d_nn^{-1/2}. 
        \label{eqn:mgf-cutoff-eq-3}
    \end{equation}
    Combine \eqref{eqn:trunc-triangle}, \eqref{eqn:bded-conv-lemma-appl}, and \eqref{eqn:mgf-cutoff-eq-3} to prove the claim.
\end{proof}
\subsubsection{Consequences of Proposition~\ref{prop:sub-gaussian-convergence-lemma}}
We now address the corollaries of Proposition~\ref{prop:sub-gaussian-convergence-lemma}.
\begin{corollary}
\label{cor:cgf-unif-conv}
    Let $k \in \N_0$. Fix $\alpha \in (0, 1/2)$, and set $C_n := (\log n)^\alpha$. Let $P$ be a sub-Gaussian distribution. There exist constants $n_1(\alpha, k , \Xi)$ and $K(k, \Xi)$ such that for all $n \geq n_1$,
   \[
   \sup_{\lambda \in [-C_n, C_n]}|\psi_n^{(k)}
(\lambda) - \psi^{(k)}(\lambda; P)| \leq K(k, \Xi) d_n^{1+k\Xi}n^{-1/2}
   \]
with probability $1- 2n^{-1}$. 
\end{corollary}
\begin{proof}
    Work on the event 
    \[
    \left\{\sup_{\lambda \in [-C_n, C_n]} |M_n^{(k)}(\lambda) - M^{(k)}(\lambda; P)| \leq 7d_nn^{-1/2}\right\},
    \]
    which occurs with probability at least $1 - 2/n$. Let $\eps > 0$. According to Jensen's inequality, $M(\lambda) = \E[e^{\lambda X}] \geq e^{\lambda \E X} = 1$, and $M_n(\lambda) = \frac1n \sum_{i=1}^n e^{\lambda (X_i - \bar X)} \geq e^{\lambda \frac1n\sum_{i=1}^n (X_i - \bar X)} = 1$. Thus, by the Cauchy-Schwarz inequality and then the sub-Gaussianity of $X$,
    \begin{align}
        M^{(k)}(\lambda) &= \E[X^ke^{\lambda X}] \leq \{\E X^{2k} M(2\lambda)\}^{1/2} \notag\\ &\leq \{\E X^{2k} \exp(2\lambda^2\Xi)\}^{1/2} \leq  \{\E X^{2k} \exp(2C_n^2\Xi)\}^{1/2}\notag \\
        &= \{\E X^{2k} \exp(2(\log n)^{2\alpha}\Xi)\}^{1/2}\leq \{\E X^{2k} \exp(2(\log n)^{\alpha + \beta})^{\Xi}\}^{1/2} \notag \\ 
        &\leq \sqrt{2\Xi^{2k}\Gamma(k+1) d_n^{\Xi}} = K(k, \Xi) d_n^{\Xi/2}. 
        \label{eqn:M-k-ineq}
    \end{align}
    where the last step holds when $n$ is large enough, and uses $2\alpha < 1$ so that $(\log n)^{2\alpha - 1} \to 0$. To bound the empirical analogue, we write 
    \[
    M_n^{(k)}(\lambda) \leq M^{(k)}(\lambda) + |M_n^{(k)}(\lambda) - M^{(k)}(\lambda)| \leq K(k, \Xi)d_n^{\Xi/2} + 7d_nn^{-1/2} \leq K(k, \Xi)d_n^{\Xi/2}. 
    \]
    Now by a direct computation we have 
    \[
    \psi^{(k)}(\lambda) = \frac{M^{(k)}}{M} - \sum_{j = 1}^{k-1} {k - 1 \choose j- 1} \psi^{(j)}(\lambda)\frac{M^{(k - j)}(\lambda)}{M(\lambda)},
    \]
    and similarly
    \[
    \psi^{(k)}_n(\lambda) = \frac{M^{(k)}_n(\lambda)}{M_n(\lambda)} - \sum_{j = 1}^{k-1} {k - 1 \choose j- 1} \psi^{(j)}_n(\lambda)\frac{M^{(k - j)}_n(\lambda)}{M_n(\lambda)}.
    \]
    Since $M(\lambda) \geq 1$, it then follows by (\ref{eqn:M-k-ineq}) that 
    \begin{equation}
        \psi^{(k)}_n(\lambda) \leq K(k, \Xi)d_n^{\Xi}
        \label{eqn:psi-sup-bound}
    \end{equation}
    for a constant $K(k, \Xi)$. 
    The remainder of the proof goes by induction on $k$. \\
    \textit{Base case}: By the mean-value theorem applied to $x \mapsto \log x$, 
    \[
    |\psi_n(\lambda) - \psi(\lambda)| \leq \tau^{-1}|M_n(\lambda) - M(\lambda)|
    \]
    where $\tau \in [M_n(\lambda) \wedge M(\lambda), M_n(\lambda) \vee M(\lambda)]$. But $M_n(\lambda) \wedge M(\lambda) \geq 1$, so $\tau \geq 1$ and actually 
    \begin{equation}
        |\psi_n(\lambda) - \psi(\lambda)| \leq |M_n(\lambda) - M(\lambda)|.
        \label{eqn:cgf-ind-base}
    \end{equation}
    \textit{Inductive step:} Suppose the claim holds for any $j < k$. Then we have 
    \begin{equation}
        |\psi^{(k)}(\lambda) - \psi^{(k)}_n(\lambda)| \\ \leq \left|\frac{M^{(k)}_n(\lambda)}{M_n(\lambda)} - \frac{M^{(k)}(\lambda)}{M(\lambda)}\right| + \sum_{j = 1}^{k-1} {k - 1 \choose j- 1} \left|\psi^{(j)}_n(\lambda)\frac{M^{(k - j)}_n(\lambda)}{M_n(\lambda)} - \psi^{(j)}(\lambda)\frac{M^{(k - j)}(\lambda)}{M(\lambda)}\right|.
        \label{eqn:cgf-ind-est-1}
    \end{equation}
    For the first term, since $M(\lambda) \geq 1$ and $M_n(\lambda) \geq 1$, we have  
    \begin{align}
\left|\frac{M^{(k)}_n(\lambda)}{M_n(\lambda)} - \frac{M^{(k)}(\lambda)}{M(\lambda)}\right| &\leq \left|\frac{M_n^{(k)}(\lambda) - M^{(k)}(\lambda)}{M_n(\lambda)}\right| + \left|\frac{M^{(k)}(\lambda)\{M_n(\lambda) - M(\lambda)\}}{M(\lambda)M_n(\lambda)}\right| \notag \\
&\leq \left|M_n^{(k)}(\lambda) - M^{(k)}(\lambda)\right| + \left|M^{(k)}(\lambda)\{M_n(\lambda) - M(\lambda)\}\right| \label{eqn:cgf-ind-est-2} \\
&\leq 7d_n(1 + K(k, \Xi)d_n^{\Xi})n^{-1/2} \leq K(k, \Xi)d_n^{\Xi + 1}n^{-1/2}
\label{eqn:cor1-mgf-triangle}
    \end{align}
    where the last inequality uses (\ref{eqn:M-k-ineq}) and Proposition~\ref{prop:sub-gaussian-convergence-lemma}. Meanwhile, for each term in the summation, we have 
\begin{equation}
    \left|\psi^{(j)}_n(\lambda)\frac{M^{(k - j)}_n(\lambda)}{M_n(\lambda)} - \psi^{(j)}(\lambda)\frac{M^{(k - j)}(\lambda)}{M(\lambda)}\right|  \leq \left|\psi_n^{(j)} - \psi^{(j)}\right|\left|\frac{M^{(k - j)}_n(\lambda)}{M_n(\lambda)}\right| + |\psi^{(j)}|\left|\frac{M^{(k - j)}_n(\lambda)}{M_n(\lambda)} - \frac{M^{(k - j)}(\lambda)}{M(\lambda)}\right|.
    \label{eqn:cgf-ind-est-3}
\end{equation}
The first term is bounded by $K(k, \Xi)d_n^{j\Xi+1}n^{-1/2}$ by the inductive hypothesis and (\ref{eqn:M-k-ineq}). The second follows by similar reasoning to (\ref{eqn:cor1-mgf-triangle}), by using the estimate (\ref{eqn:psi-sup-bound}). 
\end{proof}
\begin{proof}[Proof of Proposition~\ref{prop:L_n-cgf-bound}]
Recall the identity \eqref{eqn:L-integral-rep}, which holds similarly for $L_n$ because $\psi'_n(0) = 0$. Now since $\psi^{(k)}$ is twice continuously differentiable, we may differentiate under the integral sign to obtain
    \[
    L^{(k)}(\lambda) = 2\int_0^1 (1 - x)x^k \psi^{(k+2)}(\lambda x) \, dx , \quad L_n^{(k)}(\lambda) = 2 \int_0^1(1 - x)x^k \psi^{(k+2)}_n(\lambda x)\, dx.
    \]
    Then it follows that, for any $\lambda \in [-C_n, C_n]$, 
    \begin{align*}
         |L_n^{(k)}(\lambda) - L^{(k)}(\lambda)|&\leq \int_0^1 (1 - x)x^k|\psi^{(k+2)}(\lambda x) - \psi_n^{(k+2)}(\lambda x)| \, dx \\ &\leq 2\sup_{t \in [-C_n, C_n]} |\psi^{(k+2)}(t) - \psi^{(k+2)}_n(t)| \int_0^1 (1 - x)x^k \, dx \\
        &\leq \sup_{t \in [-C_n, C_n]} |\psi^{(k+2)}(t) - \psi^{(k+2)}_n(t)|
    \end{align*}
    where the last step follows because $\int_0^1 (1-x)x^k \, dx = (k+1)^{-1}(k+2)^{-1} \leq 1/2$. The claim follows by combining this with Corollary~\ref{cor:cgf-unif-conv}. 
\end{proof}
\begin{proof}[Proof of Theorem~\ref{prop:second-consistency}]
Define $\tilde \xi_n^2 := \sup_{|\lambda|\leq C_n} L(\lambda)$. We then have 
\begin{align*}
    |\hat \xi^2_n - \xi_\star^2| &\leq |\hat \xi^2_n - \tilde \xi^2_n| + \xi_\star^2 - \tilde \xi_n^2 \\
    &\leq |\hat \xi^2_n - \tilde \xi^2_n| + \sup_{\lambda \in \R} L(\lambda) - \sup_{\lambda \in [-C_n, C_n]} L(\lambda) \\
    &= |\hat \xi^2_n - \tilde \xi^2_n| + \left\{ \sup_{|\lambda| > C_n} L(\lambda) - \sup_{|\lambda| \leq C_n} L(\lambda)\right\} \vee 0 \\
    &= |\hat \xi^2_n - \tilde \xi_n^2| + \delta(C_n) \vee 0.
\end{align*}
    Since $L$ and $L_n$ are continuous, there is $\tilde \lambda_n \in [-C_n, C_n]$ such that $\tilde \xi_n = L(\tilde \lambda_n)$, and $\lambda^\star_n \in [-C_n, C_n]$ such that $\hat \xi_n^2 = L_n(\lambda^\star_n)$. Then 
    \[
    \hat \xi_n^2 - \tilde \xi_n^2 \leq L_n(\lambda^\star_n) - L(\lambda^\star_n) \leq \sup_{\lambda \in [-C_n, C_n]} |L_n(\lambda) - L(\lambda)|
    \]
    and 
    \[
    \hat \xi_n^2 - \tilde \xi_n^2 \geq L_n(\tilde \lambda_n) - L(\tilde \lambda_n) \geq -\sup_{\lambda \in [-C_n, C_n]}|L_n(\lambda) - L(\lambda)| 
    \]
    so that $|\hat \xi_n^2 - \tilde \xi^2_n| \leq \sup_{|\lambda| \leq C_n} |L_n(\lambda) - L(\lambda)|$. Therefore 
    \begin{equation}
        |\hat \xi^2_n - \xi_\star^2| \leq \sup_{|\lambda| \leq C_n} |L_n(\lambda) - L(\lambda)| + \delta(C_n) \vee 0.
        \label{eqn:L-delta-bound}
    \end{equation}
    (a) By Proposition~\ref{prop:L_n-cgf-bound}(b), we have $\sup_{|\lambda| \leq C_n} |L_n(\lambda) - L(\lambda)| = o_p(1)$ for any $\eps > 0$. Since $C_n \to \infty$, $\delta(C_n) \vee 0 \to 0$. \\
    (b) Whenever $n$ is large enough that $C_n > C_0$, we have $\delta(C_n) \leq 0$, so the claim follows from (\ref{eqn:L-delta-bound}). \\
    (c) Make $r > 0$ so small that $\sup_{|\lambda| > C_0} L(\lambda) < \sup_{|\lambda| \leq C_0} L(\lambda) - r$. It follows that $\xi_\star^2 = L(\lambda^\star)$ for some $\lambda^\star \in [-C_0, C_0]$. Then, since $L_n(\lambda^\star) \xrightarrow{\mathrm{a.s.}} L(\lambda^\star)$, if $n$ is large enough, we have on one hand
    \[
    \sup_{|\lambda| \leq C_0} L_n(\lambda) \geq L_n(\lambda^\star) \geq L(\lambda^\star) - \frac r2
    \]
    and on the other
    \[
    \sup_{C_0 < |\lambda| \leq C_n} L_n(\lambda) \leq \sup_{|\lambda| \leq C_n} |L_n(\lambda) - L(\lambda)| + \sup_{C_0 < |\lambda|} L(\lambda) \leq o_p(1) + \xi_\star^2 - r  \leq \xi_\star^2 - \frac{r}{2}
    \]
    so that
    \[
    \sup_{|\lambda| \leq C} L_n(\lambda) > \sup_{C < |\lambda| \leq C_n} L_n(\lambda).
    \]
    It follows that none of $(C, C_n]$ can be a maximizer. Thus with high probability, the maximizers of $L_n$ contained in $[-C_0, C_0]$, so it actually holds that 
    \[
    |\hat \xi^2_n - \tilde \xi^2_n| \leq \sup_{|\lambda| \leq C_0} |L_n(\lambda) - L(\lambda)|.
    \]    Proposition~\ref{prop:L_n-uniform-rootn} then applies to yield $\sqrt n(\hat \xi_n^2 - \tilde \xi^2_n) = O_p(1)$. Since the hypothesis of (c) clearly implies (b), it holds again that $\delta(C_n) \vee 0 = 0$ for large enough $n$, and the claim follows from (\ref{eqn:L-delta-bound}).
\end{proof}

\subsubsection{Proof of Proposition~\ref{prop:ptwise-asymptotic-normality}}
In this section, we show the asymptotic normality of $\hat \xi_n^2$. We shall first need a lemma: 
\begin{proposition}
\label{prop:maximizer-root-n}
    Suppose that $\limsup_{|\lambda| \to \infty} L(\lambda) < \xi_\star^2$, that $L$ is uniquely maximized at $\lambda^\star$, and $L''(\lambda^\star) < 0$. Then, if $\lambda^\star_n$ maximizes $L_n$ on $[-C_n, C_n]$, we have that $\lambda^\star_n - \lambda^\star = O_p(n^{-1/2})$. 
\end{proposition}
\begin{proof}
    Fix $\eps > 0$. Since $\limsup_{|\lambda| \to \infty} L(\lambda) < \xi_\star^2$, we may choose $C, \delta_0 > 0$ such that 
    \[
    \sup_{|\lambda| > C} L(\lambda) < \xi_\star^2 - 2\delta_0.
    \]
    Furthermore, 
    the set \[
    A_\eps := [-C, C] - \{\lambda: |\lambda - \lambda^\star| < \eps\}
    \]
    is compact and contains no maximizer of $L$, so $L$ attains some maximum $\xi_A^2 < \xi_\star^2$. Put $\delta := (\xi_\star^2 - \xi_A^2)/2$, and define $E_n := \{|\lambda^\star_n - \lambda^\star| > \eps\}$. On $E_n$, if $\lambda^\star_n \in [-C, C]$ then $\lambda^\star_n \in A_\eps$ so $L(\lambda^\star_n) \leq \xi_A^2$. If $\lambda^\star_n \not\in [-C, C]$ then $L(\lambda^\star_n) \leq \sup_{|\lambda| > C} L(\lambda) \leq \xi_A^2$. Thus we have $L(\lambda^\star_n) \leq \xi_A^2$ whenever $E_n$ holds. It then follows that 
    \[
    \hat \xi_n^2 = L_n(\lambda^\star_n) \leq L(\lambda^\star) + \sup_{|\lambda| \leq C_n} |L_n(\lambda) - L(\lambda)|
    \]
    and 
    \[
    \hat \xi_n^2 = L_n(\lambda^\star_n) \geq L_n(\lambda^\star) \geq L(\lambda^\star) - \sup_{|\lambda| \leq C_n} |L_n(\lambda) - L(\lambda)| 
    \]
    so that 
    \[
    \xi_\star^2 - \xi_A^2 \leq 2\sup_{|\lambda| \leq C_n} |L_n(\lambda) - L(\lambda)|.
    \]
We thus have $E_n \subset \{\sup_{|\lambda| \leq C_n} |L_n(\lambda) - L(\lambda)| \geq \delta\}$. The probability of the second event vanishes according to Proposition~\ref{prop:L_n-cgf-bound}(b), so we conclude $\lambda^\star_n - \lambda^\star = o_p(1)$.

    We now deduce the rate of convergence.
    By the mean-value theorem, there is $\tilde \lambda_n$ between $\lambda^\star_n$ and $\lambda^\star$ such that 
    \[
    0 = L_n'(\lambda^\star_n) = L'_n(\lambda^\star) + (\lambda^\star_n - \lambda^\star)L''_n(\tilde \lambda_n),
    \]
    and by the triangle inequality, we have 
    \[
    |L''_n(\tilde \lambda_n) - L''(\lambda^\star)| \leq |L''_n(\tilde \lambda_n) - L''(\tilde \lambda_n)| + |L''(\tilde \lambda_n) - L''(\lambda^\star)|. 
    \]
    The first term is $o_p(1)$ by Proposition~\ref{prop:L_n-cgf-bound}(a). Since $\tilde \lambda_n - \lambda^\star= o_p(1)$, we may restrict $L''$ to a compact interval containing all of the $\tilde \lambda_n$'s and $\lambda^\star$. On this interval, $L''$ must be uniformly continuous, so it follows that the second term is $o_p(1)$ as well. Thus we have $L''_n(\tilde \lambda_n) - L''(\lambda^\star) = o_p(1)$. Since $L''(\lambda^\star) < 0$, provided $n$ be large enough, $L''_n(\tilde \lambda_n)$ will be bounded away from zero. Therefore, we have 
    \[
    |\lambda^\star-\lambda^\star_n|  = \left|\frac{L'_n(\lambda^\star)}{L''_n(\tilde \lambda_n)}\right| = \left|\frac{L'_n(\lambda^\star) - L'(\lambda^\star)}{L''_n(\tilde \lambda_n)}\right|  \leq \sup_{\lambda^\star \in \lambda^\star} \frac{|L'_n(\lambda^\star) - L'(\lambda^\star)|}{|L''_n(\tilde \lambda_n)|}.
    \]
    Proposition~\ref{prop:L_n-cgf-bound}(a) says that the convergence of $L'_n$ is uniform and $O_p(n^{-1/2})$ over the relevant interval. This proves the claim.  
\end{proof}
We are now ready to prove a central limit theorem for $\hat \xi_n^2$. 
\begin{proof}[Proof of Proposition~\ref{prop:ptwise-asymptotic-normality}]
    Let $\lambda^\star$ be the unique maximizer of $L$. If $\lambda^\star = 0$, we just notice that $L_n(0)$ is the sample variance and $L(0)$ is $\sigma^2(P)$, so the claim is easy to prove. Thus, we assume $\lambda^\star \ne 0$. Since there is $C_0$ such that $\delta(C_0) < 0$, we have that
    \[
    \sup_{|\lambda|  \geq C_0 \vee \lambda^\star} L(\lambda) < \sup_{|\lambda| \leq C_0 \vee \lambda^\star} L(\lambda) = \xi_\star^2
    \]
    and therefore the hypotheses of Proposition~\ref{prop:maximizer-root-n} hold. 
     Then, by Taylor's theorem there is $\tilde \lambda_n$ between $\lambda^\star_n$ and $\lambda^\star$ such that 
    \begin{align*}
         L_n(\lambda^\star_n) - L_n(\lambda^\star) &= (\lambda^\star_n - \lambda^\star) L'_n(\lambda^\star) + \frac{(\lambda^\star_n - \lambda^\star)^2}{2}L''_n(\tilde \lambda_n) \\
         &= O_p(n^{-1/2})O_p(n^{-1/2}) + O_p(n^{-1})O_p(1) = O_p(n^{-1})
    \end{align*}
    by Proposition~\ref{prop:L_n-cgf-bound}(a). 
    Therefore, it holds that 
    \begin{align*}
        \sqrt n(\hat \xi_n^2 - \xi_\star^2) &= \sqrt n(L_n(\lambda^\star) - L(\lambda^\star))  + \sqrt n(L_n(\lambda^\star_n) - L_n(\lambda^\star)) \\ &=  \sqrt n(L_n(\lambda^\star) - L(\lambda^\star)) + O_p(n^{-1/2}) \\
        &= \sqrt n(L_n(\lambda^\star) - L(\lambda^\star)) + o_p(1).
    \end{align*}
    Now let $\lambda$ be fixed. By the central limit theorem, we have 
    \[
    \sqrt n\begin{pmatrix}
        \tilde M_n(\lambda) - M(\lambda) \\ \lambda\bar X
    \end{pmatrix} \xrightarrow{d} N\left(0, \Sigma\right)
    \]
    where 
    \[
    \Sigma = \begin{pmatrix}
        M(2\lambda) - M(\lambda)^2 & \lambda M'(\lambda) \\ \lambda M'(\lambda) & \lambda^2\sigma^2
    \end{pmatrix}.
    \]
    Now since 
    \[
    L_n(\lambda) = \frac{2}{\lambda^2}\left(\log \tilde M_n(\lambda) - \lambda\bar X\right)
    \]
    we have by the Delta method that 
    \[
    \sqrt n (L_n(\lambda) - L(\lambda)) \xrightarrow{d} N(0, V(\lambda))
    \]
    where 
    \[
    V(\lambda) = \frac{4}{\lambda^4}\begin{pmatrix}
        1/M(\lambda) \\ -1
    \end{pmatrix}^\top \Sigma \begin{pmatrix}
        1/M(\lambda) \\ -1
    \end{pmatrix} = \frac{4}{\lambda^4} \left(\frac{M(2\lambda)}{M(\lambda)^2} - 2\lambda \psi'(\lambda)+  \lambda^2\sigma^2 - 1\right).
    \]
\end{proof}
\subsection{Minimax Upper Bounds}
\label{sec:app-minimax-upper}
\begin{proof}[Proof of Proposition~\ref{prop:maximizer-uniform-bound}]
By assumption, we have $\delta_P(C_0) < -\delta_0$, so 
\[
\sup_{|\lambda| \geq C_0} L(\lambda) \leq \sup_{|\lambda| \leq C_0} L(\lambda) - \delta_0.
\]
Therefore, none of the maximizers of $L$ can be outside of the interval $[-C_0, C_0]$. 

As for its empirical analogue $L_n$, first let $\lambda^\star \in \lambda^\star(P)$. By Proposition~\ref{prop:L_n-cgf-bound}(b), we may make $n_0$ large enough such that with probability $1 - 2/n$, 
    \begin{align*}
        \sup_{|\lambda| \leq C_0} L_n(\lambda) \geq L_n(\lambda^\star) &\geq L(\lambda^\star) - \sup_{|\lambda| \leq C_0} |L_n(\lambda) - L(\lambda)| \\ 
        &\geq L(\lambda^\star) - \sup_{|\lambda| \leq C_n} |L_n(\lambda) - L(\lambda)|> \xi_\star^2 - \frac{\delta_0}{2}.
    \end{align*}
    On this event, we have that
    \[
    \sup_{C_0 < |\lambda| \leq C_n} L_n(\lambda) \leq \sup_{|\lambda| \leq C_n} |L_n(\lambda) - L(\lambda)| + \sup_{|\lambda| > C_0} L(\lambda) \leq \frac{\delta_0}{2} + \xi_\star^2 - \delta_0 \leq \xi_\star^2 - \frac{\delta_0}{2},
    \]
and so 
    \[
    \sup_{|\lambda| \leq C_0} L_n(\lambda) > \sup_{C_0 < |\lambda| \leq C_n} L_n(\lambda),
    \]
    which means that $\lambda^\star_n \subset [-C_0, C_0]$. This proves the claim. 
\end{proof}

\subsubsection{Uniform Donsker theorems}
\label{sec:uniform-donsker}
The following definitions are due to \cite{vdv&w2012weak}. Let $\mathbb D$ be a normed vector space, and let $\mathcal B(\mathbb D)$ denote the Borel $\sigma$-algebra on $\mathbb D$ (generated by the open sets under the norm topology of $\mathbb D$). For a probability space $(\Omega, \mathcal A, P)$, we define the \textbf{outer expectation} for any function $X: \Omega \to \mathbb R$ (even non-measurable) as 
\[
\E^\star_P X := \inf \left\{\E_PY: \text{$Y$ is $(\mathcal A, \mathcal B(\R))$-measurable and $Y \geq X$}\right\}.
\]
If $(\Omega_n, \mathcal A_n, P_n)$ are a sequence of probability spaces and $X_n : \Omega_n \to \mathbb D$ are arbitrary, then $X_n$ \textbf{converges weakly} to a Borel probability measure $P$ on $\mathbb D$ if as $n \to \infty$, 
\[
\E^\star_{P_n} f(X_n) \to \E_P f(X)
\]
for any $f: \mathbb D \to \R$ continuous and bounded. If $\mathbb D$ is separable, then by Theorem 1.12.1 of \cite{vdv&w2012weak}, weak convergence is equivalent to 
\[
\lim_{n \to \infty} \sup_{h \in \mathrm{BL}_1(\mathbb D)} |\E^\star_Ph(X_n) - \E_Ph(X)| \to 0,
\]
where 
\[
\mathrm{BL}_1(\mathbb D) = \{h: \mathbb D \to \R: \text{$h$ is 1-Lipschitz and $\|h\|_{\mathbb D} \leq 1$}\}.
\]
Let $\mathcal F$ be a family of functions on $\R$. For any function $z: \mathcal F \to \R$, we define 
\[
\|z\|_{\mathcal F} := \sup_{f \in \mathcal F} |z(f)|
\]
and make 
\[
\ell^\infty(\mathcal F) := \left\{z: \mathcal F \to \R : \|z\|_{\mathcal F} <\infty\right\}
\]
into a Banach space by endowing it with the norm $\| \cdot \|_{\mathcal F}$. 

 Let $X_1, \dots, X_n$ be a random sample. The $X_i$'s induce an empirical measure $\mathbb P_n := \frac1n \sum_{i=1}^n \delta_{X_i}$.
 We say that $\mathcal F$ is \textbf{$P$-Donsker} if the empirical process if the empirical process $\mathbb G_{n, P} := \sqrt n(\mathbb P_n - P)$ converges weakly in $\ell^\infty (\mathcal F)$ to a tight Borel measurable element $\mathbb G_P$ of $\ell^\infty(\mathcal F)$. That is, 
\begin{equation}
    \lim_{n \to \infty} \sup_{h \in \text{BL}_1} |\E^\star_Ph(\mathbb G_{n, P}) - \E h(\mathbb G_P)| = 0.
    \label{eqn:ptwise-donsker}
\end{equation}
If this convergence is uniform over $\mathcal P$, then $\mathcal F$ is \textbf{Donsker uniformly over $\mathcal P$}:
\[
\lim_{n \to \infty} \sup_{P \in \mathcal P} \sup_{h \in \text{BL}_1} |\E^*_Ph(\mathbb G_{n, P}) - \E h(\mathbb G_P)| = 0.
\]
If the mean-zero Gaussian process $\mathbb G_P$ with covariance function $(f, g) \mapsto \E_{P}f(X)g(X) - \E_Pf(X)\E_Pg(X)$ has a version that is a tight Borel measurable element of $\ell^\infty(\mathcal F)$, then $\mathcal F$ is \textbf{$P$-pre-Gaussian.} If $\mathcal F$ is $P$-pre-Gaussian for each $P \in \mathcal P$, and it holds that
\[
\sup_{P \in \mathcal P} \E_P \|\mathbb G_P\|_{\mathcal F} < \infty \quad \text{and} \quad \lim_{\delta \to 0} \sup_{P \in \mathcal P} \E_P \sup_{\rho_P(f-g) < \delta} |\mathbb G_P(f) - \mathbb G_P(g)| = 0
\]
where $\rho_P(f) = \left(\E_P\{f(X) - \E_Pf(X)\}^2\right)^{1/2}$, then $\mathcal F$ is \textbf{pre-Gaussian uniformly over $\mathcal P$}. 

In order to prove Proposition~\ref{prop:mgf-uniform-donsker}, we will use a modification of Theorem 2.8.4. in \cite{vdv&w2012weak}, which establishes a uniform-in-$\mathcal P$ Donsker property of $\mathcal F$. The proof actually establishes the following stronger fact, which by Markov's inequality then yields the usual Donsker property stated in the original theorem. 
\begin{proposition}[\cite{vdv&w2012weak}]
\label{prop:uniform-donkser}
    Let $\mathcal P$ be a class of distributions, and $\mathcal F$ be a class of measurable functions with envelope function $F$ that satisfies 
    \[
    \lim_{R \to \infty} \sup_{P \in \mathcal P} \E_P(F^2(X)\1\{F(X) > R\}) = 0.
    \]
    Furthermore, assume that 
    \[
    \int_0^\infty \sup_{P \in \mathcal P} \sqrt{\log N_{[\,]}(\eps \|F\|_{P, 2}, \mathcal F, L_2(P))} \, d\eps < \infty
    \]
    where $N_{[\,]}$ denotes bracketing number. Then 
    \begin{equation}
        \sup_{n \in \N}\sup_{P \in \mathcal P} \E_P^\star \|\mathbb G_n\|_{\mathcal F} < \infty. 
        \label{eqn:outer-uniform-emp-proc-bound}
    \end{equation}
\end{proposition}
We may promote the outer expectation in \eqref{eqn:outer-uniform-emp-proc-bound} to an ordinary expectation by restricting $\mathcal F$ to $[-C, C] \cap \Q$, making $\ell^\infty$ separable and hence recovering measurability. 
\begin{proposition}
\label{prop:mgf-uniform-donsker}
    $\mathcal F_k$ is Donsker and pre-Gaussian, both uniformly over $\mathcal P(\Xi)$. There exists $K = K(C_0, k, \Xi) > 0$ such that 
    \[
    \sup_{P \in \mathcal P(\Xi) }\E_P \left[\sup_{|\lambda| \leq C_0} |M_n^{(k)}(\lambda) - M^{(k)}(\lambda)|\right] \leq \frac{K}{\sqrt n}. 
\]
\end{proposition}
\begin{proof}[Proof of Proposition~\ref{prop:mgf-uniform-donsker}]
    An envelope function for $\mathcal F_k$ is $F(x) := (1 + |x|)^{k+1}e^{C|x|}$. By the Cauchy-Schwarz inequality and then Proposition~\ref{prop:sub-Gaussian-props},
    \[
    \sup_{P \in \mathcal P} \E_PF(X) \leq \{\E_P(1+|X|)^{2k+2} \E_P e^{2C|X|}\}^{1/2} \leq K(k, C, \Xi) < \infty
    \]
    by sub-Gaussianity. 
Similarly, by the generalized H\"older inequality and Markov's inequality and then Proposition~\ref{prop:sub-Gaussian-props}, 
    \begin{align*}
        \E_P(F^2(X) \1\{F(X) > R\}) &= \E_P(1 + |X|)^{2k+2}e^{2C|X|} \1\{F(X) > R\} \\
        &\leq \{\E_P(1 + |X|)^{6k+6}M(6C)\mathbb P(F(X) > R)\}^{1/3} \\
        &\leq R^{-1/3} \{\E_P(1 + |X|)^{6k+6}M(6C)\E_PF(X)\}^{1/3} \\ 
        &\leq R^{-1/3} K_1(k, C, \Xi) \to 0
    \end{align*}
    as $R \to \infty$. Therefore the envelope condition holds. Moving onto the bracketing entropy, let $P \in \mathcal P(\Xi)$ and $\eps > 0$. Now we have, for any $\lambda, \lambda' \in [-C, C]$,
\[
|x^ke^{\lambda x} - x^ke^{\lambda'x}| \leq |x^{k+1}e^{\tilde \lambda x}| |\lambda - \lambda'| \leq |x^{k+1}e^{C|x|}||\lambda - \lambda'| \leq |(1 + |x|)^{k+1}e^{C|x|}||\lambda - \lambda'|,
\]
where $\tilde \lambda$ is prescribed by the mean-value theorem. Therefore we conclude that 
\[
\left\|m_\lambda^{(k)} - m^{(k)}_{\lambda'}\right\|_{P, 2} \leq \|F\|_{P, 2}|\lambda - \lambda'|.
\]
Set $\delta := \|F\|_{P, 2}^{-1}\eps$, and let $\Pi_\delta$ be a $\delta$-net over $[-C, C]$ such that $|\Pi_\delta| \leq \lceil2C\delta^{-1}\rceil = \lceil2C\|F\|_{P, 2}\eps^{-1}\rceil$. Now for each $\pi \in \Pi_\delta$, we define 
\[
\ell_\pi (x) := m^{(k)}_\pi(x) - \delta\,|x|^{k+1}e^{C|x|} , \quad u_{\pi}(x) := m^{(k)}_\pi(x) + \delta\,|x|^{k+1}e^{C|x|}.
\]
Then, whenever $\lambda \in [\pi - \delta, \pi + \delta]$, we have $m_\lambda^{(k)} \in [\ell_\pi, u_\pi]$, and furthermore 
\[
\left\|u_\pi - \ell_\pi\right\|_{P, 2} \leq 2\delta \|F\|_{P, 2} \leq 2\eps.
\]
We thus have 
\[
\sup_{P \in \mathcal P}N_{[\,]}\!\left(\eps\|F\|_{P, 2},\, \mathcal F_k,\, L_2(P)\right) 
    \leq
    \left\lceil \frac{4C}{\eps} \right\rceil.
\]
Finally, if $\eps > 4C$, then $N_{[\,]}\!\left(\eps,\mathcal F_k, L_2(P)\right)  = 1$, so we have 
\[
\int_0^\infty \sup_{P \in \mathcal P} \sqrt{\log N_{[\,]}\!\left(\eps\|F\|_{P, 2}, \mathcal F_k, L_2(P)\right)} \leq \int_0^{4C} \sqrt{\log \left\lceil \frac{4C }{\eps} \right\rceil}\, d\eps < \infty.
\]
We conclude via Proposition~\ref{prop:uniform-donkser}, setting $K := \sup_{n \in \N} \sup_{P \in \mathcal P} \E^\star_P \|\mathbb G_n\|_{\mathcal F}$. 
\end{proof}

A small technical modification must be made to transfer this property to $M_n(\lambda)$.
\begin{proposition}
\label{prop:centered-mgf-uniform-donsker}
Fix $k \in \N_0$. There exists a constant $K = K(k, \Xi, C_0)$ such that \[
\E_P \left[\sup_{|\lambda| \leq C_0} |M_n^{(k)}(\lambda) - M^{(k)}(\lambda; P)| \right] \leq \frac{C}{\sqrt n}
\]
uniformly over $P \in \mathcal P(\Xi)$.
\end{proposition}
\begin{proof}
    Recalling that $\tilde M_n(\lambda) := \frac1n\sum_{i=1}^n e^{\lambda X_i}$, we have by the triangle inequality that  
    \begin{align*}
         |M_n^{(k)}(\lambda) - M^{(k)}(\lambda)|&\leq |M_n^{(k)}(\lambda) - \tilde M_n^{(k)}(\lambda)| + |\tilde M_n^{(k)}(\lambda) - M^{(k)}(\lambda)|.
    \end{align*}
    by Proposition~\ref{prop:mgf-uniform-donsker}, we may supremize over $K$ and take expectations to make the second term $O(n^{-1/2})$, 
    uniformly over $P \in \mathcal P(\Xi)$. To address the first term, observe that $M_n(\lambda) = e^{-\lambda \bar X} \tilde M_n(\lambda)$ so by direct differentiation, we have 
    \[
    M_n^{(k)}(\lambda) - \tilde M_n^{(k)}(\lambda) = (e^{-\lambda \bar X} - 1)\tilde M_n^{(k)}(\lambda) + e^{-\lambda \bar X} \sum_{j=1}^k {k \choose j}(-\bar X)^j\tilde M_n^{(k - j)}(\lambda).
    \]
    For the first term, use the mean-value theorem to see that $|e^{-\lambda \bar X} - 1| \leq C |\bar X|e^{C|\bar X|}$. Therefore the generalized H\"older inequality gives  
    \begin{align*}
    \E_P\left[\sup_{\lambda \in K} (e^{-\lambda \bar X} - 1)\tilde M_n^{(k)}(\lambda) \right]&\leq C\E_P\left[|\bar X|e^{C \bar X} \times \left(\frac1n\sum_{i=1}^n|X_i|^{k}e^{C|X_i|}\right)\right] \\
        &\leq C\E_P\left[|\bar X|^3\right]^{1/3}\E_P\left[ e^{3C\bar X}\right]^{1/3} \E_P\left[\left(\frac1n\sum_{i=1}^n|X_i|^{k}e^{C|X_i|}\right)^{3}\right]^{1/3}.
    \end{align*}
    Since $\xi_\star^2(\bar X) \leq \Xi/n$, we see that $\E_P [|\bar X|^j] \leq K_1(\Xi)n^{-j/2}$ for some constant $K_1(\Xi)$. Since the other two factors are bounded by a function of $C_0$ and $\Xi$, we conclude that 
    \[
    \E_P\left[\sup_{\lambda \in K} (e^{-\lambda \bar X} - 1)\tilde M_n^{(k)}(\lambda) \right] \leq \frac{K(k, \Xi, C_0)}{\sqrt n}. 
    \]
    The other terms can be similarly controlled by H\"older's inequality and sub-Gaussianity. Conclude that they are also bounded by a constant depending on $k, \Xi, C_0$ times $1/\sqrt n$, completing the proof. 
\end{proof}
Having established the uniform estimate we need, the remainder of the proof closely mirrors Corollary~\ref{cor:cgf-unif-conv}, but is in expectation. To make it easier to read, we will reproduce some of the estimates used. 
\begin{proof}[Proof of Proposition~\ref{prop:L_n-uniform-rootn}]
    We claim that 
    \begin{equation}
        \sup_{|\lambda| \leq C_0} |\psi_n^{(k)}(\lambda) - \psi^{(k)}(\lambda)| \leq K(k, C_0, \Xi) \underbrace{\max_{0 \leq r \leq 2k} \sup_{|\lambda| \leq C_0} |M_n^{(r)}(\lambda) - M^{(r)}(\lambda)|}_{=: \Delta_n(k, C_0)},
        \label{eqn:psi-unif-est}
    \end{equation}
    which by Proposition~\ref{prop:centered-mgf-uniform-donsker} and \eqref{eqn:L-psi2-bound} prove the claim. To that end, we proceed by induction. 
    The base case is trivial in view of \eqref{eqn:cgf-ind-base}. For the inductive step, recall \eqref{eqn:cgf-ind-est-1}, that 
    \[
    |\psi^{(k)}(\lambda) - \psi^{(k)}_n(\lambda)| \leq \left|\frac{M^{(k)}_n(\lambda)}{M_n(\lambda)} - \frac{M^{(k)}(\lambda)}{M(\lambda)}\right| + \sum_{j = 1}^{k-1} {k - 1 \choose j- 1} \left|\psi^{(j)}_n(\lambda)\frac{M^{(k - j)}_n(\lambda)}{M_n(\lambda)} - \psi^{(j)}(\lambda)\frac{M^{(k - j)}(\lambda)}{M(\lambda)}\right|.
    \]
    The first term is controlled by \eqref{eqn:cgf-ind-est-2} and Proposition~\ref{prop:centered-mgf-uniform-donsker}, since $M^{(k)}$ is now uniformly controlled by sub-Gaussianity: 
    \[
    M^{(k)} (\lambda) = \E_P[X^k e^{\lambda X}] \leq \E|X|^{2k} M(2C) \leq K(k, \Xi). 
    \]
    For the other terms of \eqref{eqn:cgf-ind-est-1}, we recall that 
    \[
\left|\psi^{(j)}_n(\lambda)\frac{M^{(k - j)}_n(\lambda)}{M_n(\lambda)} - \psi^{(j)}(\lambda)\frac{M^{(k - j)}(\lambda)}{M(\lambda)}\right|  \leq \left|\psi_n^{(j)} - \psi^{(j)}\right|\left|\frac{M^{(k - j)}_n(\lambda)}{M_n(\lambda)}\right| + |\psi^{(j)}|\left|\frac{M^{(k - j)}_n(\lambda)}{M_n(\lambda)} - \frac{M^{(k - j)}(\lambda)}{M(\lambda)}\right|.
    \]
    We now split into two cases. When $\Delta_n(k, C_0) \leq 1$, we have that 
    \[
   \left| \frac{M_n^{(k-j)}(\lambda)}{M_n(\lambda)}\right|\leq |M_n^{(k-j)}(\lambda)| \leq 1 + K(k, \Xi),
    \]
    so the inductive hypothesis treats the first term. The second term can be handled similarly by bounding $|\psi^{(j)}| \leq K(j, \Xi)$ and using \eqref{eqn:cgf-ind-est-2} again. In the case $\Delta_n(k, C_0) > 1$, we note that $\psi_n^{(k)}$ is a polynomial in $M_n^{(r)}/M_n$ of degree $k$, we see that 
    \begin{equation}
        \sup_{|\lambda| \leq C_0} |\psi_n^{(k)}(\lambda)|  \leq K(k)\left(\max_{0 \leq r \leq 2k}\sup_{|\lambda| \leq C_0} M^{(r)}_n(\lambda)\right)^{1/2},
        \label{eqn:ecgf-unif-norm-est}
    \end{equation}
    so use the fact that $\Delta_n > 1$ to absorb constants, obtaining 
    \[
    \sup_{|\lambda| \leq C_0} |\psi_n^{(k)}(\lambda) - \psi^{(k)}(\lambda)| \leq \sup_{|\lambda| \leq C_0} |\psi^{(k)}(\lambda)| + \sup_{|\lambda| \leq C_0} |\psi_n^{(k)}(\lambda)| \leq K(k, \Xi) \Delta_n(k, C_0).
    \]
    This completes the induction and proves the claim. 
\end{proof}
\begin{proposition}
\label{prop:hat-xi-sup-est-crude}
    If $P$ is sub-Gaussian, then for $m \geq 1$, 
    \[
    \E_P [(\hat \xi_n^2)^{m}] \leq K_1 (\log n)^{K_2}
    \]
    where $K_1, K_2$ are constants depending only on $\xi_\star^2(P)$ and $m$.   
\end{proposition}
\begin{proof}
    Write $|X|_{(n)} := \max_{i \in [n]} |X_i|$. First we estimate $L_n'$. By using  
    $M^{(k)}_n(\lambda) \leq 8|X|_{(n)}^3 M_n(\lambda)$, we see that 
    \[
    |\psi^{(3)}_n| \leq  \frac{|M^{(3)}_nM^2_n|+ 3|M'_nM''_n M_n| + 2|M'_n|^3}{|M_n|^3} \leq 136|X|_{(n)}^3 ,
    \]
    so we differentiate under the integral sign of \eqref{eqn:L-integral-rep} to see that 
    \[
    |L_n'(\lambda)| \leq 2\int_0^1 x(1 - x) |\psi'''(\lambda x)|\, dx  \leq K_1|X|_{(n)}^3 
    \]
    where $K_1 > 0$ is a universal constant. 
    Let $|\lambda| \leq C_n$. By the Fundamental theorem,
    \[
    |L_n(\lambda)| \leq | \hat\sigma^2_n| + \int_0^{C_n} |L'_n(t)| \, dt \leq |\hat\sigma^2_n| + K_1 C_n |X|_{(n)}^3.
    \]
    Since the above inequality is independent of $\lambda$, it controls $\hat \xi_n^2$ too. Thus we use the  inequality $(a + b)^m \leq 2^{m-1}(a^m + b^m)$ and then take expectations to obtain  
    \[
    \E_P[(\hat \xi_n^2)^m]\leq K_1(\E_P(\hat\sigma^{2m}_n )+  C_n \E_P |X|_{(n)}^{3m}),
    \]
    where we have increased constants as needed. Since $\hat \sigma_n^2 \leq \frac12 (X_{(n)} - X_{(1)})^2 \leq 2|X|_{(n)}^2$, we use the layercake representation to see that the first term is bounded by a constant depending on $\Xi$ times a polylog.
\end{proof}
\begin{proof}[Proof of Theorem~\ref{thm:minimax-upper-bound}]
    (a) According to Proposition~\ref{prop:L_n-cgf-bound}, the event 
    \begin{equation}
        E = \left\{\sup_{|\lambda| \leq C_n} |L_n(\lambda) - L(\lambda)| \leq K(\Xi)d_n^{1+k\Xi}n^{-1/2}\right\}
        \label{eqn:good-event-1}
    \end{equation}
    occurs with probability at least $1-2/n$, uniformly over $\mathcal P(\Xi)$. First we control the expectation over the complement of \eqref{eqn:good-event-1}. By H\"older's inequality, Proposition~\ref{prop:hat-xi-sup-est-crude}, and the fact that $\xi^2_\star(P) \leq \Xi$, we have constants $K_1, K_2 > 0$ depending only on $\Xi > 0$ such that 
    \begin{equation}
        \E_P [|\hat \xi_n^2 - \xi_\star^2| \1_{E^c}] \leq \E_P[|\hat \xi_n^2 - \xi_\star^2|^4]^{1/4} \P(E^c)^{3/4} \leq K_1n^{-3/4}\log(n)^{K_2} \leq K(\Xi)n^{-1/2}
        \label{eqn:hat-xi-error-bad-expectation}
    \end{equation}
    provided $n$ be large enough. 
    On \eqref{eqn:good-event-1}, we recall the estimate \eqref{eqn:xi-hat-n-est}, so 
    \begin{align*}
        \E_P [|\hat \xi_n^2 - \xi_\star^2|\1_{E}] &\leq \E_P \left[\sup_{|\lambda| \leq C_n} |L_n(\lambda) - L(\lambda)|\times \1_E\right] + \delta_P(C_n) \vee 0 \\
        &\leq K(\Xi)d_n^{1 + \Xi}n^{-1/2} + \delta_P(C_n) \vee 0 \leq K(\Xi, \alpha, \eps) n^{-\frac12 + \eps} + r(C_n)_+. 
    \end{align*}
    Combine the above two displays to conclude. \\
    (b) According to Proposition~\ref{prop:maximizer-uniform-bound}, there is $n_0 > 0$ such that whenever $n > n_0$, the event 
    \begin{equation}
        E_1 = \{\lambda^\star_n \cup \lambda^\star(P) \subset [-C_0, C_0]\}
        \label{eqn:maximizer-good-event}
    \end{equation}
    occurs with probability $1 - 2/n$, uniformly over $\mathcal P(\Xi, C_0, r_-)$. Repeat the estimate \eqref{eqn:hat-xi-error-bad-expectation} for the event \eqref{eqn:maximizer-good-event} to see that 
    \begin{align*}
        \E_P[|\hat \xi_n^2 - \xi_\star^2|\1_{E^c_1}] = O(n^{-1/2}).
    \end{align*}
    Now we increase $n_0 > 0$ until $C_{n_0} > C_0$, so that on $\delta_P(C_n) \vee 0 = 0$ for every $n > n_0$ and $P \in \mathcal P(\Xi, C_0, r_{-})$. We then have by (\ref{eqn:maximizer-good-event}) that 
    \[
    |\hat \xi_n^2 - \xi_\star^2| \leq \sup_{|\lambda| \leq C_0} |L_n(\lambda) - L(\lambda)|. 
    \]
    Take expectations and apply Proposition~\ref{prop:L_n-uniform-rootn} to conclude.
\end{proof}
\begin{proof}[Proof of the upper bound in Theorem~\ref{thm:minimax-lower-bound}]
    Set $\Xi = 1$ and $\alpha = \frac12(1 - \eps)$. The bound in Theorem~\ref{thm:minimax-upper-bound}(a) yields
    \[
    \mathcal M_n(\mathcal P(\Xi, r, C_0)) \leq r\left(\Xi^{-\frac12}(\log n)^{\frac{1 - \eps}{2}}\right)_+ + C\Xi n^{-1/2 + \eps} \leq r(\Xi^{-\frac12}(\log n)^{\frac12 - \eps})_+ + C\Xi n^{-1/2 + \eps}
    \]
    after rescaling, where the last part is true 
    provided $n > n_0$ for some $n_0$.  
\end{proof}
\subsubsection{Proofs for Section~\ref{sec:divergence}}
Before we present the proof of Proposition~\ref{prop:divergence}, we shall need a brief lemma. 
\begin{proposition}
\label{prop:std-divergence}
    Let $X_1, \dots, X_n \sim P$ and write $\hat \sigma_n^2 = \frac1n \sum_{i=1}^n (X_i - \bar X)^2$. If $\E_PX^2 = \infty$, then $\hat \sigma_n^2 \xrightarrow{\mathrm{a.s.}} \infty$.
\end{proposition}
\begin{proof}
    Observe that 
    \[
    \hat \sigma_n^2 = \frac1n\sum_{i=1}^n (X_i - \bar X)^2 = \frac{1}{2n^2} \sum_{i,j=1}^n (X_i - X_j)^2.
    \]
    Fix $R > 0$, and let 
    \[
    T_R(x) = \begin{cases}
        R & x > R \\
        x & -R \leq x \leq R \\
        -R & x < -R.
    \end{cases}
    \]
    It then follows that $|X_i - X_j| \geq |T_R(X_i) - T_R(X_j)|$, and therefore by the law of large numbers for $U$-statistics \citep{hoeffding1961strong}, we have
    \[
     \frac{1}{2n^2} \sum_{i,j=1}^n (X_i - X_j)^2 \geq \frac{1}{2n^2} \sum_{i,j=1}^n (T_R(X_i) - T_R(X_j))^2 \xrightarrow{\mathrm{a.s.}} \mathrm{Var}(T_R(X_1)).
    \]
    Let $X'_1$ be an iid copy of $X_1$. Then, we have by symmetrization that 
    \[
    \mathrm{Var}(T_R(X_1)) = \frac12 \E (T_R(X_1) - T_R(X_1'))^2,
    \]
    but $T_R(X_1) \to X_1$ pointwise a.s. as $R \to \infty$, so Fatou's lemma says that 
    \[
    \liminf_{R \to \infty} \mathrm{Var}(T_R(X_1)) \geq \frac12 \E (X_1 - X_1')^2 = \infty. 
    \]
\end{proof}
Now we can prove Proposition~\ref{prop:divergence}.
\begin{proof}[Proof of Proposition~\ref{prop:divergence}]
First, if $\E X^2 = \infty$ then Proposition~\ref{prop:std-divergence} says that $\hat \xi_n^2 \geq \hat \sigma_n^2\xrightarrow{\mathrm{a.s.}}\infty$. Otherwise, we may assume $\E X < \infty$. 
     If $\E e^{\lambda(X - \E X)} = \infty$ at some $\lambda$ with $\E X < \infty$, then provided $n$ be large enough, $\lambda$ will be contained in $[-C_n, C_n]$. By the LLN for infinite means (\cite{durrett2019probability}, Theorem 2.3.8), $M_n(\lambda) \xrightarrow{\mathrm{a.s.}} \infty$, and then 
    \[
    \hat \xi_n^2 \geq 2\lambda^{-2}\psi_n(\lambda)  =  2\lambda^{-2}\log M_n(\lambda)\xrightarrow{\mathrm{a.s.}}\infty.
    \]
   Finally, assume that $\E e^{\lambda(X - \E X)} < \infty$ for all $\lambda \in \R$. In this case, for any $\xi^2 > 0$ there is $\lambda$ such that $\E e^{\lambda(X - \E X)} > \exp(2^{-1}\lambda^2\xi^2)$. But $M_n(\lambda) \xrightarrow{\mathrm{a.s.}} \E e^{\lambda(X - \E X)} $, so provided $n$ be large enough, $M_n(\lambda) > \exp(2^{-1}\lambda^2\xi^2) - \eps$ occurs with probability 1. Then we have 
    \[
    \hat \xi^2_n \geq 2\lambda^{-2}\log M_n(\lambda) > \xi^2,
    \]
    so by the upper continuity of measure we have
    \[
    \P\left(\liminf_{n \to \infty} \hat \xi^2_n =\infty\right) = \P\left(\bigcap_{m=1}^\infty \liminf_{n \to \infty} \hat \xi_n^2 \geq m^2\right) = \lim_{m \to \infty} \P\left(\liminf_{n \to \infty}\hat \xi_n^2 \geq m^2\right) = 1.
    \]
\end{proof}

\begin{proof}[Proof of Proposition~\ref{prop:divergence2}]
For any $\lambda \in [0, C_n]$ we have
    \[
    \hat \xi_n^2 \geq \frac{2}{\lambda^2}\log\left(\frac1n\sum_{i=1}^n e^{\lambda(X_i - \bar X)}\right) \geq \frac{2}{\lambda^2}\log \left(\frac1ne^{(X_{(n)} - \bar X)\lambda}\right) = \frac{2}{\lambda}\left\{(X_{(n)} - \bar X)- \frac{\log n}{\lambda}\right\},
    \]
    and on the other hand if $\lambda \in [-C_n, 0]$ then
    \[
    \hat \xi_n^2 \geq \frac{2}{\lambda^2}\log\left(\frac1n\sum_{i=1}^n e^{\lambda(X_i - \bar X)}\right) \geq \frac{2}{\lambda^2}\log \left(\frac1ne^{(X_{(1)} - \bar X)\lambda}\right) \ge \frac{2}{|\lambda|}\left\{|X_{(1)} - \bar X|- \frac{\log n}{|\lambda|}\right\}.
    \]
    Note that $\Delta_n = | X_{(n)} - \bar{X}| \vee |X_{(1)} - \bar{X}|$. Assume without loss of generality that $\Delta_n = | X_{(n)} - \bar{X}|$. 

    We define $\lambda_n = (X_{(n)} - \bar X)^{-1}(2 \log n)$ and see that $\lambda_n \leq C_n$ under the condition that $\Delta_n \geq \frac{2\log n}{C_n}$. Plugging $\lambda_n$ into the RHS of the first inequality, we have
    \[
    \hat \xi_n^2 \geq \frac{(X_{(n)} - \bar X)^2}{2 \log n} = \frac{\Delta_n^2}{2 \log n}.
    \]
\end{proof}
\subsection{Details of Section 4}
\label{sec:experiments-appendix}
The algorithms were retrieved and implemented as described in \cite{liu2025beyond}. The TNFa dataset was used with the DIP \citep{xenarios2000dip} and the Fly Transcriptome dataset was used with the STRING network \citep{szklarczyk2017string}. The permutation test was implemented through a modification to the Empirical Pipeline \citep{levi2021domino, liu2025beyond}, and the modified code can be accessed at \texttt{\href{https://github.com/LiuJ0/AMI-Benchmark}{https://github.com/LiuJ0/AMI-Benchmark}}.

Finally, we record the proof of Lemma~\ref{lem:bulk-shrink}. 
\begin{proof}[Proof of Lemma~\ref{lem:bulk-shrink}]
Let $\lambda > 0$ and write $p_u := P((-\infty, u])$. By Jensen's inequality,
\[
 \E_{B_uP}[e^{\lambda Y}] = p_ue^{\lambda \E_P[X | X \leq u]} + \E_P[e^{\lambda X} \1_{X > u}]  
     \leq p_u\E[e^{\lambda X} | X \leq u]  + \E_P[e^{\lambda X} \1_{X > u}] = \E_P e^{\lambda X}
\]
so $L_{B_uP}(\lambda) \leq L_P(\lambda)$. Supremize over $\lambda > 0$ on both sides to prove the claim. 
\end{proof}

\end{document}

%% file: figures/log-log_plot.pgf
%% Creator: Matplotlib, PGF backend
%%
%% To include the figure in your LaTeX document, write
%%   \input{<filename>.pgf}
%%
%% Make sure the required packages are loaded in your preamble
%%   \usepackage{pgf}
%%
%% Also ensure that all the required font packages are loaded; for instance,
%% the lmodern package is sometimes necessary when using math font.
%%   \usepackage{lmodern}
%%
%% Figures using additional raster images can only be included by \input if
%% they are in the same directory as the main LaTeX file. For loading figures
%% from other directories you can use the `import` package
%% and then include the figures with
%%   \import{<path to file>}{<filename>.pgf}
%%
%% Matplotlib used the following preamble
%%   \def\mathdefault#1{#1}
%%   \everymath=\expandafter{\the\everymath\displaystyle}
%%   
%%   \ifdefined\pdftexversion\else  % non-pdftex case.
%%     \usepackage{fontspec}
%%     \setmainfont{DejaVuSerif.ttf}[Path=\detokenize{C:/Users/jl3089/AppData/Local/Programs/Python/Python312/Lib/site-packages/matplotlib/mpl-data/fonts/ttf/}]
%%     \setsansfont{DejaVuSans.ttf}[Path=\detokenize{C:/Users/jl3089/AppData/Local/Programs/Python/Python312/Lib/site-packages/matplotlib/mpl-data/fonts/ttf/}]
%%     \setmonofont{DejaVuSansMono.ttf}[Path=\detokenize{C:/Users/jl3089/AppData/Local/Programs/Python/Python312/Lib/site-packages/matplotlib/mpl-data/fonts/ttf/}]
%%   \fi
%%   \makeatletter\@ifpackageloaded{underscore}{}{\usepackage[strings]{underscore}}\makeatother
%%
\begingroup%
\makeatletter%
\begin{pgfpicture}%
\pgfpathrectangle{\pgfpointorigin}{\pgfqpoint{10.000000in}{6.000000in}}%
\pgfusepath{use as bounding box, clip}%
\begin{pgfscope}%
\pgfsetbuttcap%
\pgfsetmiterjoin%
\definecolor{currentfill}{rgb}{1.000000,1.000000,1.000000}%
\pgfsetfillcolor{currentfill}%
\pgfsetlinewidth{0.000000pt}%
\definecolor{currentstroke}{rgb}{1.000000,1.000000,1.000000}%
\pgfsetstrokecolor{currentstroke}%
\pgfsetdash{}{0pt}%
\pgfpathmoveto{\pgfqpoint{0.000000in}{0.000000in}}%
\pgfpathlineto{\pgfqpoint{10.000000in}{0.000000in}}%
\pgfpathlineto{\pgfqpoint{10.000000in}{6.000000in}}%
\pgfpathlineto{\pgfqpoint{0.000000in}{6.000000in}}%
\pgfpathlineto{\pgfqpoint{0.000000in}{0.000000in}}%
\pgfpathclose%
\pgfusepath{fill}%
\end{pgfscope}%
\begin{pgfscope}%
\pgfsetbuttcap%
\pgfsetmiterjoin%
\definecolor{currentfill}{rgb}{1.000000,1.000000,1.000000}%
\pgfsetfillcolor{currentfill}%
\pgfsetlinewidth{0.000000pt}%
\definecolor{currentstroke}{rgb}{0.000000,0.000000,0.000000}%
\pgfsetstrokecolor{currentstroke}%
\pgfsetstrokeopacity{0.000000}%
\pgfsetdash{}{0pt}%
\pgfpathmoveto{\pgfqpoint{1.342500in}{0.862778in}}%
\pgfpathlineto{\pgfqpoint{9.712602in}{0.862778in}}%
\pgfpathlineto{\pgfqpoint{9.712602in}{5.850000in}}%
\pgfpathlineto{\pgfqpoint{1.342500in}{5.850000in}}%
\pgfpathlineto{\pgfqpoint{1.342500in}{0.862778in}}%
\pgfpathclose%
\pgfusepath{fill}%
\end{pgfscope}%
\begin{pgfscope}%
\pgfpathrectangle{\pgfqpoint{1.342500in}{0.862778in}}{\pgfqpoint{8.370102in}{4.987222in}}%
\pgfusepath{clip}%
\pgfsetbuttcap%
\pgfsetroundjoin%
\pgfsetlinewidth{0.803000pt}%
\definecolor{currentstroke}{rgb}{0.690196,0.690196,0.690196}%
\pgfsetstrokecolor{currentstroke}%
\pgfsetstrokeopacity{0.700000}%
\pgfsetdash{{2.960000pt}{1.280000pt}}{0.000000pt}%
\pgfpathmoveto{\pgfqpoint{2.619918in}{0.862778in}}%
\pgfpathlineto{\pgfqpoint{2.619918in}{5.850000in}}%
\pgfusepath{stroke}%
\end{pgfscope}%
\begin{pgfscope}%
\pgfsetbuttcap%
\pgfsetroundjoin%
\definecolor{currentfill}{rgb}{0.000000,0.000000,0.000000}%
\pgfsetfillcolor{currentfill}%
\pgfsetlinewidth{0.803000pt}%
\definecolor{currentstroke}{rgb}{0.000000,0.000000,0.000000}%
\pgfsetstrokecolor{currentstroke}%
\pgfsetdash{}{0pt}%
\pgfsys@defobject{currentmarker}{\pgfqpoint{0.000000in}{-0.048611in}}{\pgfqpoint{0.000000in}{0.000000in}}{%
\pgfpathmoveto{\pgfqpoint{0.000000in}{0.000000in}}%
\pgfpathlineto{\pgfqpoint{0.000000in}{-0.048611in}}%
\pgfusepath{stroke,fill}%
}%
\begin{pgfscope}%
\pgfsys@transformshift{2.619918in}{0.862778in}%
\pgfsys@useobject{currentmarker}{}%
\end{pgfscope}%
\end{pgfscope}%
\begin{pgfscope}%
\definecolor{textcolor}{rgb}{0.000000,0.000000,0.000000}%
\pgfsetstrokecolor{textcolor}%
\pgfsetfillcolor{textcolor}%
\pgftext[x=2.619918in,y=0.765556in,,top]{\color{textcolor}{\sffamily\fontsize{18.000000}{21.600000}\selectfont\catcode`\^=\active\def^{\ifmmode\sp\else\^{}\fi}\catcode`\%=\active\def%{\%}6.4}}%
\end{pgfscope}%
\begin{pgfscope}%
\pgfpathrectangle{\pgfqpoint{1.342500in}{0.862778in}}{\pgfqpoint{8.370102in}{4.987222in}}%
\pgfusepath{clip}%
\pgfsetbuttcap%
\pgfsetroundjoin%
\pgfsetlinewidth{0.803000pt}%
\definecolor{currentstroke}{rgb}{0.690196,0.690196,0.690196}%
\pgfsetstrokecolor{currentstroke}%
\pgfsetstrokeopacity{0.700000}%
\pgfsetdash{{2.960000pt}{1.280000pt}}{0.000000pt}%
\pgfpathmoveto{\pgfqpoint{4.376355in}{0.862778in}}%
\pgfpathlineto{\pgfqpoint{4.376355in}{5.850000in}}%
\pgfusepath{stroke}%
\end{pgfscope}%
\begin{pgfscope}%
\pgfsetbuttcap%
\pgfsetroundjoin%
\definecolor{currentfill}{rgb}{0.000000,0.000000,0.000000}%
\pgfsetfillcolor{currentfill}%
\pgfsetlinewidth{0.803000pt}%
\definecolor{currentstroke}{rgb}{0.000000,0.000000,0.000000}%
\pgfsetstrokecolor{currentstroke}%
\pgfsetdash{}{0pt}%
\pgfsys@defobject{currentmarker}{\pgfqpoint{0.000000in}{-0.048611in}}{\pgfqpoint{0.000000in}{0.000000in}}{%
\pgfpathmoveto{\pgfqpoint{0.000000in}{0.000000in}}%
\pgfpathlineto{\pgfqpoint{0.000000in}{-0.048611in}}%
\pgfusepath{stroke,fill}%
}%
\begin{pgfscope}%
\pgfsys@transformshift{4.376355in}{0.862778in}%
\pgfsys@useobject{currentmarker}{}%
\end{pgfscope}%
\end{pgfscope}%
\begin{pgfscope}%
\definecolor{textcolor}{rgb}{0.000000,0.000000,0.000000}%
\pgfsetstrokecolor{textcolor}%
\pgfsetfillcolor{textcolor}%
\pgftext[x=4.376355in,y=0.765556in,,top]{\color{textcolor}{\sffamily\fontsize{18.000000}{21.600000}\selectfont\catcode`\^=\active\def^{\ifmmode\sp\else\^{}\fi}\catcode`\%=\active\def%{\%}7.2}}%
\end{pgfscope}%
\begin{pgfscope}%
\pgfpathrectangle{\pgfqpoint{1.342500in}{0.862778in}}{\pgfqpoint{8.370102in}{4.987222in}}%
\pgfusepath{clip}%
\pgfsetbuttcap%
\pgfsetroundjoin%
\pgfsetlinewidth{0.803000pt}%
\definecolor{currentstroke}{rgb}{0.690196,0.690196,0.690196}%
\pgfsetstrokecolor{currentstroke}%
\pgfsetstrokeopacity{0.700000}%
\pgfsetdash{{2.960000pt}{1.280000pt}}{0.000000pt}%
\pgfpathmoveto{\pgfqpoint{6.132792in}{0.862778in}}%
\pgfpathlineto{\pgfqpoint{6.132792in}{5.850000in}}%
\pgfusepath{stroke}%
\end{pgfscope}%
\begin{pgfscope}%
\pgfsetbuttcap%
\pgfsetroundjoin%
\definecolor{currentfill}{rgb}{0.000000,0.000000,0.000000}%
\pgfsetfillcolor{currentfill}%
\pgfsetlinewidth{0.803000pt}%
\definecolor{currentstroke}{rgb}{0.000000,0.000000,0.000000}%
\pgfsetstrokecolor{currentstroke}%
\pgfsetdash{}{0pt}%
\pgfsys@defobject{currentmarker}{\pgfqpoint{0.000000in}{-0.048611in}}{\pgfqpoint{0.000000in}{0.000000in}}{%
\pgfpathmoveto{\pgfqpoint{0.000000in}{0.000000in}}%
\pgfpathlineto{\pgfqpoint{0.000000in}{-0.048611in}}%
\pgfusepath{stroke,fill}%
}%
\begin{pgfscope}%
\pgfsys@transformshift{6.132792in}{0.862778in}%
\pgfsys@useobject{currentmarker}{}%
\end{pgfscope}%
\end{pgfscope}%
\begin{pgfscope}%
\definecolor{textcolor}{rgb}{0.000000,0.000000,0.000000}%
\pgfsetstrokecolor{textcolor}%
\pgfsetfillcolor{textcolor}%
\pgftext[x=6.132792in,y=0.765556in,,top]{\color{textcolor}{\sffamily\fontsize{18.000000}{21.600000}\selectfont\catcode`\^=\active\def^{\ifmmode\sp\else\^{}\fi}\catcode`\%=\active\def%{\%}8.0}}%
\end{pgfscope}%
\begin{pgfscope}%
\pgfpathrectangle{\pgfqpoint{1.342500in}{0.862778in}}{\pgfqpoint{8.370102in}{4.987222in}}%
\pgfusepath{clip}%
\pgfsetbuttcap%
\pgfsetroundjoin%
\pgfsetlinewidth{0.803000pt}%
\definecolor{currentstroke}{rgb}{0.690196,0.690196,0.690196}%
\pgfsetstrokecolor{currentstroke}%
\pgfsetstrokeopacity{0.700000}%
\pgfsetdash{{2.960000pt}{1.280000pt}}{0.000000pt}%
\pgfpathmoveto{\pgfqpoint{7.889229in}{0.862778in}}%
\pgfpathlineto{\pgfqpoint{7.889229in}{5.850000in}}%
\pgfusepath{stroke}%
\end{pgfscope}%
\begin{pgfscope}%
\pgfsetbuttcap%
\pgfsetroundjoin%
\definecolor{currentfill}{rgb}{0.000000,0.000000,0.000000}%
\pgfsetfillcolor{currentfill}%
\pgfsetlinewidth{0.803000pt}%
\definecolor{currentstroke}{rgb}{0.000000,0.000000,0.000000}%
\pgfsetstrokecolor{currentstroke}%
\pgfsetdash{}{0pt}%
\pgfsys@defobject{currentmarker}{\pgfqpoint{0.000000in}{-0.048611in}}{\pgfqpoint{0.000000in}{0.000000in}}{%
\pgfpathmoveto{\pgfqpoint{0.000000in}{0.000000in}}%
\pgfpathlineto{\pgfqpoint{0.000000in}{-0.048611in}}%
\pgfusepath{stroke,fill}%
}%
\begin{pgfscope}%
\pgfsys@transformshift{7.889229in}{0.862778in}%
\pgfsys@useobject{currentmarker}{}%
\end{pgfscope}%
\end{pgfscope}%
\begin{pgfscope}%
\definecolor{textcolor}{rgb}{0.000000,0.000000,0.000000}%
\pgfsetstrokecolor{textcolor}%
\pgfsetfillcolor{textcolor}%
\pgftext[x=7.889229in,y=0.765556in,,top]{\color{textcolor}{\sffamily\fontsize{18.000000}{21.600000}\selectfont\catcode`\^=\active\def^{\ifmmode\sp\else\^{}\fi}\catcode`\%=\active\def%{\%}8.8}}%
\end{pgfscope}%
\begin{pgfscope}%
\pgfpathrectangle{\pgfqpoint{1.342500in}{0.862778in}}{\pgfqpoint{8.370102in}{4.987222in}}%
\pgfusepath{clip}%
\pgfsetbuttcap%
\pgfsetroundjoin%
\pgfsetlinewidth{0.803000pt}%
\definecolor{currentstroke}{rgb}{0.690196,0.690196,0.690196}%
\pgfsetstrokecolor{currentstroke}%
\pgfsetstrokeopacity{0.700000}%
\pgfsetdash{{2.960000pt}{1.280000pt}}{0.000000pt}%
\pgfpathmoveto{\pgfqpoint{9.645666in}{0.862778in}}%
\pgfpathlineto{\pgfqpoint{9.645666in}{5.850000in}}%
\pgfusepath{stroke}%
\end{pgfscope}%
\begin{pgfscope}%
\pgfsetbuttcap%
\pgfsetroundjoin%
\definecolor{currentfill}{rgb}{0.000000,0.000000,0.000000}%
\pgfsetfillcolor{currentfill}%
\pgfsetlinewidth{0.803000pt}%
\definecolor{currentstroke}{rgb}{0.000000,0.000000,0.000000}%
\pgfsetstrokecolor{currentstroke}%
\pgfsetdash{}{0pt}%
\pgfsys@defobject{currentmarker}{\pgfqpoint{0.000000in}{-0.048611in}}{\pgfqpoint{0.000000in}{0.000000in}}{%
\pgfpathmoveto{\pgfqpoint{0.000000in}{0.000000in}}%
\pgfpathlineto{\pgfqpoint{0.000000in}{-0.048611in}}%
\pgfusepath{stroke,fill}%
}%
\begin{pgfscope}%
\pgfsys@transformshift{9.645666in}{0.862778in}%
\pgfsys@useobject{currentmarker}{}%
\end{pgfscope}%
\end{pgfscope}%
\begin{pgfscope}%
\definecolor{textcolor}{rgb}{0.000000,0.000000,0.000000}%
\pgfsetstrokecolor{textcolor}%
\pgfsetfillcolor{textcolor}%
\pgftext[x=9.645666in,y=0.765556in,,top]{\color{textcolor}{\sffamily\fontsize{18.000000}{21.600000}\selectfont\catcode`\^=\active\def^{\ifmmode\sp\else\^{}\fi}\catcode`\%=\active\def%{\%}9.6}}%
\end{pgfscope}%
\begin{pgfscope}%
\definecolor{textcolor}{rgb}{0.000000,0.000000,0.000000}%
\pgfsetstrokecolor{textcolor}%
\pgfsetfillcolor{textcolor}%
\pgftext[x=5.527551in,y=0.468057in,,top]{\color{textcolor}{\sffamily\fontsize{20.000000}{24.000000}\selectfont\catcode`\^=\active\def^{\ifmmode\sp\else\^{}\fi}\catcode`\%=\active\def%{\%}$\log(n)$}}%
\end{pgfscope}%
\begin{pgfscope}%
\pgfpathrectangle{\pgfqpoint{1.342500in}{0.862778in}}{\pgfqpoint{8.370102in}{4.987222in}}%
\pgfusepath{clip}%
\pgfsetbuttcap%
\pgfsetroundjoin%
\pgfsetlinewidth{0.803000pt}%
\definecolor{currentstroke}{rgb}{0.690196,0.690196,0.690196}%
\pgfsetstrokecolor{currentstroke}%
\pgfsetstrokeopacity{0.700000}%
\pgfsetdash{{2.960000pt}{1.280000pt}}{0.000000pt}%
\pgfpathmoveto{\pgfqpoint{1.342500in}{1.047595in}}%
\pgfpathlineto{\pgfqpoint{9.712602in}{1.047595in}}%
\pgfusepath{stroke}%
\end{pgfscope}%
\begin{pgfscope}%
\pgfsetbuttcap%
\pgfsetroundjoin%
\definecolor{currentfill}{rgb}{0.000000,0.000000,0.000000}%
\pgfsetfillcolor{currentfill}%
\pgfsetlinewidth{0.803000pt}%
\definecolor{currentstroke}{rgb}{0.000000,0.000000,0.000000}%
\pgfsetstrokecolor{currentstroke}%
\pgfsetdash{}{0pt}%
\pgfsys@defobject{currentmarker}{\pgfqpoint{-0.048611in}{0.000000in}}{\pgfqpoint{-0.000000in}{0.000000in}}{%
\pgfpathmoveto{\pgfqpoint{-0.000000in}{0.000000in}}%
\pgfpathlineto{\pgfqpoint{-0.048611in}{0.000000in}}%
\pgfusepath{stroke,fill}%
}%
\begin{pgfscope}%
\pgfsys@transformshift{1.342500in}{1.047595in}%
\pgfsys@useobject{currentmarker}{}%
\end{pgfscope}%
\end{pgfscope}%
\begin{pgfscope}%
\definecolor{textcolor}{rgb}{0.000000,0.000000,0.000000}%
\pgfsetstrokecolor{textcolor}%
\pgfsetfillcolor{textcolor}%
\pgftext[x=0.661028in, y=0.952625in, left, base]{\color{textcolor}{\sffamily\fontsize{18.000000}{21.600000}\selectfont\catcode`\^=\active\def^{\ifmmode\sp\else\^{}\fi}\catcode`\%=\active\def%{\%}\ensuremath{-}4.8}}%
\end{pgfscope}%
\begin{pgfscope}%
\pgfpathrectangle{\pgfqpoint{1.342500in}{0.862778in}}{\pgfqpoint{8.370102in}{4.987222in}}%
\pgfusepath{clip}%
\pgfsetbuttcap%
\pgfsetroundjoin%
\pgfsetlinewidth{0.803000pt}%
\definecolor{currentstroke}{rgb}{0.690196,0.690196,0.690196}%
\pgfsetstrokecolor{currentstroke}%
\pgfsetstrokeopacity{0.700000}%
\pgfsetdash{{2.960000pt}{1.280000pt}}{0.000000pt}%
\pgfpathmoveto{\pgfqpoint{1.342500in}{2.124159in}}%
\pgfpathlineto{\pgfqpoint{9.712602in}{2.124159in}}%
\pgfusepath{stroke}%
\end{pgfscope}%
\begin{pgfscope}%
\pgfsetbuttcap%
\pgfsetroundjoin%
\definecolor{currentfill}{rgb}{0.000000,0.000000,0.000000}%
\pgfsetfillcolor{currentfill}%
\pgfsetlinewidth{0.803000pt}%
\definecolor{currentstroke}{rgb}{0.000000,0.000000,0.000000}%
\pgfsetstrokecolor{currentstroke}%
\pgfsetdash{}{0pt}%
\pgfsys@defobject{currentmarker}{\pgfqpoint{-0.048611in}{0.000000in}}{\pgfqpoint{-0.000000in}{0.000000in}}{%
\pgfpathmoveto{\pgfqpoint{-0.000000in}{0.000000in}}%
\pgfpathlineto{\pgfqpoint{-0.048611in}{0.000000in}}%
\pgfusepath{stroke,fill}%
}%
\begin{pgfscope}%
\pgfsys@transformshift{1.342500in}{2.124159in}%
\pgfsys@useobject{currentmarker}{}%
\end{pgfscope}%
\end{pgfscope}%
\begin{pgfscope}%
\definecolor{textcolor}{rgb}{0.000000,0.000000,0.000000}%
\pgfsetstrokecolor{textcolor}%
\pgfsetfillcolor{textcolor}%
\pgftext[x=0.661028in, y=2.029188in, left, base]{\color{textcolor}{\sffamily\fontsize{18.000000}{21.600000}\selectfont\catcode`\^=\active\def^{\ifmmode\sp\else\^{}\fi}\catcode`\%=\active\def%{\%}\ensuremath{-}4.4}}%
\end{pgfscope}%
\begin{pgfscope}%
\pgfpathrectangle{\pgfqpoint{1.342500in}{0.862778in}}{\pgfqpoint{8.370102in}{4.987222in}}%
\pgfusepath{clip}%
\pgfsetbuttcap%
\pgfsetroundjoin%
\pgfsetlinewidth{0.803000pt}%
\definecolor{currentstroke}{rgb}{0.690196,0.690196,0.690196}%
\pgfsetstrokecolor{currentstroke}%
\pgfsetstrokeopacity{0.700000}%
\pgfsetdash{{2.960000pt}{1.280000pt}}{0.000000pt}%
\pgfpathmoveto{\pgfqpoint{1.342500in}{3.200723in}}%
\pgfpathlineto{\pgfqpoint{9.712602in}{3.200723in}}%
\pgfusepath{stroke}%
\end{pgfscope}%
\begin{pgfscope}%
\pgfsetbuttcap%
\pgfsetroundjoin%
\definecolor{currentfill}{rgb}{0.000000,0.000000,0.000000}%
\pgfsetfillcolor{currentfill}%
\pgfsetlinewidth{0.803000pt}%
\definecolor{currentstroke}{rgb}{0.000000,0.000000,0.000000}%
\pgfsetstrokecolor{currentstroke}%
\pgfsetdash{}{0pt}%
\pgfsys@defobject{currentmarker}{\pgfqpoint{-0.048611in}{0.000000in}}{\pgfqpoint{-0.000000in}{0.000000in}}{%
\pgfpathmoveto{\pgfqpoint{-0.000000in}{0.000000in}}%
\pgfpathlineto{\pgfqpoint{-0.048611in}{0.000000in}}%
\pgfusepath{stroke,fill}%
}%
\begin{pgfscope}%
\pgfsys@transformshift{1.342500in}{3.200723in}%
\pgfsys@useobject{currentmarker}{}%
\end{pgfscope}%
\end{pgfscope}%
\begin{pgfscope}%
\definecolor{textcolor}{rgb}{0.000000,0.000000,0.000000}%
\pgfsetstrokecolor{textcolor}%
\pgfsetfillcolor{textcolor}%
\pgftext[x=0.661028in, y=3.105752in, left, base]{\color{textcolor}{\sffamily\fontsize{18.000000}{21.600000}\selectfont\catcode`\^=\active\def^{\ifmmode\sp\else\^{}\fi}\catcode`\%=\active\def%{\%}\ensuremath{-}4.0}}%
\end{pgfscope}%
\begin{pgfscope}%
\pgfpathrectangle{\pgfqpoint{1.342500in}{0.862778in}}{\pgfqpoint{8.370102in}{4.987222in}}%
\pgfusepath{clip}%
\pgfsetbuttcap%
\pgfsetroundjoin%
\pgfsetlinewidth{0.803000pt}%
\definecolor{currentstroke}{rgb}{0.690196,0.690196,0.690196}%
\pgfsetstrokecolor{currentstroke}%
\pgfsetstrokeopacity{0.700000}%
\pgfsetdash{{2.960000pt}{1.280000pt}}{0.000000pt}%
\pgfpathmoveto{\pgfqpoint{1.342500in}{4.277286in}}%
\pgfpathlineto{\pgfqpoint{9.712602in}{4.277286in}}%
\pgfusepath{stroke}%
\end{pgfscope}%
\begin{pgfscope}%
\pgfsetbuttcap%
\pgfsetroundjoin%
\definecolor{currentfill}{rgb}{0.000000,0.000000,0.000000}%
\pgfsetfillcolor{currentfill}%
\pgfsetlinewidth{0.803000pt}%
\definecolor{currentstroke}{rgb}{0.000000,0.000000,0.000000}%
\pgfsetstrokecolor{currentstroke}%
\pgfsetdash{}{0pt}%
\pgfsys@defobject{currentmarker}{\pgfqpoint{-0.048611in}{0.000000in}}{\pgfqpoint{-0.000000in}{0.000000in}}{%
\pgfpathmoveto{\pgfqpoint{-0.000000in}{0.000000in}}%
\pgfpathlineto{\pgfqpoint{-0.048611in}{0.000000in}}%
\pgfusepath{stroke,fill}%
}%
\begin{pgfscope}%
\pgfsys@transformshift{1.342500in}{4.277286in}%
\pgfsys@useobject{currentmarker}{}%
\end{pgfscope}%
\end{pgfscope}%
\begin{pgfscope}%
\definecolor{textcolor}{rgb}{0.000000,0.000000,0.000000}%
\pgfsetstrokecolor{textcolor}%
\pgfsetfillcolor{textcolor}%
\pgftext[x=0.661028in, y=4.182316in, left, base]{\color{textcolor}{\sffamily\fontsize{18.000000}{21.600000}\selectfont\catcode`\^=\active\def^{\ifmmode\sp\else\^{}\fi}\catcode`\%=\active\def%{\%}\ensuremath{-}3.6}}%
\end{pgfscope}%
\begin{pgfscope}%
\pgfpathrectangle{\pgfqpoint{1.342500in}{0.862778in}}{\pgfqpoint{8.370102in}{4.987222in}}%
\pgfusepath{clip}%
\pgfsetbuttcap%
\pgfsetroundjoin%
\pgfsetlinewidth{0.803000pt}%
\definecolor{currentstroke}{rgb}{0.690196,0.690196,0.690196}%
\pgfsetstrokecolor{currentstroke}%
\pgfsetstrokeopacity{0.700000}%
\pgfsetdash{{2.960000pt}{1.280000pt}}{0.000000pt}%
\pgfpathmoveto{\pgfqpoint{1.342500in}{5.353850in}}%
\pgfpathlineto{\pgfqpoint{9.712602in}{5.353850in}}%
\pgfusepath{stroke}%
\end{pgfscope}%
\begin{pgfscope}%
\pgfsetbuttcap%
\pgfsetroundjoin%
\definecolor{currentfill}{rgb}{0.000000,0.000000,0.000000}%
\pgfsetfillcolor{currentfill}%
\pgfsetlinewidth{0.803000pt}%
\definecolor{currentstroke}{rgb}{0.000000,0.000000,0.000000}%
\pgfsetstrokecolor{currentstroke}%
\pgfsetdash{}{0pt}%
\pgfsys@defobject{currentmarker}{\pgfqpoint{-0.048611in}{0.000000in}}{\pgfqpoint{-0.000000in}{0.000000in}}{%
\pgfpathmoveto{\pgfqpoint{-0.000000in}{0.000000in}}%
\pgfpathlineto{\pgfqpoint{-0.048611in}{0.000000in}}%
\pgfusepath{stroke,fill}%
}%
\begin{pgfscope}%
\pgfsys@transformshift{1.342500in}{5.353850in}%
\pgfsys@useobject{currentmarker}{}%
\end{pgfscope}%
\end{pgfscope}%
\begin{pgfscope}%
\definecolor{textcolor}{rgb}{0.000000,0.000000,0.000000}%
\pgfsetstrokecolor{textcolor}%
\pgfsetfillcolor{textcolor}%
\pgftext[x=0.661028in, y=5.258879in, left, base]{\color{textcolor}{\sffamily\fontsize{18.000000}{21.600000}\selectfont\catcode`\^=\active\def^{\ifmmode\sp\else\^{}\fi}\catcode`\%=\active\def%{\%}\ensuremath{-}3.2}}%
\end{pgfscope}%
\begin{pgfscope}%
\definecolor{textcolor}{rgb}{0.000000,0.000000,0.000000}%
\pgfsetstrokecolor{textcolor}%
\pgfsetfillcolor{textcolor}%
\pgftext[x=0.605472in,y=3.356389in,,bottom,rotate=90.000000]{\color{textcolor}{\sffamily\fontsize{20.000000}{24.000000}\selectfont\catcode`\^=\active\def^{\ifmmode\sp\else\^{}\fi}\catcode`\%=\active\def%{\%}$\log($Average $|\hat{\xi}_n - 1|)$}}%
\end{pgfscope}%
\begin{pgfscope}%
\pgfpathrectangle{\pgfqpoint{1.342500in}{0.862778in}}{\pgfqpoint{8.370102in}{4.987222in}}%
\pgfusepath{clip}%
\pgfsetrectcap%
\pgfsetroundjoin%
\pgfsetlinewidth{2.007500pt}%
\definecolor{currentstroke}{rgb}{0.235294,0.701961,0.443137}%
\pgfsetstrokecolor{currentstroke}%
\pgfsetdash{}{0pt}%
\pgfpathmoveto{\pgfqpoint{1.722959in}{5.129013in}}%
\pgfpathlineto{\pgfqpoint{3.244796in}{4.416219in}}%
\pgfpathlineto{\pgfqpoint{4.766633in}{3.648006in}}%
\pgfpathlineto{\pgfqpoint{6.288469in}{2.917547in}}%
\pgfpathlineto{\pgfqpoint{7.810306in}{2.161189in}}%
\pgfpathlineto{\pgfqpoint{9.332143in}{1.089470in}}%
\pgfusepath{stroke}%
\end{pgfscope}%
\begin{pgfscope}%
\pgfpathrectangle{\pgfqpoint{1.342500in}{0.862778in}}{\pgfqpoint{8.370102in}{4.987222in}}%
\pgfusepath{clip}%
\pgfsetbuttcap%
\pgfsetroundjoin%
\definecolor{currentfill}{rgb}{0.235294,0.701961,0.443137}%
\pgfsetfillcolor{currentfill}%
\pgfsetlinewidth{1.003750pt}%
\definecolor{currentstroke}{rgb}{0.235294,0.701961,0.443137}%
\pgfsetstrokecolor{currentstroke}%
\pgfsetdash{}{0pt}%
\pgfsys@defobject{currentmarker}{\pgfqpoint{-0.055556in}{-0.055556in}}{\pgfqpoint{0.055556in}{0.055556in}}{%
\pgfpathmoveto{\pgfqpoint{0.000000in}{-0.055556in}}%
\pgfpathcurveto{\pgfqpoint{0.014734in}{-0.055556in}}{\pgfqpoint{0.028866in}{-0.049702in}}{\pgfqpoint{0.039284in}{-0.039284in}}%
\pgfpathcurveto{\pgfqpoint{0.049702in}{-0.028866in}}{\pgfqpoint{0.055556in}{-0.014734in}}{\pgfqpoint{0.055556in}{0.000000in}}%
\pgfpathcurveto{\pgfqpoint{0.055556in}{0.014734in}}{\pgfqpoint{0.049702in}{0.028866in}}{\pgfqpoint{0.039284in}{0.039284in}}%
\pgfpathcurveto{\pgfqpoint{0.028866in}{0.049702in}}{\pgfqpoint{0.014734in}{0.055556in}}{\pgfqpoint{0.000000in}{0.055556in}}%
\pgfpathcurveto{\pgfqpoint{-0.014734in}{0.055556in}}{\pgfqpoint{-0.028866in}{0.049702in}}{\pgfqpoint{-0.039284in}{0.039284in}}%
\pgfpathcurveto{\pgfqpoint{-0.049702in}{0.028866in}}{\pgfqpoint{-0.055556in}{0.014734in}}{\pgfqpoint{-0.055556in}{0.000000in}}%
\pgfpathcurveto{\pgfqpoint{-0.055556in}{-0.014734in}}{\pgfqpoint{-0.049702in}{-0.028866in}}{\pgfqpoint{-0.039284in}{-0.039284in}}%
\pgfpathcurveto{\pgfqpoint{-0.028866in}{-0.049702in}}{\pgfqpoint{-0.014734in}{-0.055556in}}{\pgfqpoint{0.000000in}{-0.055556in}}%
\pgfpathlineto{\pgfqpoint{0.000000in}{-0.055556in}}%
\pgfpathclose%
\pgfusepath{stroke,fill}%
}%
\begin{pgfscope}%
\pgfsys@transformshift{1.722959in}{5.129013in}%
\pgfsys@useobject{currentmarker}{}%
\end{pgfscope}%
\begin{pgfscope}%
\pgfsys@transformshift{3.244796in}{4.416219in}%
\pgfsys@useobject{currentmarker}{}%
\end{pgfscope}%
\begin{pgfscope}%
\pgfsys@transformshift{4.766633in}{3.648006in}%
\pgfsys@useobject{currentmarker}{}%
\end{pgfscope}%
\begin{pgfscope}%
\pgfsys@transformshift{6.288469in}{2.917547in}%
\pgfsys@useobject{currentmarker}{}%
\end{pgfscope}%
\begin{pgfscope}%
\pgfsys@transformshift{7.810306in}{2.161189in}%
\pgfsys@useobject{currentmarker}{}%
\end{pgfscope}%
\begin{pgfscope}%
\pgfsys@transformshift{9.332143in}{1.089470in}%
\pgfsys@useobject{currentmarker}{}%
\end{pgfscope}%
\end{pgfscope}%
\begin{pgfscope}%
\pgfpathrectangle{\pgfqpoint{1.342500in}{0.862778in}}{\pgfqpoint{8.370102in}{4.987222in}}%
\pgfusepath{clip}%
\pgfsetrectcap%
\pgfsetroundjoin%
\pgfsetlinewidth{2.007500pt}%
\definecolor{currentstroke}{rgb}{0.866667,0.627451,0.866667}%
\pgfsetstrokecolor{currentstroke}%
\pgfsetdash{}{0pt}%
\pgfpathmoveto{\pgfqpoint{1.722959in}{5.623308in}}%
\pgfpathlineto{\pgfqpoint{3.244796in}{5.130555in}}%
\pgfpathlineto{\pgfqpoint{4.766633in}{4.577926in}}%
\pgfpathlineto{\pgfqpoint{6.288469in}{4.398924in}}%
\pgfpathlineto{\pgfqpoint{7.810306in}{3.957899in}}%
\pgfpathlineto{\pgfqpoint{9.332143in}{3.158012in}}%
\pgfusepath{stroke}%
\end{pgfscope}%
\begin{pgfscope}%
\pgfpathrectangle{\pgfqpoint{1.342500in}{0.862778in}}{\pgfqpoint{8.370102in}{4.987222in}}%
\pgfusepath{clip}%
\pgfsetbuttcap%
\pgfsetmiterjoin%
\definecolor{currentfill}{rgb}{0.866667,0.627451,0.866667}%
\pgfsetfillcolor{currentfill}%
\pgfsetlinewidth{1.003750pt}%
\definecolor{currentstroke}{rgb}{0.866667,0.627451,0.866667}%
\pgfsetstrokecolor{currentstroke}%
\pgfsetdash{}{0pt}%
\pgfsys@defobject{currentmarker}{\pgfqpoint{-0.055556in}{-0.055556in}}{\pgfqpoint{0.055556in}{0.055556in}}{%
\pgfpathmoveto{\pgfqpoint{-0.055556in}{-0.055556in}}%
\pgfpathlineto{\pgfqpoint{0.055556in}{-0.055556in}}%
\pgfpathlineto{\pgfqpoint{0.055556in}{0.055556in}}%
\pgfpathlineto{\pgfqpoint{-0.055556in}{0.055556in}}%
\pgfpathlineto{\pgfqpoint{-0.055556in}{-0.055556in}}%
\pgfpathclose%
\pgfusepath{stroke,fill}%
}%
\begin{pgfscope}%
\pgfsys@transformshift{1.722959in}{5.623308in}%
\pgfsys@useobject{currentmarker}{}%
\end{pgfscope}%
\begin{pgfscope}%
\pgfsys@transformshift{3.244796in}{5.130555in}%
\pgfsys@useobject{currentmarker}{}%
\end{pgfscope}%
\begin{pgfscope}%
\pgfsys@transformshift{4.766633in}{4.577926in}%
\pgfsys@useobject{currentmarker}{}%
\end{pgfscope}%
\begin{pgfscope}%
\pgfsys@transformshift{6.288469in}{4.398924in}%
\pgfsys@useobject{currentmarker}{}%
\end{pgfscope}%
\begin{pgfscope}%
\pgfsys@transformshift{7.810306in}{3.957899in}%
\pgfsys@useobject{currentmarker}{}%
\end{pgfscope}%
\begin{pgfscope}%
\pgfsys@transformshift{9.332143in}{3.158012in}%
\pgfsys@useobject{currentmarker}{}%
\end{pgfscope}%
\end{pgfscope}%
\begin{pgfscope}%
\pgfsetrectcap%
\pgfsetmiterjoin%
\pgfsetlinewidth{0.803000pt}%
\definecolor{currentstroke}{rgb}{0.000000,0.000000,0.000000}%
\pgfsetstrokecolor{currentstroke}%
\pgfsetdash{}{0pt}%
\pgfpathmoveto{\pgfqpoint{1.342500in}{0.862778in}}%
\pgfpathlineto{\pgfqpoint{1.342500in}{5.850000in}}%
\pgfusepath{stroke}%
\end{pgfscope}%
\begin{pgfscope}%
\pgfsetrectcap%
\pgfsetmiterjoin%
\pgfsetlinewidth{0.803000pt}%
\definecolor{currentstroke}{rgb}{0.000000,0.000000,0.000000}%
\pgfsetstrokecolor{currentstroke}%
\pgfsetdash{}{0pt}%
\pgfpathmoveto{\pgfqpoint{9.712602in}{0.862778in}}%
\pgfpathlineto{\pgfqpoint{9.712602in}{5.850000in}}%
\pgfusepath{stroke}%
\end{pgfscope}%
\begin{pgfscope}%
\pgfsetrectcap%
\pgfsetmiterjoin%
\pgfsetlinewidth{0.803000pt}%
\definecolor{currentstroke}{rgb}{0.000000,0.000000,0.000000}%
\pgfsetstrokecolor{currentstroke}%
\pgfsetdash{}{0pt}%
\pgfpathmoveto{\pgfqpoint{1.342500in}{0.862778in}}%
\pgfpathlineto{\pgfqpoint{9.712602in}{0.862778in}}%
\pgfusepath{stroke}%
\end{pgfscope}%
\begin{pgfscope}%
\pgfsetrectcap%
\pgfsetmiterjoin%
\pgfsetlinewidth{0.803000pt}%
\definecolor{currentstroke}{rgb}{0.000000,0.000000,0.000000}%
\pgfsetstrokecolor{currentstroke}%
\pgfsetdash{}{0pt}%
\pgfpathmoveto{\pgfqpoint{1.342500in}{5.850000in}}%
\pgfpathlineto{\pgfqpoint{9.712602in}{5.850000in}}%
\pgfusepath{stroke}%
\end{pgfscope}%
\begin{pgfscope}%
\pgfsetbuttcap%
\pgfsetmiterjoin%
\definecolor{currentfill}{rgb}{1.000000,1.000000,1.000000}%
\pgfsetfillcolor{currentfill}%
\pgfsetfillopacity{0.800000}%
\pgfsetlinewidth{1.003750pt}%
\definecolor{currentstroke}{rgb}{0.800000,0.800000,0.800000}%
\pgfsetstrokecolor{currentstroke}%
\pgfsetstrokeopacity{0.800000}%
\pgfsetdash{}{0pt}%
\pgfpathmoveto{\pgfqpoint{6.229205in}{4.615983in}}%
\pgfpathlineto{\pgfqpoint{9.518158in}{4.615983in}}%
\pgfpathquadraticcurveto{\pgfqpoint{9.573713in}{4.615983in}}{\pgfqpoint{9.573713in}{4.671539in}}%
\pgfpathlineto{\pgfqpoint{9.573713in}{5.655556in}}%
\pgfpathquadraticcurveto{\pgfqpoint{9.573713in}{5.711111in}}{\pgfqpoint{9.518158in}{5.711111in}}%
\pgfpathlineto{\pgfqpoint{6.229205in}{5.711111in}}%
\pgfpathquadraticcurveto{\pgfqpoint{6.173650in}{5.711111in}}{\pgfqpoint{6.173650in}{5.655556in}}%
\pgfpathlineto{\pgfqpoint{6.173650in}{4.671539in}}%
\pgfpathquadraticcurveto{\pgfqpoint{6.173650in}{4.615983in}}{\pgfqpoint{6.229205in}{4.615983in}}%
\pgfpathlineto{\pgfqpoint{6.229205in}{4.615983in}}%
\pgfpathclose%
\pgfusepath{stroke,fill}%
\end{pgfscope}%
\begin{pgfscope}%
\pgfsetrectcap%
\pgfsetroundjoin%
\pgfsetlinewidth{2.007500pt}%
\definecolor{currentstroke}{rgb}{0.235294,0.701961,0.443137}%
\pgfsetstrokecolor{currentstroke}%
\pgfsetdash{}{0pt}%
\pgfpathmoveto{\pgfqpoint{6.284761in}{5.421151in}}%
\pgfpathlineto{\pgfqpoint{6.562539in}{5.421151in}}%
\pgfpathlineto{\pgfqpoint{6.840316in}{5.421151in}}%
\pgfusepath{stroke}%
\end{pgfscope}%
\begin{pgfscope}%
\pgfsetbuttcap%
\pgfsetroundjoin%
\definecolor{currentfill}{rgb}{0.235294,0.701961,0.443137}%
\pgfsetfillcolor{currentfill}%
\pgfsetlinewidth{1.003750pt}%
\definecolor{currentstroke}{rgb}{0.235294,0.701961,0.443137}%
\pgfsetstrokecolor{currentstroke}%
\pgfsetdash{}{0pt}%
\pgfsys@defobject{currentmarker}{\pgfqpoint{-0.055556in}{-0.055556in}}{\pgfqpoint{0.055556in}{0.055556in}}{%
\pgfpathmoveto{\pgfqpoint{0.000000in}{-0.055556in}}%
\pgfpathcurveto{\pgfqpoint{0.014734in}{-0.055556in}}{\pgfqpoint{0.028866in}{-0.049702in}}{\pgfqpoint{0.039284in}{-0.039284in}}%
\pgfpathcurveto{\pgfqpoint{0.049702in}{-0.028866in}}{\pgfqpoint{0.055556in}{-0.014734in}}{\pgfqpoint{0.055556in}{0.000000in}}%
\pgfpathcurveto{\pgfqpoint{0.055556in}{0.014734in}}{\pgfqpoint{0.049702in}{0.028866in}}{\pgfqpoint{0.039284in}{0.039284in}}%
\pgfpathcurveto{\pgfqpoint{0.028866in}{0.049702in}}{\pgfqpoint{0.014734in}{0.055556in}}{\pgfqpoint{0.000000in}{0.055556in}}%
\pgfpathcurveto{\pgfqpoint{-0.014734in}{0.055556in}}{\pgfqpoint{-0.028866in}{0.049702in}}{\pgfqpoint{-0.039284in}{0.039284in}}%
\pgfpathcurveto{\pgfqpoint{-0.049702in}{0.028866in}}{\pgfqpoint{-0.055556in}{0.014734in}}{\pgfqpoint{-0.055556in}{0.000000in}}%
\pgfpathcurveto{\pgfqpoint{-0.055556in}{-0.014734in}}{\pgfqpoint{-0.049702in}{-0.028866in}}{\pgfqpoint{-0.039284in}{-0.039284in}}%
\pgfpathcurveto{\pgfqpoint{-0.028866in}{-0.049702in}}{\pgfqpoint{-0.014734in}{-0.055556in}}{\pgfqpoint{0.000000in}{-0.055556in}}%
\pgfpathlineto{\pgfqpoint{0.000000in}{-0.055556in}}%
\pgfpathclose%
\pgfusepath{stroke,fill}%
}%
\begin{pgfscope}%
\pgfsys@transformshift{6.562539in}{5.421151in}%
\pgfsys@useobject{currentmarker}{}%
\end{pgfscope}%
\end{pgfscope}%
\begin{pgfscope}%
\definecolor{textcolor}{rgb}{0.000000,0.000000,0.000000}%
\pgfsetstrokecolor{textcolor}%
\pgfsetfillcolor{textcolor}%
\pgftext[x=7.062539in,y=5.323929in,left,base]{\color{textcolor}{\sffamily\fontsize{20.000000}{24.000000}\selectfont\catcode`\^=\active\def^{\ifmmode\sp\else\^{}\fi}\catcode`\%=\active\def%{\%}$\hat \xi_n^2$ ($C_n = (\log n)^{1/4}$)}}%
\end{pgfscope}%
\begin{pgfscope}%
\pgfsetrectcap%
\pgfsetroundjoin%
\pgfsetlinewidth{2.007500pt}%
\definecolor{currentstroke}{rgb}{0.866667,0.627451,0.866667}%
\pgfsetstrokecolor{currentstroke}%
\pgfsetdash{}{0pt}%
\pgfpathmoveto{\pgfqpoint{6.284761in}{4.934418in}}%
\pgfpathlineto{\pgfqpoint{6.562539in}{4.934418in}}%
\pgfpathlineto{\pgfqpoint{6.840316in}{4.934418in}}%
\pgfusepath{stroke}%
\end{pgfscope}%
\begin{pgfscope}%
\pgfsetbuttcap%
\pgfsetmiterjoin%
\definecolor{currentfill}{rgb}{0.866667,0.627451,0.866667}%
\pgfsetfillcolor{currentfill}%
\pgfsetlinewidth{1.003750pt}%
\definecolor{currentstroke}{rgb}{0.866667,0.627451,0.866667}%
\pgfsetstrokecolor{currentstroke}%
\pgfsetdash{}{0pt}%
\pgfsys@defobject{currentmarker}{\pgfqpoint{-0.055556in}{-0.055556in}}{\pgfqpoint{0.055556in}{0.055556in}}{%
\pgfpathmoveto{\pgfqpoint{-0.055556in}{-0.055556in}}%
\pgfpathlineto{\pgfqpoint{0.055556in}{-0.055556in}}%
\pgfpathlineto{\pgfqpoint{0.055556in}{0.055556in}}%
\pgfpathlineto{\pgfqpoint{-0.055556in}{0.055556in}}%
\pgfpathlineto{\pgfqpoint{-0.055556in}{-0.055556in}}%
\pgfpathclose%
\pgfusepath{stroke,fill}%
}%
\begin{pgfscope}%
\pgfsys@transformshift{6.562539in}{4.934418in}%
\pgfsys@useobject{currentmarker}{}%
\end{pgfscope}%
\end{pgfscope}%
\begin{pgfscope}%
\definecolor{textcolor}{rgb}{0.000000,0.000000,0.000000}%
\pgfsetstrokecolor{textcolor}%
\pgfsetfillcolor{textcolor}%
\pgftext[x=7.062539in,y=4.837195in,left,base]{\color{textcolor}{\sffamily\fontsize{20.000000}{24.000000}\selectfont\catcode`\^=\active\def^{\ifmmode\sp\else\^{}\fi}\catcode`\%=\active\def%{\%}$\hat \xi_{n, U}^2$ ($C_n = \infty$)}}%
\end{pgfscope}%
\end{pgfpicture}%
\makeatother%
\endgroup%